\documentclass{article}
\pdfoutput=1

\usepackage{microtype}
\usepackage{graphicx}
\usepackage{subfigure}
\usepackage{booktabs} 

\usepackage{hyperref}

\usepackage{float}
\usepackage{mathtools}

\usepackage[accepted]{icml2020} 

\usepackage{algorithm}
\usepackage{algorithmic}
\usepackage{amsmath,amssymb,mdframed,amsthm}
\usepackage{cleveref}
\usepackage{dsfont}

\makeatletter
\newtheorem*{rep@theorem}{\rep@title}
\newcommand{\newreptheorem}[2]{%
\newenvironment{rep#1}[1]{%
 \def\rep@title{#2 \ref{##1}}%
 \begin{rep@theorem}}%
 {\end{rep@theorem}}}
\makeatother

\newreptheorem{theorem}{Theorem}
\newreptheorem{lemma}{Lemma}
\newreptheorem{corollary}{Corollary}

\newtheorem{lemma}{Lemma}
\newtheorem{example}{Example}
\newtheorem{theorem}{Theorem}

\newtheorem{definition}{Definition}

\theoremstyle{definition}
\newtheorem{remark}{Remark}
\newtheorem{assumption}{Assumption \!\!}

\usepackage{makecell}
\usepackage[table]{xcolor}
\DeclareMathOperator{\dist}{dist}
\DeclareMathOperator{\diag}{diag}

\DeclareMathOperator{\dom}{dom}

\DeclareMathOperator{\prox}{prox}

\newcommand{\E}[1]{\mathbb{E}\left[#1\right]}
\newcommand{\Ec}[2]{\mathbb{E}_{#1}\left[{#2}\right]}

\icmltitlerunning{Random extrapolation for primal-dual coordinate descent}

\begin{document}

\twocolumn[
\icmltitle{Random extrapolation for primal-dual coordinate descent}

\begin{icmlauthorlist}
\icmlauthor{Ahmet Alacaoglu}{to}
\icmlauthor{Olivier Fercoq}{goo}
\icmlauthor{Volkan Cevher}{to}
\end{icmlauthorlist}

\icmlaffiliation{to}{EPFL, Switzerland}
\icmlaffiliation{goo}{LTCI, T\'el\'ecom Paris, Institut Polytechnique de Paris, France}
\icmlcorrespondingauthor{Ahmet Alacaoglu}{ahmet.alacaoglu@epfl.ch}

\vskip 0.3in
]

\printAffiliationsAndNotice{} 
\begin{abstract}
We introduce a randomly extrapolated primal-dual coordinate descent method that adapts to sparsity of the data matrix and the favorable structures of the objective function. Our method updates only a subset of primal and dual variables with sparse data, and it uses large step sizes with dense data, retaining the benefits of the specific methods designed for each case. 
In addition to adapting to sparsity, our method attains fast convergence guarantees in favorable cases \textit{without any modifications}. 
In particular, we prove linear convergence under metric subregularity, which applies to strongly convex-strongly concave problems and piecewise linear quadratic functions. 
We show almost sure convergence of the sequence and optimal sublinear convergence rates for the primal-dual gap and objective values, in the general convex-concave case. Numerical evidence demonstrates the state-of-the-art empirical performance of our method in sparse and dense settings, matching and improving the existing methods.
\end{abstract}

\section{Introduction}
In this paper, we consider the problem
\begin{equation}\label{eq: prob_temp}
\min_{x\in\mathcal{X}} f(x) + g(x) + h(Ax),
\end{equation}
where $f, g\colon \mathcal{X} \to \mathbb{R}\cup \{+\infty\}$ and $h\colon \mathcal{Y} \to \mathbb{R}\cup\{+\infty\}$ are proper, lower semicontinuous, convex functions, $A\colon \mathcal{X} \to \mathcal{Y}$ is a linear operator.
$\mathcal{X}$ and $\mathcal{Y}$ are Euclidean spaces such that $\mathcal{X} = \prod_{i=1}^n \mathcal{X}_i$, and $\mathcal{Y}=\prod_{j=1}^m \mathcal{Y}_j$.
Moreover, $f$ is assumed to have coordinatewise Lipschitz continuous gradients and $g, h$ admit easily computable proximal operators.


\begin{table*}[t]
  \centering
  \begin{tabular}{| l | l | c | c | c | c |}
 \hline
      									& \shortstack{Step sizes\\ with dense data} 				& \shortstack{per iteration \\ cost} 			& \shortstack{block-wise \\ Lipschitz}  &\shortstack{probability \\ law} & \shortstack{Efficient \\ implementation} \\\hline
    \cite{chambolle2018stochastic} 			& \cellcolor{lightgray}{$n\tau_i \sigma \| A_i \|^2 < 1$} 		& $m$ 								& N/A 						& \cellcolor{lightgray}{arbitrary}	& \cellcolor{lightgray}{direct$^{\dagger}$} \\ \hline
        \cite{fercoq2019coordinate} 			& \centering{$n^2\tau_i\sigma\|A_i\|^2 < 1$} 						& \cellcolor{lightgray}{$\vert J(i) \vert^{*}$} 		& \cellcolor{lightgray}{Yes} & uniform & \cellcolor{lightgray}{\shortstack{direct or duplication}}  \\ \hline
                \cite{latafat2019new} 				& $n^2\tau_i\sigma\|A_i\|^2 < 1$ 						& \cellcolor{lightgray}{$\vert J(i) \vert^{*}$} 		& No 				& \cellcolor{lightgray}{arbitrary} & duplication \\ \hline
    PURE-CD 							& \cellcolor{lightgray}{$n\tau_i\sigma\|A_i\|^2 < 1$} 			& \cellcolor{lightgray}{$\vert J(i) \vert^{*}$} 		& \cellcolor{lightgray}{Yes} & \cellcolor{lightgray}{arbitrary} & \cellcolor{lightgray}{direct} \\
    \hline
  \end{tabular}
  \caption{\small{Comparison of primal-dual coordinate descent methods. Note that we only compare here the most related methods to ours and include a comprehensive review of other existing methods with comparison to PURE-CD in Section~\ref{sec: lit}. 
  In the last column, we refer to the way one needs to implement the algorithm, for it to be efficient in both sparse and dense settings.
  $^{*}J(i)$ is defined in~\eqref{eq: def_ji}. $^{\dagger}$SPDHG only has implementation for dense setting and not for sparse. The concept of duplication for PDCD is described in \cite{fercoq2019coordinate}.}}
  \label{tab:1}
\end{table*}

Problem~\eqref{eq: prob_temp} is a general template that covers many problems in different fields, such as regularized empirical risk minimization~\cite{shalev2013stochastic,zhang2017stochastic}, optimization with large number of constraints~\cite{patrascu2017nonasymptotic,fercoq2019almost}, and total variation (TV) regularized problems~\cite{chambolle2018stochastic,fercoq2019coordinate}.

The classic choice for solving problem~\eqref{eq: prob_temp} is to use primal-dual methods~\cite{chambolle2011first,vu2013splitting,condat2013primal}.
These methods utilize the proximal operators for $g, h^\ast$ and gradient of the differentiable component $f$.
Randomized versions that we refer to as primal dual coordinate descent (PDCD), are proposed in several works~\cite{zhang2017stochastic,dang2014randomized,gao2019randomized,fercoq2019coordinate,chambolle2018stochastic,latafat2019new}.

First advantage of coordinate-based methods is that they access to blocks of $A$ and update a subset of variables, resulting in cheap per iteration costs.
Moreover, they utilize larger step sizes depending on the properties of the problem in selected blocks.

Existing PDCD methods fail to retain both these advantages, as sparsity of $A$ varies.
In particular, methods that have cheap per-iteration costs with sparse $A$~\cite{fercoq2019coordinate,latafat2019new}, are restricted to use small step sizes with dense $A$.
On the other hand, methods that can use large step sizes with dense $A$~\cite{chambolle2018stochastic}, have high per-iteration costs with sparse $A$.

\textbf{Contributions.} In this paper, we identify random extrapolation as the key to design a method that combines the benefits of the methods in two camps and propose the primal-dual method with random extrapolation and coordinate descent (PURE-CD).
PURE-CD exhibits the advantages of~\cite{fercoq2019coordinate,latafat2019new} in the sparse setting and the advantages of~\cite{chambolle2018stochastic} in the dense setting simultaneously, achieving the best of both worlds.
As PURE-CD has the favorable properties in both ends of the spectrum, it has the best performance in the regime in between: moderately sparse data.

In addition to adapting to the sparsity of $A$, we prove that PURE-CD also adapts to unknown structures in the problem, and obtains linear rate of convergence, without any modifications in the step sizes.
Our linear convergence results apply to strongly convex-strongly concave problems, linear programs, and problems with piecewise linear quadratic functions, involving Lasso, support vector machines and linearly constrained problems with piecewise linear-quadratic objectives.
In the general convex case, we prove that the iterates of PURE-CD converges almost surely to a solution of problem~\eqref{eq: prob_temp}.
Moreover, we show that in this case, the ergodic sequence obtains the optimal $\mathcal{O}(1/k)$ sublinear rate of convergence.

\section{Preliminaries}
\subsection{Notation}
For a positive definite matrix $V$, we denote $\| x \|_V^2 = \langle x, Vx \rangle$.
We define the distance of a point $x$ from a set $\mathcal{X}$ as $\dist(x, \mathcal{X}) = \min_{u\in\mathcal{X}} \| u - x \|^2$.
Given an index $i\in\{1,\ldots, n\}$, the corresponding coordinate of the gradient vector is $\nabla _i f(x)$.
Graph of mapping $F$ is denoted by $\mathrm{gra~} F$.
Recall that $\mathcal{X} = \prod_{i=1}^n \mathcal{X}_i$, and $\mathcal{Y}=\prod_{j=1}^m \mathcal{Y}_j$ and $\mathcal{Z}=\mathcal{X}\times\mathcal{Y}$.
For $u\in\mathcal{X}_i$, $U_i(u)\in\mathcal{X}$ is such that each element of $U_i(u)$ is $0$, except the block $i$ which contains $u$.
We denote the indicator function of a set $\mathcal{X}$ as $\delta_\mathcal{X}$.

Proximal operator with a positive definite $V$ is defined as
\begin{equation}
\prox_{V, g}(x) = \arg\min_{u} g(u) + \frac{1}{2} \| u - x \|^2_{V^{-1}}.\notag
\end{equation}
We will need the following notation for the sparse setting,
\begin{equation}\label{eq: def_ji}
\begin{aligned}
J(i) = \{ j\in\{1, \dots, m\}\colon A_{j, i} \neq 0 \} \\
I(j) = \{ i\in\{1, \dots, n\}\colon A_{j, i} \neq 0 \}.
\end{aligned}
\end{equation}
Given a matrix $A$ and $i \in\{1,\ldots,n\}$, $J(i)$ denotes the row indices that correspond to nonzero values in the column indexed by $i$.
Similarly, with $j\in\{1,\ldots,m\}$, $I(j)$ gives the column indices corresponding to nonzero values in the row indexed by $j$.

Moreover, given positive probabilities $(p_i)_{1 \leq i \leq n}$, we define
\begin{equation}
\pi_j = \sum_{i\in I(j)} p_i.
\end{equation}
In the simple case of $p_i = 1/n$, it is easy to see that $n \pi_j$ corresponds to number of nonzeros in the row indexed by $j$.

At iteration $k$, the algorithm randomly picks an index $i_{k+1}\in\{1,\ldots, n\}$. 
To govern the selection rule, we define the probability matrix $P=\diag(p_1, \dots, p_n)$, where $p_i = \mathbb{P}(i_{k+1} = i)$, and $\underline p = \min_i p_i$. 
We define as $\mathcal{F}_k$ the filtration generated by the random indices $\{i_1, \dots, i_k\}$.

Denoting $z=(x,y)$, we define the functions
\begin{align}
&D_p(x_{k+1}; z) = f(x_{k+1}) + g(x_{k+1}) - f(x) - g(x) \notag\\
&\qquad\qquad\qquad\qquad\qquad\qquad+ \langle A^\top y, x_{k+1} - x \rangle,\notag\\
&D_d(\bar{y}_{k+1}; z) = h^\ast(\bar{y}_{k+1}) - h^\ast(y) - \langle Ax, \bar{y}_{k+1} - y \rangle.\notag
\end{align}

\subsection{Optimality}
Problem~\eqref{eq: prob_temp} has the following saddle point formulation
\begin{align}
\min_{x\in\mathcal{X}} \max_{y\in\mathcal{Y}} f(x) + g(x) + \langle Ax, y \rangle - h^\ast(y).\notag
\end{align}
Karush-Kuhn-Tucker (KKT) conditions state that the vector $z_\star = (x_\star, y_\star)$ is a primal-dual solution of the problem when
\begin{align}\label{eq: kkt}
0 \in \begin{bmatrix} \nabla f(x_\star) + \partial g(x_\star) + A^\top y_\star \\
Ax_\star - \partial h^\ast(y_\star) 
\end{bmatrix} =: F(z_\star).
\end{align}
We call $\mathcal{Z}_\star$ the set of such solutions.

\subsection{Metric subregularity}\label{sec: ms}
We utilize the metric subregularity assumption for proving linear convergence.
This assumption has been used in primal-dual optimization literature for both deterministic~\cite{liang2016convergence} and randomized algorithms~\cite{latafat2019new,alacaoglu2019convergence}.
\begin{definition}\label{eq: def_ms}
A set valued mapping $F\colon \mathcal{X} \rightrightarrows \mathcal{Y}$ is metrically subregular at $\bar{x}$ for $\bar{y}$, with $(\bar{x}, \bar{y}) \in \mathrm{gra~} F$, if there exists $\eta > 0$ with a neighborhood of regularity $\mathcal{N}(\bar{x})$ such that
\begin{equation}
\dist(x, F^{-1} \bar{y}) \leq \eta \dist(\bar{y}, Fx),~~~~\forall x \in \mathcal{N}(\bar{x}).\notag
\end{equation}
\end{definition}
We will be interested in the metric subregularity of KKT operator $F$, given in~\eqref{eq: kkt}, for $0$. Intuitively speaking, as $0 \in F(z_\star),\forall z^\star \in \mathcal{Z}_\star$, metric subregularity of $F$ for $0$ essentially gives us a way to characterize the behavior of the iterates around the solution set.

Even though Definition~\ref{eq: def_ms} looks daunting, fortunately, one does not need to check it for a given problem.
Metric subregularity is well-studied in the literature and it is known to be satisfied in the following cases:
\begin{example}\label{ex: ms_examples}~\\
$\triangleright$ If $f+g$ and $h^\ast$ are strongly convex, Definition~\ref{eq: def_ms} holds with $\mathcal{N}(\bar{x}) = \mathbb{R}^d$~\citep[Lemma IV.2]{latafat2019new}. \\
$\triangleright$ If $f, g, h$ are piecewise linear quadratic (PLQ) functions, Definition~\ref{eq: def_ms} holds with any bounded neighborhood $\mathcal{N}(\bar{x})$~\citep[Lemma IV.4]{latafat2019new}.
\end{example}
PLQ functions include $\ell_1$ norm, hinge loss, indicator of polyhedral sets. Therefore third bullet point apply to Lasso, support vector machines, elastic net, and linearly constrained problems with PLQ loss functions~\cite{latafat2019new}.

We now state our main assumptions which are standard in the literature~\cite{fercoq2019coordinate,chambolle2018stochastic,latafat2019new,bauschke2011convex}:
\begin{assumption}\label{asmp: asmp1}~\\
$\triangleright$ $f$, $g$ and $h$ are proper, lower semicontinuous, convex. \\[1mm]
$\triangleright$ $g$ is separable, i.e., $g(x) = \sum_{i=1}^n g_i(x_i)$, and $f$ has coordinatewise Lipschitz gradients such that $\forall x \in \mathcal{X}, \forall u \in \mathcal{X}_i$,
\begin{equation}
f(x + U_i(u)) \leq f(x) + \langle \nabla_i f(x), u \rangle + \frac{\beta_i}{2} \| u \|^2.
\end{equation}
$\triangleright$ Set of solutions to problem~\eqref{eq: prob_temp}, defined in~\eqref{eq: kkt} is nonempty. \\[1mm]
$\triangleright$ Slater's condition holds, which states that $0 \in \mathrm{ri}(\dom h - A\dom g)$ where $\mathrm{ri}$ denotes the relative interior.
\end{assumption}

\section{Algorithm}\label{sec: alg}
In this section, we will sketch the main ideas behind our algorithm.
Primal-dual method\footnote{This method is also known as V\~u-Condat algorithm.}, due to~\cite{chambolle2011first,condat2013primal,vu2013splitting} reads as
\begin{equation}\label{eq: vc_alg}
\begin{aligned}
&\bar{x}_{k+1} = \prox_{\tau, g} \left(\bar{x}_k - \tau\left( \nabla f(\bar x_k) + A^\top \bar{y}_k \right) \right)  \\
&\bar{y}_{k+1} = \prox_{\sigma, h^\ast} \left( \bar{y}_k + \sigma A (2\bar{x}_{k+1} - \bar{x}_k)\right).
\end{aligned}
\end{equation}
The main intuition behind PDCD methods proposed by~\cite{zhang2017stochastic,fercoq2019coordinate,chambolle2018stochastic} is to incorporate coordinate based updates.
Among these methods,~\cite{zhang2017stochastic} specializes in strongly convex-strongly concave problems, whereas the other other ones apply to more general classes of problems.

A closely related approach focused on the following interpretation of primal-dual method~\eqref{eq: vc_alg} which is named as TriPD in~\citep[Algorithm 1]{latafat2019new}
\begin{equation}\label{eq: tripd}
\begin{aligned}
&\bar{y}_{k+1} = \prox_{\sigma, h^\ast}\left( \hat{y}_k + \sigma A \bar{x}_k \right)  \\
&\bar{x}_{k+1} = \prox_{\tau, g}\left(\bar{x}_k - \tau\left( \nabla f(\bar{x}_k) + A^\top \bar{y}_{k+1} \right) \right)  \\
&\hat{y}_{k+1} = \bar{y}_{k+1} + \sigma A(\bar{x}_{k+1} - \bar{x}_k). 
\end{aligned}
\end{equation}
We notice that by moving the $\bar{y}_{k+1}$ update in TriPD to take place after $\hat{y}_{k+1}$ update, one obtains~\eqref{eq: vc_alg}.

As observed in~\cite{latafat2019new}, this particular interpretation of primal-dual method is useful for randomization.
TriPD-BC as proposed in~\cite{latafat2019new} iterates as
\begin{align}
&\bar{y}_{k+1} = \prox_{\sigma, h^\ast}\left( {y}_k + \sigma A {x}_k \right) \notag \\
&\bar{x}_{k+1} = \prox_{\tau, g}\left({x}_k - \tau\left( \nabla f(x_k) + A^\top \bar{y}_{k+1} \right) \right) \notag \\
&\hat{y}_{k+1} = \bar{y}_{k+1} + \sigma A(\bar{x}_{k+1} - x_k) \notag \\
&\text{Draw an index } i_{k+1}\in\{1,\dots,n\} \text{ randomly.} \notag \\
&x_{k+1}^{i_{k+1}} = \bar{x}_{k+1}^{i_{k+1}},~~~x_{k+1}^j = x_k^j, \forall j\neq i_{k+1}\notag\\
&y_{k+1}^j = \hat{y}_{k+1}^j, \forall j \in J(i_{k+1}),~~~y_{k+1}^j = y_k^j, \forall j\not\in J(i_{k+1}). \notag
\end{align}
One immediate limitation of TriPD-BC is that to update $y_{k+1}$, one needs to know $\bar{x}_{k+1}$, whereas only $\bar{x}_{k+1}^{i_{k+1}}$ is needed to update $x_{k+1}$.
As also discussed in~\cite{latafat2019new}, this scheme is suitable when $A$ has special structure such as sparsity.
When $A$ is dense, one needs to update all elements of $y_{k+1}$ and $\hat{y}_{k+1}$, in which case one needs to compute both $\bar{y}_{k+1}$ and $\bar{x}_{k+1}$ which has the same cost as a deterministic algorithm.

In the dense setting, for an efficient implementation, one can use duplication of dual variables as described in~\cite{fercoq2019coordinate}. 
However, in this case one is restricted to use small step sizes as discussed in~\cite{fercoq2019coordinate}. Compared to SPDHG in~\cite{chambolle2018stochastic}, the step sizes can be $n$ times worse, deteriorating the performance of the method considerably in the dense setting.

On the other hand, the drawback of~SPDHG is that it needs to update all dual variables at every iteration, whereas the methods in~\cite{fercoq2019coordinate,latafat2019new} update only a subset of dual variables depending on the sparsity of $A$. When the dual dimension is high, per iteration cost of~\cite{chambolle2018stochastic} becomes prohibitive.

Our idea, inspired by~\cite{chambolle2018stochastic}, to make TriPD-BC efficient for dense setting is to use $x_{k+1}$ rather than $\bar{x}_{k+1}$ in the update of $\hat{y}_{k+1}$.
Although simple to state, this modification makes $\hat{y}_{k+1}$ random, rendering the analysis of~\cite{latafat2019new} and other analyses based on monotone operator theory not applicable.

This leads to our algorithm, primal-dual method with random extrapolation and coordinate descent (PURE-CD). Our method uses large step sizes as in~\cite{chambolle2018stochastic} in the dense setting, while staying efficient in terms of per iteration costs in the sparse setting as in~\cite{fercoq2019coordinate,latafat2019new}; leading to the first general PDCD algorithm that obtains favorable properties in both sparse and dense settings.

\begin{algorithm}[h]
\caption{Primal-dual method with random extrapolation and coordinate descent (PURE-CD)}
\label{alg:stripd}
\begin{algorithmic}[1]
    \STATE { \textbf{Input:}} Diagonal matrices $\theta, \tau, \sigma > 0$, chosen according to~\eqref{eq: theta_choice},~\eqref{eq: ss_choice}.
    \FOR{$k = 0,1\ldots $} 
	\STATE $\bar{y}_{k+1} = \prox_{\sigma, h^\ast}\left(y_k + \sigma Ax_k\right)$
         \STATE $\bar{x}_{k+1} = \prox_{\tau, g}\left(x_k - \tau \left(\nabla f(x_k) + A^\top \bar{y}_{k+1}\right)\right)$
        \STATE Draw $i_{k+1} \in \{ 1, \dots, n \}$ with $\mathbb{P}(i_{k+1} = i) = p_i$
         \STATE $x_{k+1}^{i_{k+1}} = \bar{x}_{k+1}^{i_{k+1}}$ 
         \STATE $x_{k+1}^{j} = x_k^j, \forall j \neq i_{k+1}$
         \STATE $y_{k+1}^{j} = \bar{y}_{k+1}^j + \sigma_j \theta_j(A (x_{k+1} - x_k))_j, \forall j \in J(i_{k+1})$ $y_{k+1}^j = y_k^j, \forall j\not\in J(i_{k+1})$
	\ENDFOR
\end{algorithmic}	
\end{algorithm}

\section{Convergence Analysis}
In this section, we  analyze the behavior of Algorithm~\ref{alg:stripd} under various assumptions.
We first start with a lemma analyzing one iteration of the algorithm.
\vspace{0.05cm}
\begin{lemma}\label{lem: lem1}
Under Assumption~\ref{asmp: asmp1}, let $\theta = \diag(\theta_1, \ldots, \theta_m)$ and $\pi = \diag(\pi_1, \ldots, \pi_m)$ be chosen as
\begin{equation}
\theta_j = \frac{\pi_j}{\underline p}, \text{ where } \pi_j = \sum_{i \in I(j)} p_i, \text{ and } \underline p = \min_i p_i.\label{eq: theta_choice}
\end{equation}
We define the functions, given $z=(x, y)$,
\begin{align}
&V(z) = \frac{\underline p}{2} \| x \|^2_{\tau^{-1}P^{-1}} + \frac{\underline p}{2} \| y \|^2_{\sigma^{-1}\pi^{-1}},\notag\\
&\tilde{V}(z) = \frac{\underline p}{2} \| x \|^2_{C(\tau)} + \frac{\underline p}{2} \| y \|^2_{\sigma^{-1}},\notag
\end{align}
where 
$
C(\tau)_i = \frac{2p_i}{\underline p \tau_i} - \frac{1}{\tau_i} - p_i \sum_{j=1}^m \pi_j^{-1}\sigma_j \theta_j^2 A_{j, i}^2 - \frac{\beta_i p_i}{\underline p}$.

Then, for the iterates of Algorithm~\ref{alg:stripd}, $\forall z=(x, y) \in \mathcal{Z}$, it holds that:
\begin{multline}
\mathbb{E}_k \left[ D_p(x_{k+1}; z) \right] +\underline p D_d(\bar{y}_{k+1}; z) + \mathbb{E}_k \left[ V(z_{k+1} - z) \right] \\ \leq (1-\underline p)D_p (x_k; z) + V(z_k - z) - \tilde{V}(\bar{z}_{k+1} - z_k).\notag
\end{multline}
\end{lemma}
The main technical challenge in the proof of the lemma, compared to the corresponding results in~\cite{latafat2019new} and~\cite{chambolle2018stochastic} is handling stochasticity in both variables $x_{k+1}, y_{k+1}$ (and also $\hat{y}_{k+1}$ for~\cite{latafat2019new}).
Using coordinatewise Lipschitz constants of $f$ with arbitrary sampling also requires an intricate analysis.

The result of Lemma~\ref{lem: lem1} is promising for deriving convergence results for Algorithm~\ref{alg:stripd}. 
As $D_p(x_{k+1}; z_\star) \geq 0$, $D_d(\bar{y}_{k+1}; z_\star)\geq0$ and when step sizes are chosen such that $\tilde{V}$ is nonnegative, Lemma~\ref{lem: lem1} describes a stochastic monotonicity property.
In particular, it shows that $D_p(x_{k+1}; z_\star) + V(z_{k+1} - z_\star)$ which measures the distance to solution in a Bregman distance sense, is monotonically nonincreasing in expectation.

\subsection{Almost sure convergence}\label{sec: as}
Almost sure convergence is a fundamental property for randomized methods describing the limiting behavior of the iterates in different realization of the algorithm.
The following theorem states that the iterates of Algorithm~\ref{alg:stripd} converge almost surely to a point in the solution set.
\begin{theorem}\label{thm: thm1}
Let Assumption~\ref{asmp: asmp1} hold, $\theta$ and $\pi$ be as in~\Cref{lem: lem1}, and the step sizes $\tau, \sigma$ satisfy
\begin{equation}
\tau_i < \frac{2p_i - \underline{p}}{\beta_i p_i + \underline p^{-1}p_i\sum_{j=1}^m \pi_j \sigma_j A_{j, i}^2}.\label{eq: ss_choice}
\end{equation}
The iterates $z_k$ are produced by~\Cref{alg:stripd}. 
Then, almost surely, there exist $z_\star = (x_\star, y_\star) \in \mathcal{Z}_\star$ such that $z_k \to z_\star$.
\end{theorem}

We analyze the step size rule~\eqref{eq: ss_choice} in Theorem~\ref{thm: thm1} and compare with existing efficient methods in dense and sparse settings.
\begin{remark}\label{rm: rm1}
~\\
$\triangleright$ Let $A$ be dense, with all its elements being nonzero, $p_i = 1/n$ and $f(\cdot)=0$, then the step size rule reduces to
\begin{align}
\tau_i < \frac{1}{n \sigma \| A_i \|^2},\notag
\end{align}
which is the step size rule of SPDHG~\cite{chambolle2018stochastic,alacaoglu2019convergence}, which is shown to be favorable in the dense setting. In contrast, step size rules of~\cite{fercoq2019coordinate,latafat2019new} are $n$ times worse due to duplication, in this case.\\[2mm]
$\triangleright$ Let $A$ be such that it contains one nonzero element per row, and we use $p_i = \frac{1}{n}$, which results in $\pi_j = \frac{1}{n}$. Then,
\begin{align}
\tau_i < \frac{1}{\beta_i + \sum_{j=1}^m \sigma_j A_{j, i}^2},\notag
\end{align}
which is the step size rule of Vu-Condat-CD~\cite{fercoq2019coordinate}, upon using the definition of $J(i)$ from~\eqref{eq: def_ji}.
Similarly, Algorithm~\ref{alg:stripd} updates $1$ dual coordinate and $1$ primal coordinate, in this case. In contrast, SPDHG~\cite{chambolle2018stochastic} updates $m$ dual coordinates, resulting in $m$ times higher per iteration cost.
\end{remark}

We note that the step size of TriPD-BC~\cite{latafat2019new} depends on global Lipschitz constant of $f$ rather than coordinatewise Lipschitz constants.
In fact, using coordinatewise Lipschitz constants is very important one of the most important reasons of the success of coordinate descent methods, as it results in larger step sizes~\cite{nesterov2012efficiency,richtarik2014iteration,fercoq2015accelerated}.

The takeaway from Remark~\ref{rm: rm1} is that Algorithm~\ref{alg:stripd} recovers the characteristics of the best performing methods in fully dense and fully sparse settings.
Moreover, as it is the only method that has the desirable dependencies in both cases, it has the best properties in the moderate sparse cases.
We also validate this observation in numerical experiments. 

\subsection{Linear convergence}
Linear convergence of primal-dual methods in practice is a widely observed phenomenon~\cite{chambolle2011first,liang2016convergence}.
We show that Algorithm~\ref{alg:stripd} also shares this property and obtains linear convergence under metric subregularity, without any modification on the algorithm.

We define the Bregman-type projection onto the solution set 
\begin{equation}\label{eq: def_zkstar}
z_k^\star = \arg\min_{u\in\mathcal{Z}_\star} D_p(x_k; u) + V(z_k - u).
\end{equation}
We now show that $z_k^\star$ is well-defined under our assumptions.
First, the solution set is convex and closed. 
Second, $D_p(x_{k}; u) \geq 0$ for all $u\in\mathcal{Z}_\star$ and it is also lower semicontinuous. 
Third, we remark that $V(z_k-u)$ is a squared norm (see Lemma~\ref{lem: lem1}), thus coercive, therefore the sum is coercive and lower semicontinuous over $\mathcal{Z}_\star$.
Hence, $z_k^\star$ exists.

The definition of $z_k^\star$ in~\eqref{eq: def_zkstar} is more involved compared to the corresponding quantity in~\cite{latafat2019new}.
This is in fact due to us using coordinatewise Lipschitz constants in our step sizes, rather than the global Lipschitz constant in~\cite{latafat2019new}.
\begin{assumption}\label{asmp: asmp2}~\\
KKT operator $F$ is metrically subregular at all $z_\star \in \mathcal{Z}_\star$ for $0$, and $\bar{z}_k\in\mathcal{N}(z_\star), \forall z_\star, \forall k$.
\end{assumption}

\begin{theorem}\label{thm: thm2}
Let~\Cref{asmp: asmp1} and~\ref{asmp: asmp2} hold.
Let $\theta$ and the step sizes $\tau, \sigma$ be chosen according to~\eqref{eq: theta_choice} and~\eqref{eq: ss_choice}.
Moreover, $z_k^\star = (x_k^\star, y_k^\star)$ is as defined in~\eqref{eq: def_zkstar}.
Then, for $z_k$ generated by Algorithm~\ref{alg:stripd}, it follows that
\begin{multline}
\mathbb{E} \left[\frac{\underline p}{2} \| x_k - x_k^\star \|^2_{\tau^{-1} P^{-1}} + \frac{\underline p}{2} \| y_k - y_k^\star \|^2_{\sigma^{-1}\pi^{-1}} \right] \\ 
\leq (1-\rho)^k \Delta_0,\notag
\end{multline}
where $\rho = \min\left( \underline p, \frac{C_{2, \tilde V}}{C_{V, 2} ((2+2c)+(1+c)(\eta \| H-M\|+\beta))^2} \right)$, $\Delta_0 = D_p(x_0; z_0^\star) + V(z_0 - z_0^\star)$, $\bar \beta$ is the global Lipschitz constant of $f$, \\$C_{2, \tilde V} = \frac{\underline{p}}{2} \min \left\{ \min_i C(\tau)_i, \min_j \sigma_j^{-1} \right\}$, \\
$C_{V, 2}=\frac{1}{2} \max\left\{ \max_i \frac{1}{\tau_i}, \max_j \frac{1}{\sigma_j} \right\}$, \\
$C_{2,V} = \sqrt{\frac{2}{\underline p\min_i\{\tau_i^{-1}p_i^{-1}\}}}$, $c=C_{2,V}\sqrt{\| A\|/2}$, and
\begin{equation}
H = \begin{bmatrix} \tau^{-1} & A^\top \\ 0 & \sigma^{-1} \end{bmatrix}, ~~~~ M = \begin{bmatrix} 0 & A^\top \\ -A & 0 \end{bmatrix}.\notag
\end{equation}
\end{theorem}

The first remark about Theorem~\ref{thm: thm2} is that since metric subregularity constant $\eta$ is not required in the algorithm, the step sizes to achieve linear convergence are the same step sizes as~\eqref{eq: ss_choice}.
Therefore, PURE-CD adapts to structures on the problem, without any need to modify the algorithm, and attains linear rate of convergence.
This supports the well-known observation that primal-dual algorithms converge linearly on most problems, with standard step sizes in~\eqref{eq: ss_choice}.

In particular, a direct corollary of our theorem is that for problems listed in Example~\ref{ex: ms_examples}, PURE-CD obtains linear rate of convergence.
For the first case in Example~\ref{ex: ms_examples}, our result applies directly since the neighborhood of subregularity $\mathcal{N}(z_\star)$ is the whole space.
For the second case, we have to assume additionally that $\bar{z}_k$ is contained in a compact set, since the $\mathcal{N}(z_\star)$ is not the whole space, and is bounded.
A sufficient assumption for this is when the domains of $g$ and $h^\ast$ are compact.
We note that compactness is only required for this result in our paper. 
This is common to other results for PDCD methods with metric subregularity~\cite{latafat2019new,alacaoglu2019convergence}.
The issue, as explained in~\cite{alacaoglu2019convergence}, stems from a fundamental limitation of the existing analyses of PDCD methods.

Many results in the literature for linear convergence only applies to the first case in Example~\ref{ex: ms_examples}, when $g, h^\ast$ are strongly convex~\cite{zhang2017stochastic,chambolle2018stochastic}.
Moreover, these results require setting step sizes depending on strong convexity constants of $g, h^\ast$, therefore not applicable when strong convexity is absent.
Our result applies to more general problems and it uses step sizes independent of these constants.
Our algorithm can be directly applied to any problem satisfying Assumption~\ref{asmp: asmp1} and fast convergence will occur provably, if the selected problem is in Example~\ref{ex: ms_examples}.

Compared with the linear convergence rate in~\cite{latafat2019new} for TriPD-BC, our result have a similar contraction factor, however, due to larger step sizes (see~\Cref{rm: rm1}), the rate comes with a better constant.

\subsection{Ergodic rates}\label{sec: ergodic}
In this section, we study Algorithm~\ref{alg:stripd} in the general case, under Assumption~\ref{asmp: asmp1}, and show the optimal $\mathcal{O}(1/k)$ convergence rate on the ergodic sequence.
The quantity of interest is the primal-dual gap function~\cite{chambolle2011first}
\begin{align}
G(\bar{x}, \bar{y}) &= \sup_{z=(x, y)\in\mathcal{Z}} f(\bar{x})+g(\bar{x}) + \langle A\bar{x}, y \rangle - h^\ast(y) \notag \\
&-f(x) - g(x) - \langle Ax, \bar{y} \rangle + h^\ast(\bar{y}).\label{eq: def_pdgap}
\end{align}
A related quantity is the restricted gap function (see~\cite{chambolle2011first}) for any set $\mathcal{C}\subset\mathcal{Z}$
\begin{align}
G_\mathcal{C}(\bar{x}, \bar{y}) &= \sup_{z\in\mathcal{C}} f(\bar{x})+g(\bar{x}) + \langle A\bar{x}, y \rangle - h^\ast(y) \notag \\
&-f(x) - g(x) - \langle Ax, \bar{y} \rangle + h^\ast(\bar{y}).
\end{align}

Due to randomization in PDCD, we are interested in the expected primal-dual gap, denoted as $\mathbb{E}\left[ G_{\mathcal{C}}(\bar{x}, \bar{y}) \right]$.
As noted by~\citet{dang2014randomized}, it is technically challenging to prove rates for this quantity as it is the expectation of supremum.
Recently,~\cite{alacaoglu2019convergence} used a technique to show convergence of expected primal-dual gap for SPDHG of~\cite{chambolle2018stochastic}.
This rate is for ergodic sequence averaging $x_k$ and the full dual variable $\bar{y}_k$.
We can use this technique for our analysis. However, there remains another technical challenge as full dual variable is not computed in PURE-CD.
Thus, averaging $\bar{y}_k$ is not feasible in our case.

In addition to~\Cref{asmp: asmp1}, in this section we will assume separability of $h$, to be able to do an efficient averaging with the dual iterate.

Due to the asymmetric nature of Algorithm~\ref{alg:stripd}, there are fundamental difficulties for proving a rate with averaging $y_{k+1}$.
On this front, we propose a new type of analysis for the dual variable.
To start with, we define the following iterate which has the same cost to compute as $y_{k+1}$ each iteration.
Let $\breve{y}_1 = y_1 = \bar{y}_1$,
\begin{equation}
\begin{aligned}\label{eq: breve_def}
\breve{y}_{k+1}^j &= \bar{y}_{k+1}^j, &&\forall j \in J(i_{k+1}),\\
\breve{y}_{k+1}^j &= \breve{y}_k^j, &&\forall j\not\in J(i_{k+1}).
\end{aligned}
\end{equation}
We note that $\breve{y}_k$ is $\mathcal{F}_k$-measurable and more useful properties of $\breve{y}_k$ for analysis are discussed in~\Cref{lem: breve} in the appendix.

Due to the definition of $\breve{y}_k$, it is now feasible to compute and average this iterate.
We can show the convergence of expected primal-dual gap by averaging $\breve{y}_k$ and $x_k$.
We remark that we use some coarse inequalities to give simple constants for~\Cref{th:erg} and~\Cref{cor: erg} in this section, which results in suboptimal dependence with respect to dimension $n$.
In~\Cref{sec: proofs}, we give these theorems with their original, tighter bounds and we show how we transform the tighter bounds into the constants we give in this section.
\begin{theorem}\label{th:erg}
Let Assumption~\ref{asmp: asmp1} hold, $\theta, \tau, \sigma$ are chosen as in~\eqref{eq: theta_choice},~\eqref{eq: ss_choice}, and $h$ be separable.
Let $x^{av}_K = \frac{1}{K} \sum_{k=1}^K x_{k}$ and $y^{av}_K = \frac{1}{K} \sum_{k=1}^K \breve{y}_{k}$, where $\breve{y}_k$ is defined in~\eqref{eq: breve_def}. Then it holds that for any bounded set $\mathcal{C}=\mathcal{C}_x\times\mathcal{C}_y\subset\mathcal{Z}$
\begin{align}
\mathbb{E} \left[G_\mathcal{C}(x^{av}_K, y^{av}_K) \right] \leq \frac{C_{g}}{\underline pK},\notag
\end{align}
where $C_g = \sum_{i=1}^4 C_{g, i}$, $C_{\tau, \tilde V} =  \min_i C(\tau)_i\tau_i$, \\
$C_{g, 1} = \sup_{z\in \mathcal{C}} \big\{2\underline p \| x_0 - x\|^2_{\tau^{-1}P^{-1}} + 2\underline p\| y_0 - y \|^2_{\sigma^{-1}\pi^{-1}} \big\} 
+4\sqrt{\Delta_0\underline p^{-1}}\|A\| \sup_{y\in \mathcal{C}_y}\|y\|_{\tau P} \\+ \sqrt{\Delta_0(\underline p^{-1} - 2\underline p^{-2}C_{\tau,\tilde V}^{-1})}\|A\|\sup_{x\in\mathcal{C}_x}\|x\|_{\sigma\pi}$,\\
$\sum_{i=2}^4 C_{g, i} = \Delta_0\left( 5+9\underline p^{-1} + C_{\tau,\tilde V}^{-1}\left(1+10\underline p^{-1} + 8\underline p^{-2}\right) \right)  \\
+(1-\underline p)(f(x_0) + g(x_0) - f(x_\star) - g(x_\star)) + h^\ast(y_0) - h^\ast(y_\star)  \\
+\underline p\| Ax_\star \|_{\sigma\pi^{-1}} + \|A^\top y_\star \|^2_{\tau P}$.
\end{theorem}
\begin{remark}
When implementing averaging of $x_k$, and $\breve{y}_k$, one should use a technique similar to~\cite{dang2015stochastic}.
The main idea is to only update the averaged vector at the coordinates where an update occurred.
For this, one needs to remember for each coordinate, the last time it is updated, and update the averaged vector using this information, when a coordinate is selected.
\end{remark}
The result in~\Cref{th:erg} would give a rate for primal-dual gap when $\mathcal{C}=\mathcal{Z}$. 
However, in general such a rate is not desirable as taking a supremum over $\mathcal{Z}$ might result in an infinite bound.
This rate would be meaningful when both primal and dual domains are bounded in which case one would take the supremum in $C_{g,1}$ over the bounded domains. 

Alternatively, in the following theorem, we show that for two important special cases, we can extend this result to show guarantees without bounded domains.
Namely, we show the same rate for the case when $h(\cdot) = \delta_{\{b\}}(\cdot)$, $b\in\mathbb{R}^m$ to cover linearly constrained problems.
Moreover, we show the result for the case when $h$ is Lipschitz continuous.

\begin{theorem}\label{cor: erg} 
Let Assumption~\ref{asmp: asmp1} hold. We use the same parameters $\theta,\tau,\sigma$ and the definitions for $x_K^{av}$ and $y_K^{av}$ as Theorem~\ref{th:erg}. 
We consider two cases separately: \\
$\triangleright$ If $h(\cdot) = \delta_{\{ b \}}(\cdot)$, we obtain
\begin{align}
&\mathbb{E} [ f(x_K^{av}) + g(x_K^{av}) - f(x_\star) - g(x_\star)] \leq \frac{C_o}{\underline pK}. \notag\\
&\mathbb{E}[ \|Ax_K^{av}-b \|] \leq \frac{C_f}{\underline pK}.\notag
\end{align}
$\triangleright$ If $h$ is $L_h$-Lipschitz continuous, we obtain
\begin{multline}
\mathbb{E}[ f(x_K^{av}) + g(x_K^{av}) + h(Ax_K^{av}) \\
- f(x_\star) - g(x_\star) -h(Ax_\star) ] \leq \frac{C_l}{\underline pK},\notag
\end{multline}
where 
$C_f = 2c_1\|y_\star - y_0 \|_{\sigma^{-1}\pi^{-1}} + 2\sqrt{c_1C_s} + 2\sqrt{2}c_1$, \\
$C_o = C_s + \| y_\star \|_{\sigma^{-1}\pi^{-1}} C_f + 2c_1\underline p^{-1}V(z_0-z_\star)$,\\
$C_l = C_s + c_1 \| x_\star-x_0\|^2_{\tau^{-1}P^{-1}} + 4 c_1 L_h^2$,\\
$C_s =  C_{g, 2}+ C_{g, 5}+C_{g,6}$, with $c_1 = 2\underline p + 2$, $C_{g, 2}$ as defined in Theorem~\ref{th:erg} and $C_{g,5}, C_{g, 6}$ are defined in the proof in~\eqref{eq: cg3_def},~\eqref{eq: cg4_def}.
\end{theorem}

\section{Related works}\label{sec: lit}
In the deterministic setting, many primal-dual methods are proposed~\cite{chambolle2011first,vu2013splitting,condat2013primal,tan2020accelerated,latafat2019new}.
The standard results in these papers include linear convergence when $g, h^\ast$ are strongly convex, with step sizes selected by using strong convexity constants.
In addition, these papers also show sublinear $\mathcal{O}(1/k)$ rate with general convexity, which is known to be optimal~\cite{nesterov2005smooth}.
Moreover, linear rates under metric subregularity is shown for deterministic methods in~\cite{liang2016convergence,latafat2019new}.

Randomized coordinate descent is proposed in~\cite{nesterov2012efficiency} and improved by a large body of subsequent papers~\cite{richtarik2014iteration,fercoq2015accelerated}.
Primal randomized coordinate descent requires full separability on the nonsmooth parts of the objective function.
Nonsmooth and nonseparable functions are handled by primal-dual coordinate descent methods~\cite{fercoq2019coordinate}.

One of the first primal-dual coordinate descent (PDCD) methods is SPDC, which is proposed in~\cite{zhang2017stochastic}, that solves a special case of problem~\eqref{eq: prob_temp} with $f=0$.
SPDC has linear convergence when $g, h^\ast$ are strongly convex and the step sizes are selected according to strong convexity constants.
In the general convex case, SPDC has perturbation-based analysis, which needs to set an $\epsilon$, requires knowing $\| x_\star\|^2$, and shows $\epsilon$-based iteration complexity results, and not anytime convergence rates.
Almost sure convergence of the iterates of SPDC is not proven in the general convex case.
Moreover, the step sizes of SPDC are scalar and they depend on the maximum block norm of $A$.
It is shown in~\cite{zhang2017stochastic} that in the specific cases when $g(x)=\| x \|^2$ or $g(x)=\| x\|_1 + \| x \|^2$, one can use a special implementation for efficiency with sparse data.

\citet{tan2020accelerated} proposed a new method similar to SPDC with the same type of guarantees as~\cite{zhang2017stochastic}.
Due to similar analysis techniques, this method inherits the abovementioned drawbacks of SPDC.
For this method,~\citet{tan2020accelerated} showed a new implementation technique for sparse data, that can be used with any separable $g(x)$.

For solving the specific case of empirical risk minimization problems, stochastic dual coordinate ascent (SDCA) is proposed in~\cite{shalev2013stochastic,shalev2014accelerated}.
SDCA uses strong convexity constant to set step sizes and attain linear convergence.
A limitation of SDCA is to require strong convexity in the primal, to ensure smoothness of the dual objective, which is essential in the design of the method.

Another early PDCD method is by~\cite{dang2014randomized} where the authors focused on showing sublinear convergence rates.
The authors showed guarantees for a relaxed version of expected primal-dual gap function in~\eqref{eq: def_pdgap}.

Building on~\cite{dang2014randomized}, block-coordinate variants of alternating direction method of multipliers (ADMM) are proposed in~\cite{gao2019randomized,xu2018accelerated}.
These papers focus on linearly constrained problems and show ergodic sublinear convergence rates.
Moreover,~\cite{xu2018accelerated} showed that under strong convexity assumption and special decomposition of the blocks, the method achieves linear convergence.
This linear convergence result, similar to~\cite{zhang2017stochastic} requires knowing the strong convexity constants to set the algorithmic parameters.
Moreover, these results generally set step sizes depending on global Lipschitz constants and norm of $A$ rather than the norm of the blocks of $A$.

\begin{figure*}[ht]
\begin{center}
\includegraphics[width=0.6\columnwidth]{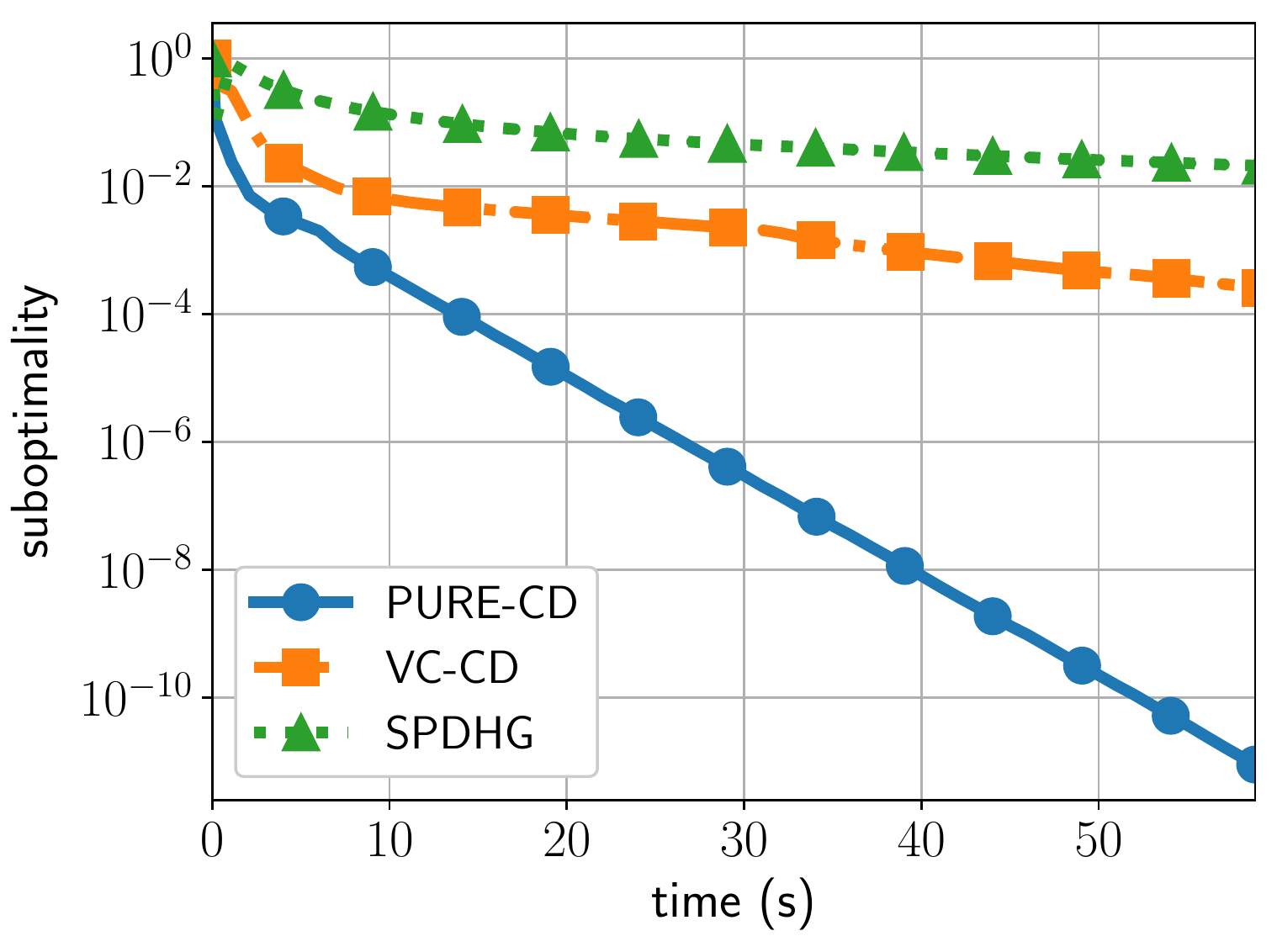}
\includegraphics[width=0.6\columnwidth]{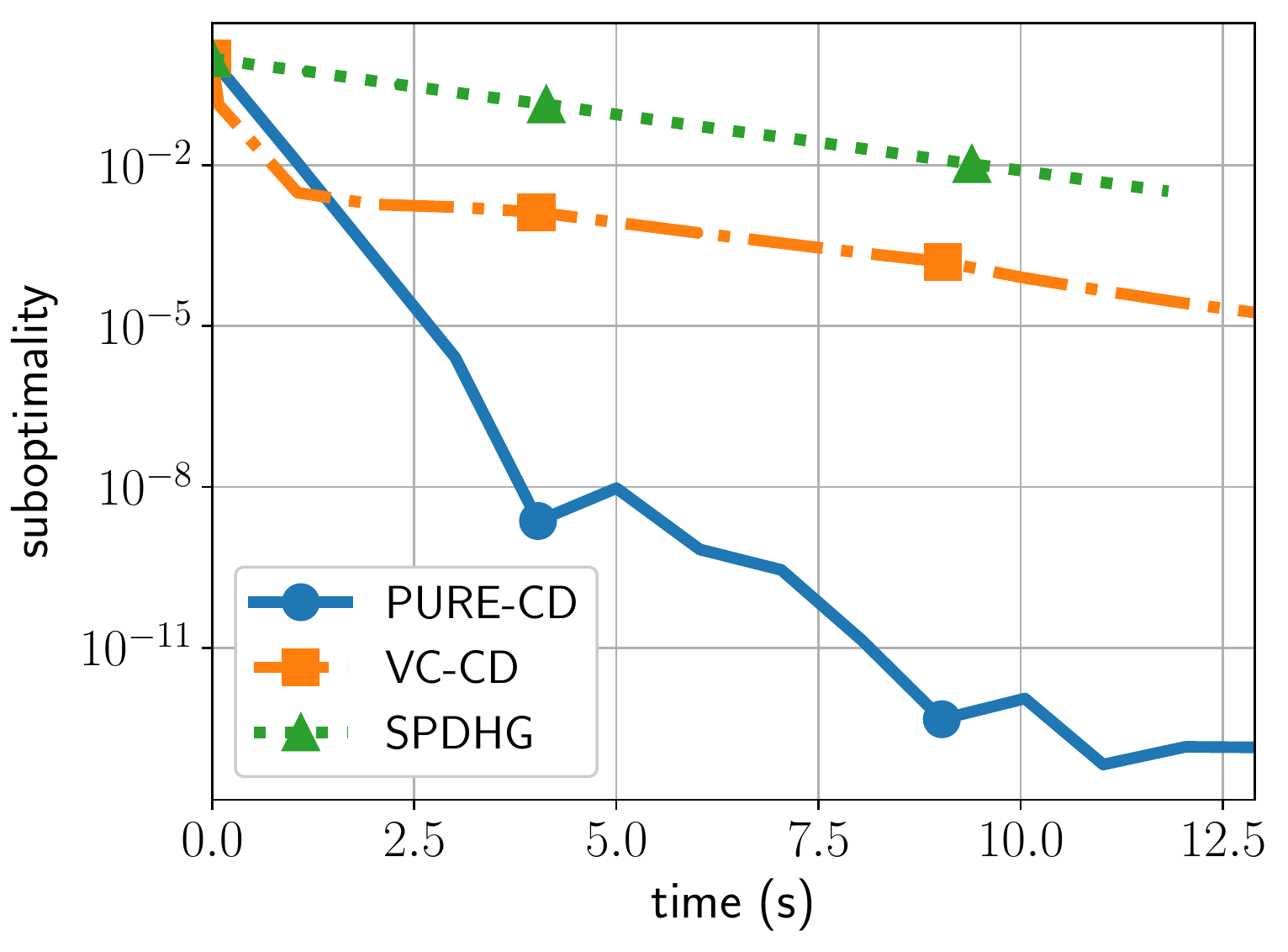}
\includegraphics[width=0.6\columnwidth]{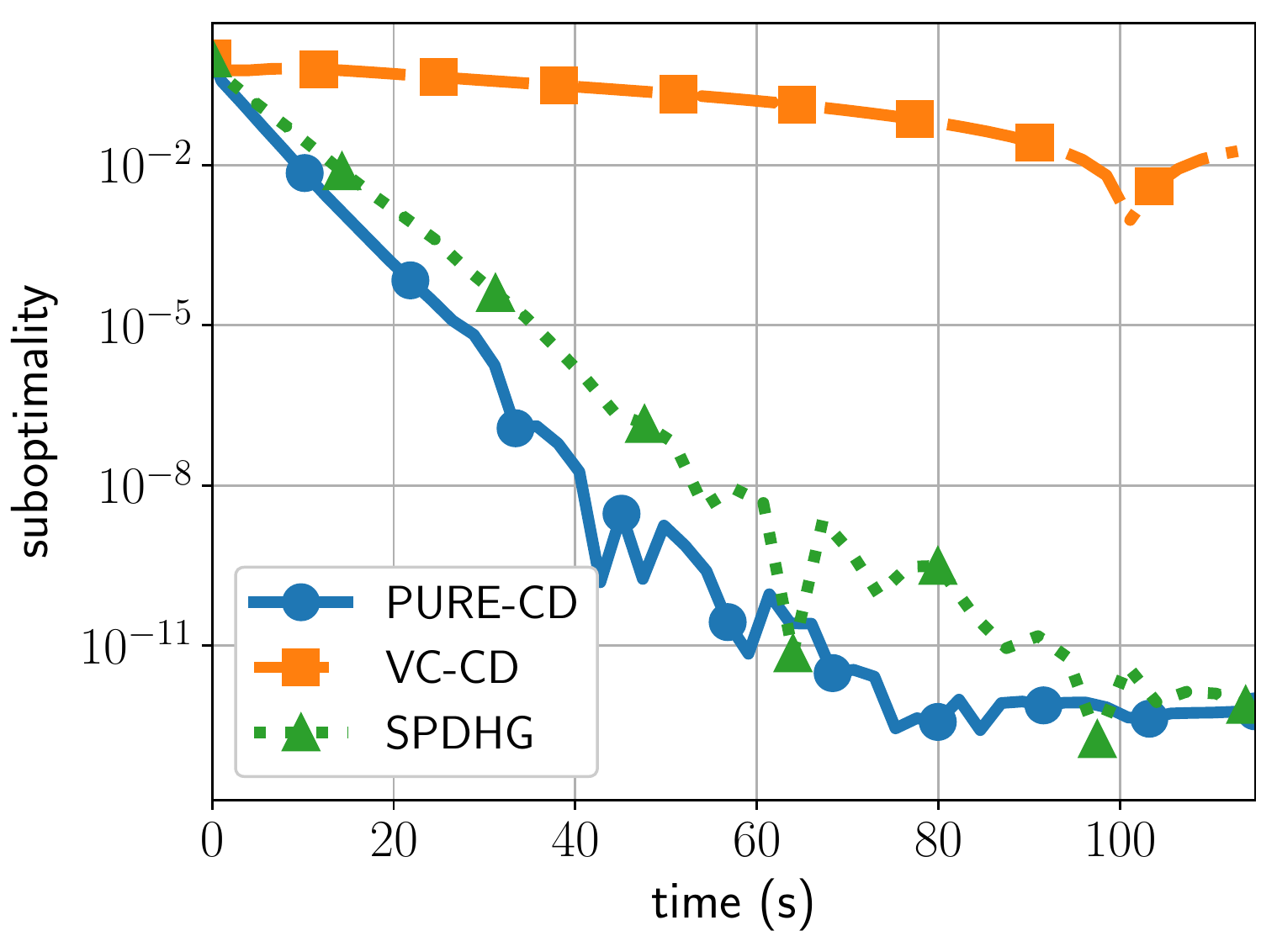}
\caption{Lasso: Left: rcv1, $n=20,242, m=47,236$, density $=0.16 \%$, $\lambda=10$; Middle: w8a, $n=49,749, m=300$, density $=3.9\%$, $\lambda=10^{-1}$; Right: covtype, $n=581,012$, $m=54$, density $=22.1\%$, $\lambda=10$.}
\label{fig:1}
\end{center}
\vskip -0.2in
\end{figure*}
\begin{figure*}[ht]
\begin{center}
\includegraphics[width=0.6\columnwidth]{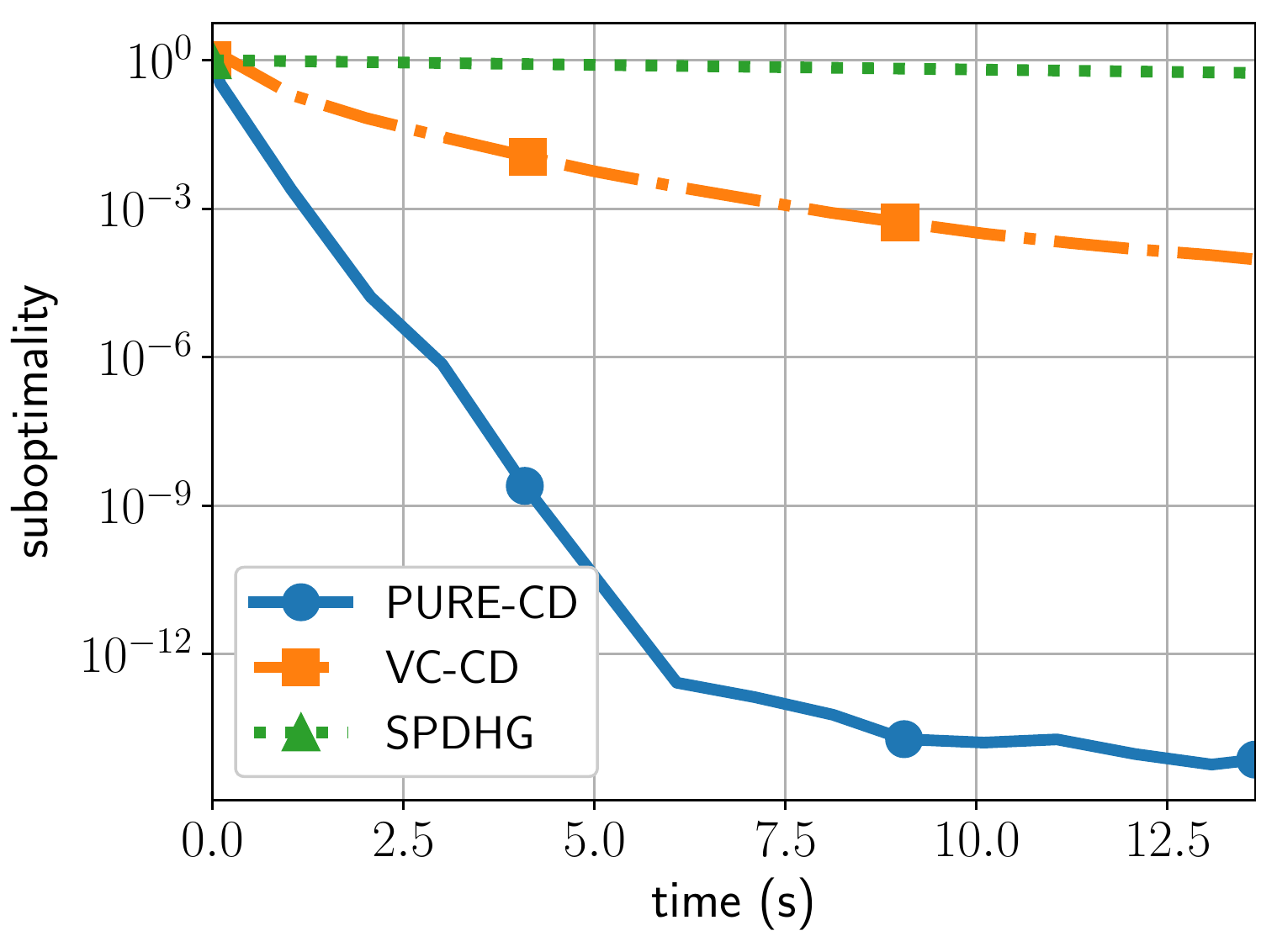}
\includegraphics[width=0.6\columnwidth]{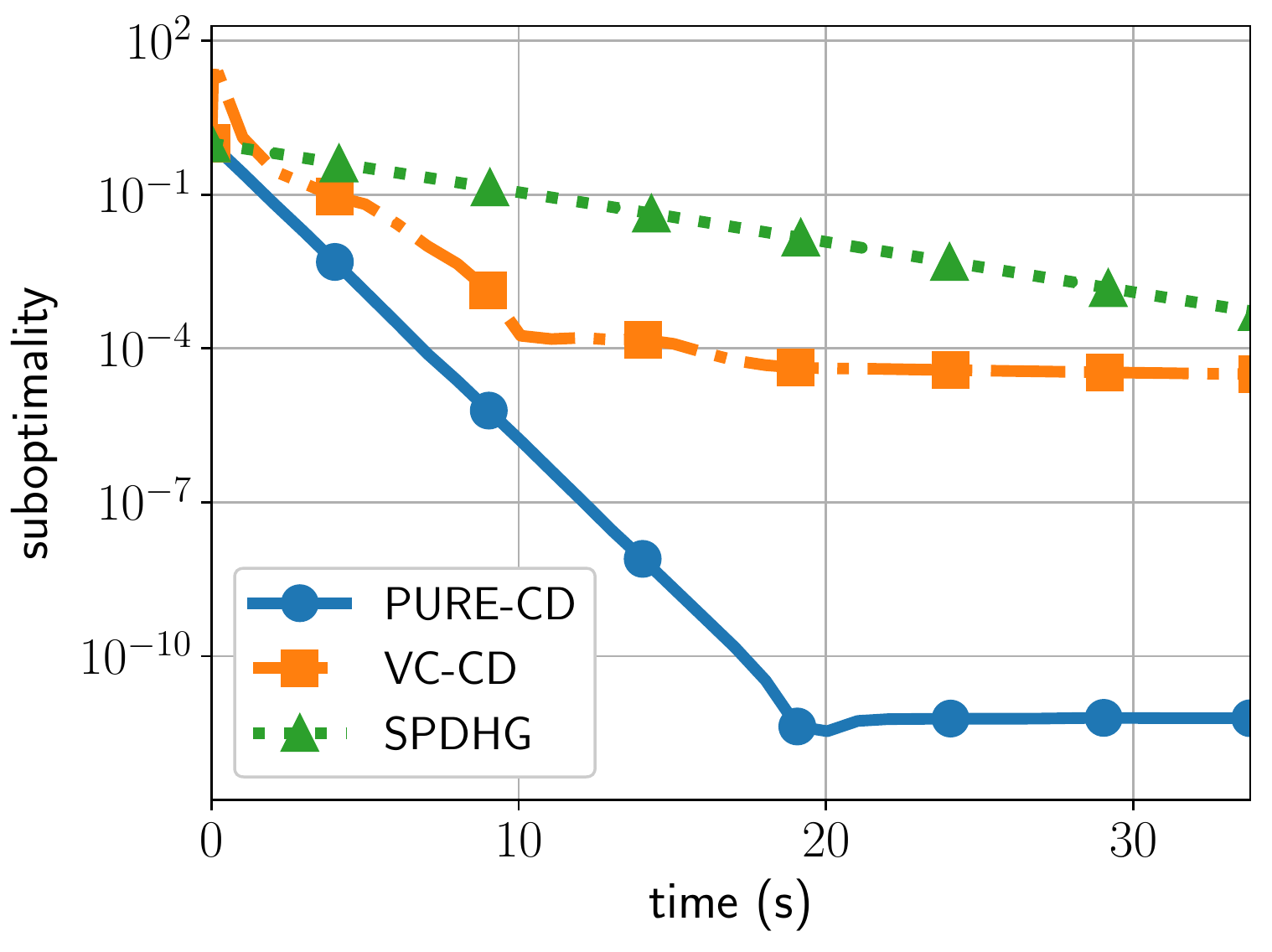}
\includegraphics[width=0.6\columnwidth]{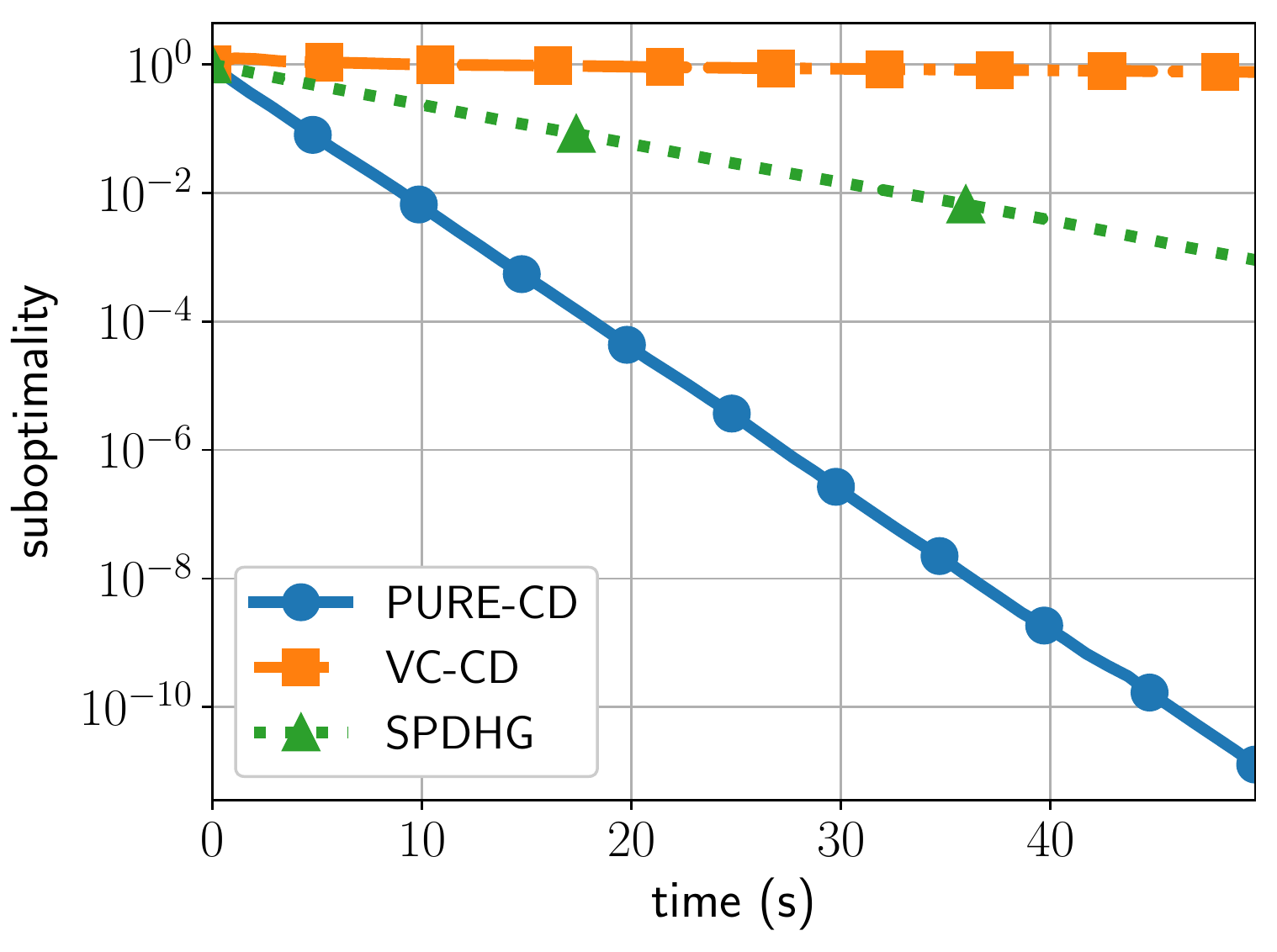}
\caption{Ridge regression: Left: sector, $n=6,412, m=55,197$, density $=0.3 \%$, $\lambda=0.1$; Middle: a9a, $n=32,561, m=123$, density $=11.3\%$, $\lambda=0.1$; Right: mnist, $n=60,000$, $m=780$, density $=19.2\%$, $\lambda=1$.}
\label{fig:2}
\end{center}
\vskip -0.2in
\end{figure*}

PDCD variants are also proposed in~\cite{combettes2015stochastic,combettes2019stochastic,pesquet2015class} and analyzed under the general setting of monotone operators. These methods use global constants of the problem such as global Lipschitz constant of smooth part and $\|A\|$, rather than blockwise constants, resulting in worse practical performance, as illustrated in the experiments of~\cite{chambolle2018stochastic}.

Another early PDCD variant to solve problem~\eqref{eq: prob_temp} in its full generality, where $f, g, h$ are all nonseparable, is by~\citet{fercoq2019coordinate}.
This method uses coordinatewise Lipschitz constants of the smooth part and it is designed to exploit sparsity of $A$.
This method has almost sure convergence guarantees as well as linear convergence when $g, h^\ast$ are strongly convex.
As opposed to most results in this nature, it is not required to know strong convexity constants to set the step sizes.
In the general convex case, the method has $\mathcal{O}(1/\sqrt{k})$ rate for a randomly selected iterate.
As argued in Section~\ref{sec: as}, main limitation of~\cite{fercoq2019coordinate} is that small step sizes are required when matrix $A$ is dense.
Moreover, the results in this paper are restricted to uniform probability law for selecting coordinates.

One of the most related works to ours, and a building block of PURE-CD is TriPD-BC from~\cite{latafat2019new}.
The authors showed almost sure convergence of the iterates and linear convergence under metric subregularity, by using global Lipschitz constants of $f$ for the step sizes. This paper did not have any sublinear convergence rates in the general convex case.
Similar to~\cite{fercoq2019coordinate}, TriPD-BC is designed for sparse setting and a naive implementation in the dense setting requires the same per iteration cost as the deterministic algorithm.
An efficient implementation is by duplication of dual variables, which as explained in~\cite{fercoq2019coordinate} results in small step sizes.

Another building block of PURE-CD is SPDHG by~\cite{chambolle2018stochastic}, to solve~\eqref{eq: prob_temp} when $f=0$.
Linear convergence result of SPDHG by~\cite{chambolle2018stochastic} is similar to~\cite{zhang2017stochastic} and requires setting step sizes with strong convexity constants.
In the general convex case and partially strongly convex case,~\cite{chambolle2018stochastic} proved optimal sublinear rates.
Recently,~\cite{alacaoglu2019convergence} analyzed SPDHG and proved additional theoretical results.
In particular, this work showed almost sure convergence of the iterates of SPDHG and linear convergence under metric subregularity.
Even though it is fast in the dense setting, the main limitation of SPDHG, as discussed in Section~\ref{sec: as} is that it needs to update all the dual coordinates, resulting in high per iterations costs in the sparse setting.

Similar algorithms are proposed in~\cite{luke2018block,alacaoglu2017smooth} where the authors focused on the linearly constrained problems and proved sublinear rates.

\section{Numerical experiments}\label{sec: numexp}
\subsection{Effect of sparsity}\label{sec:sparse}
As explained in Section~\ref{sec: as}, and Remark~\ref{rm: rm1}, PURE-CD brings together the benefits of different methods that are designed for dense and sparse cases.
We will now compare the empirical performance of PURE-CD with Vu-Condat-CD from~\cite{fercoq2019coordinate} which has desirable properties with sparse data and SPDHG from~\cite{chambolle2018stochastic} which has desirable properties with dense data.

We select uniform sampling, $p_i = 1/n$, so~\eqref{eq: ss_choice} simplifies to
\begin{align}\label{eq: ss_rule_exp}
\tau_i < \frac{1}{\sum_{j=1}^m \theta_j \sigma_j A_{j, i}^2}.
\end{align}
We provide a step size policy inspired by the step size rules chosen in~\cite{chambolle2018stochastic} and~\cite{fercoq2019coordinate}.
We use the following step sizes, for $\gamma < 1$
\begin{align}
\sigma_j = \frac{1}{\theta_j \max_{i'}\|A_{i'}\|}, ~~~~ \tau_i = \frac{\gamma \max_{i'} \|A_{i'} \|}{\| A_i \|^2}.\notag
\end{align}
We note that in contrast to~\cite{chambolle2018stochastic}, step sizes are both diagonal. In our case, it is important to utilize diagonal step sizes for both primal and dual variables since we perform coordinate-wise updates for both primal and dual variables and the step sizes need to be set appropriately to obtain good practical performance.
For SPDHG and Vu-Condat-CD, we use step sizes suggested in the papers.

In the edge cases (one nonzero element per row or fully dense), it is easy to see that our step size policy reduces to the suggested step sizes of~\cite{chambolle2018stochastic} and~\cite{fercoq2019coordinate}.

For experiments, we used the generic coordinate descent solver, implemented in Cython, by~\citet{fercoq2019generic}, which includes an implementation of Vu-Condat-CD with duplication and we implemented SPDHG and PURE-CD.
We solve Lasso and ridge regression, where we let
$g(x) = \lambda \|x\|_1$, $h(Ax) = \frac{1}{2} \| Ax-b\|^2$, $f=0$, and $g(x) = \frac{\lambda}{2} \|x\|^2$, $h(Ax) = \frac{1}{2} \| Ax-b\|^2$, $f=0$, respectively, in our template~\eqref{eq: prob_temp}. Then, we apply all the methods to the dual problems of these, to access data by rows.

\begin{figure*}[ht]
\begin{center}
\includegraphics[width=1.8\columnwidth]{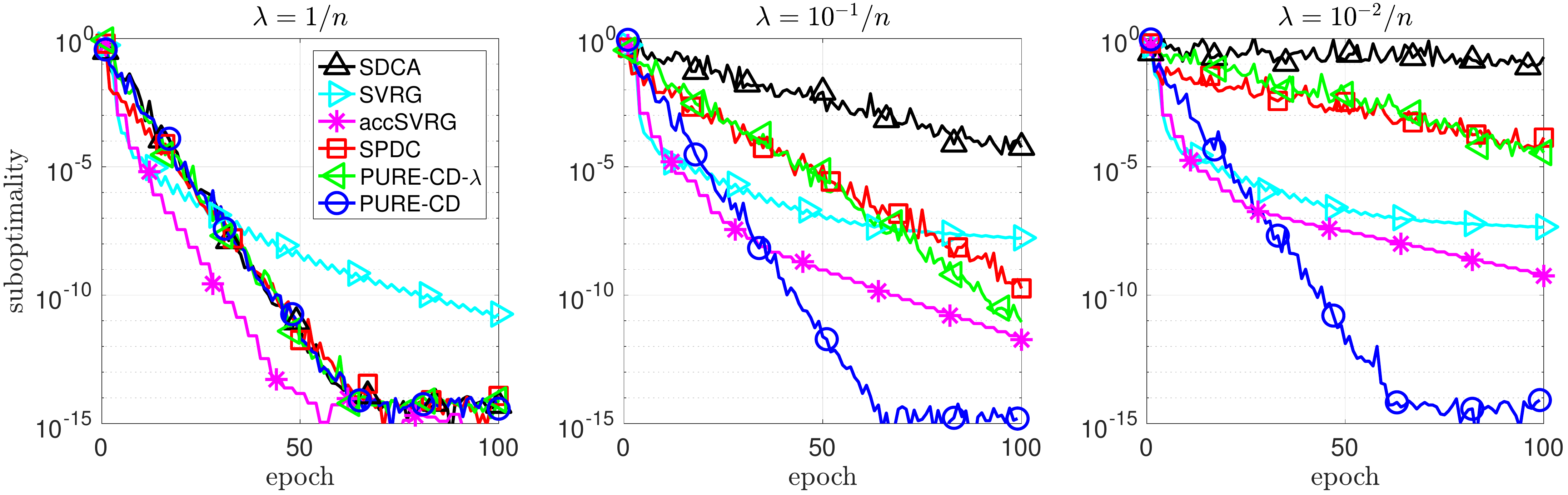}
\end{center}
\vskip -0.2in
\end{figure*}
\begin{figure*}[ht]
\begin{center}
\includegraphics[width=1.8\columnwidth]{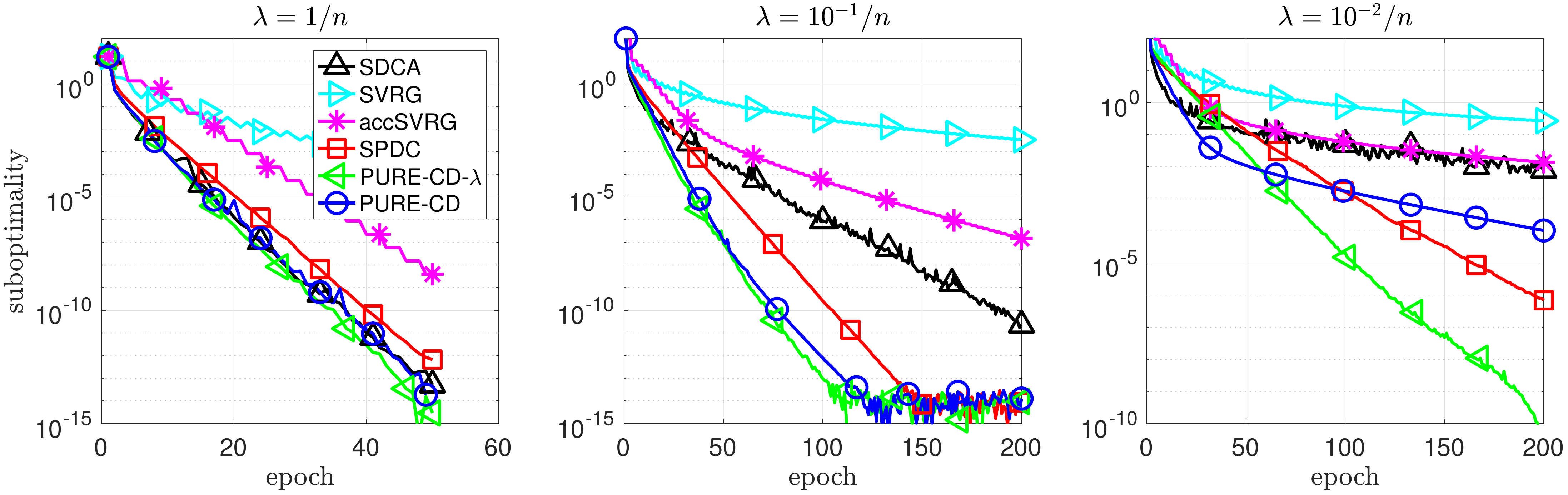}
\caption{top: a9a, $n=32,561$, $m=123$, bottom: sector, $n=6,412$, $m=55,197$.}
\label{fig:3}
\end{center}
\vskip -0.2in
\end{figure*}

We use datasets from LIBSVM with different sparsity levels~\cite{CC01a}. The properties of each data matrix are given in the caption of the corresponding figures.
For preprocessing, we removed all-zero rows and all-zero columns of $A$ and we performed row normalization.
The results are compiled in~\Cref{fig:1,fig:2}.

We observe the behavior predicted by theory. With sparse data such as rcv1, where density level is $0.16\%$, SPDHG makes very little progress in the given time window. The reason is that the per iteration cost of SPDHG in this case is updating $47,236$ dual variables, whereas for PURE-CD and Vu-Condat-CD, the cost is updating $75$ dual variables.
We note that PURE-CD is faster than Vu-Condat-CD due to better step sizes.
On the other hand, with moderate sparsity, SPDHG and Vu-Condat-CD is comparable, whereas PURE-CD exhibits the best performance.
For denser data, SPDHG and PURE-CD exhibit similar behavior where Vu-Condat-CD is slower than both due to smaller step sizes.

\subsection{Comparison with specialized methods}
In this section, we compare the practical performance of PURE-CD with state-of-the-art algorithms that are designed for strongly convex-strongly concave problems.
Due to space constraints, we defer some of the plots and more details about experiments to the appendix.
We focus on the problem
$
\min_x \frac{1}{n} \sum_{i=1}^n h_i(A_i x) + \frac{\lambda}{2} \| x \|^2,
$
where $h_i(x) = (x - b_i)^2$.
Each $h_i$ is smooth with Lipschitz constants $L_i = 2$ and the second component is strongly convex.
This is equivalent to strong convexity in both primal and dual problems.

In this case, the algorithms SDCA~\cite{shalev2013stochastic}, ProxSVRG~\cite{xiao2014proximal}, Accelerated SVRG~\cite{zhou2018simple}, SPDC~\cite{zhang2017stochastic} are all designed to use the strong convexity to obtain linear convergence.
These algorithms use the strong convexity constant $\lambda$ for setting the algorithmic parameters.
Moreover, as all the abovementioned algorithms have special implementations to exploit sparsity, we make the comparison with respect to number of passes of the data, rather than time.
The results are compiled for two datasets in~\Cref{fig:3} and more datasets are included in~\Cref{app: exp}.

$\bullet$ PURE-CD-$\lambda$: This variant uses the non-agnostic step sizes, using $\lambda$, which still satisfy the theoretical requirement~\eqref{eq: ss_rule_exp}.
\begin{align}
\sigma_j = \frac{n}{\theta_j \sqrt{n\lambda} \max_{i'}\|A_{i'}\|}, ~~~~ \tau_i = \frac{\gamma\sqrt{n\lambda} \max_{i'} \|A_{i'} \|}{n\| A_i \|^2}.\notag
\end{align}
$\bullet$ PURE-CD: This variant is with the standard agnostic step sizes. We note that the step sizes are scaled by $n$ since the problem is scaled by $1/n$, compared to~\Cref{sec:sparse}.
\begin{align}
\sigma_j = \frac{n}{\theta_j \max_{i'}\|A_{i'}\|}, ~~~~ \tau_i = \frac{\gamma \max_{i'} \|A_{i'} \|}{n \| A_i \|^2}.\notag
\end{align}
We observe that PURE-CD has a consistent linear convergence behavior as predicted by theory.
In most of the datasets (see~\Cref{app: exp}), it has the fastest convergence behavior.
However, in some datasets, as $\lambda$ gets smaller, we observed that the linear rate of PURE-CD slowed down, which motivated us to try PURE-CD-$\lambda$, which incorporates the knowledge of $\lambda$ as the other methods.
It seems to show favorable behavior when PURE-CD slows down.

The takeaway message is that PURE-CD, which is designed for a general problem, adapts to strong convexity well with agnostic step sizes in most cases.
However, in some cases, it does not perform as good as the algorithms which are designed to exploit strong convexity.
In those cases however one can choose separating step sizes of PURE-CD accordingly, and use PURE-CD-$\lambda$ to get better performance.

\section*{Acknowledgements}
This project has received funding from the European Research Council (ERC) under the European Union's Horizon $2020$ research and innovation programme (grant agreement no $725594$ - time-data) and the Swiss National Science Foundation (SNSF) under grant number $200021\_178865 / 1$.

\bibliography{purecd}
\bibliographystyle{icml2020}

\appendix
\onecolumn
\allowdisplaybreaks

\newpage
\section{More experimental results}\label{app: exp}
In this section, we compare the practical performance of PURE-CD with state-of-the-art algorithms that are designed to exploit problem structures.
In particular, we focus on the problem
\begin{equation}\label{eq: app_exp}
\min_x \frac{1}{n} \sum_{i=1}^n h_i(A_i x) + \frac{\lambda}{2} \| x \|^2,
\end{equation}
where $h_i(x) = (x - b_i)^2$.
Each $h_i$ is smooth with Lipschitz constants $L_i = 2$ and the second component is strongly convex.
This is equivalent to strong convexity in both primal and dual problems.

In this case, the algorithms SDCA, SVRG, Accelerated SVRG/Katyusha, SPDC are all designed to use the strong convexity to obtain linear convergence.
These algorithms use the strong convexity constant $\lambda$ for setting the algorithmic parameters (with the exception of SVRG which theoretically needs it to set number of inner loop iterations).
Moreover, as all the abovementioned algorithms have special implementations to exploit sparsity, for fairness, as all algorithms have different structures, we did not try to implement them in the most efficient manner, therefore we make the comparison with respect to number of passes of the data, rather than time.

We use PURE-CD with the agnostic step size and also with a non-agnostic step size that uses $\lambda$.
Both step size rules are supported by theory.
Moreover, similar to~\Cref{sec: numexp}, we apply PURE-CD to the dual problem of~\eqref{eq: app_exp} to access the data row-wise as other methods.

The details of parameters for each algorithm:\\
$\bullet$ SVRG: We use the theoretical step size which is chosen as $\frac{1}{4\max_i L_i}$.~\citep[Theorem 3.1]{xiao2014proximal} \\[2mm]
$\bullet$ Accelerated SVRG/Katyusha: We use the theoretically suggested step size parameter and acceleration parameter~\citep[Theorem 1, Table 2]{zhou2018simple}\\[2mm]
$\bullet$ SDCA: We use directly the specialization of SDCA for ridge regression, as decribed in~\citep[Section 6.2]{shalev2013stochastic}\\[2mm]
$\bullet$ SPDC: We use the step sizes from~\citep[Theorem 1]{zhang2017stochastic}\\[2mm]
$\bullet$ PURE-CD-$\lambda$: This variant uses the non-agnostic step sizes, using $\lambda$. We note that the step sizes satisfy the theoretical requirement~\eqref{eq: ss_rule_exp}.
\begin{align}
\sigma_j = \frac{n}{\theta_j \sqrt{n\lambda} \max_{i'}\|A_{i'}\|}, ~~~~ \tau_i = \frac{\gamma\sqrt{n\lambda} \max_{i'} \|A_{i'} \|}{n\| A_i \|^2}.\notag
\end{align}
$\bullet$ PURE-CD: This variant is with the standard agnostic step sizes, as in Section~\ref{sec: numexp}. We note that the step sizes are scaled by $n$ since the problem is scaled by $1/n$, compared to Section~\ref{sec: numexp}.
\begin{align}
\sigma_j = \frac{n}{\theta_j \max_{i'}\|A_{i'}\|}, ~~~~ \tau_i = \frac{\gamma \max_{i'} \|A_{i'} \|}{n \| A_i \|^2}.\notag
\end{align}

We use datasets from LIBSVM, and try three different regularization parameters $\frac{1}{n}$, $\frac{10^{-1}}{n}$, and $\frac{10^{-2}}{n}$.
We performed preprocessing by removing the all-zero rows and columns from the data matrix, then we normalized row norms of $A$.
We chose the parameters as described above and did not perform any tuning for any algorithm.

We observe that PURE-CD has a consistent linear convergence behavior as predicted by theory.
In most of the datasets, it has the fastest convergence behavior.
However, in some datasets, as $\lambda$ gets smaller, we observed that the linear rate of PURE-CD slowed down, which motivated us to try PURE-CD-$\lambda$, which incorporates the knowledge of $\lambda$ as the other methods.
It seems to show favorable behavior when PURE-CD slows down.

The takeaway message is that PURE-CD adapts very well with agnostic step sizes in most cases.
However, in some cases, it does not perform as good as the algorithms which are designed to exploit structure.
In those cases however one can choose separating step sizes of PURE-CD accordingly, and use PURE-CD-$\lambda$ to get better performance.

\begin{figure}[H]
\begin{center}
\includegraphics[width=0.9\columnwidth]{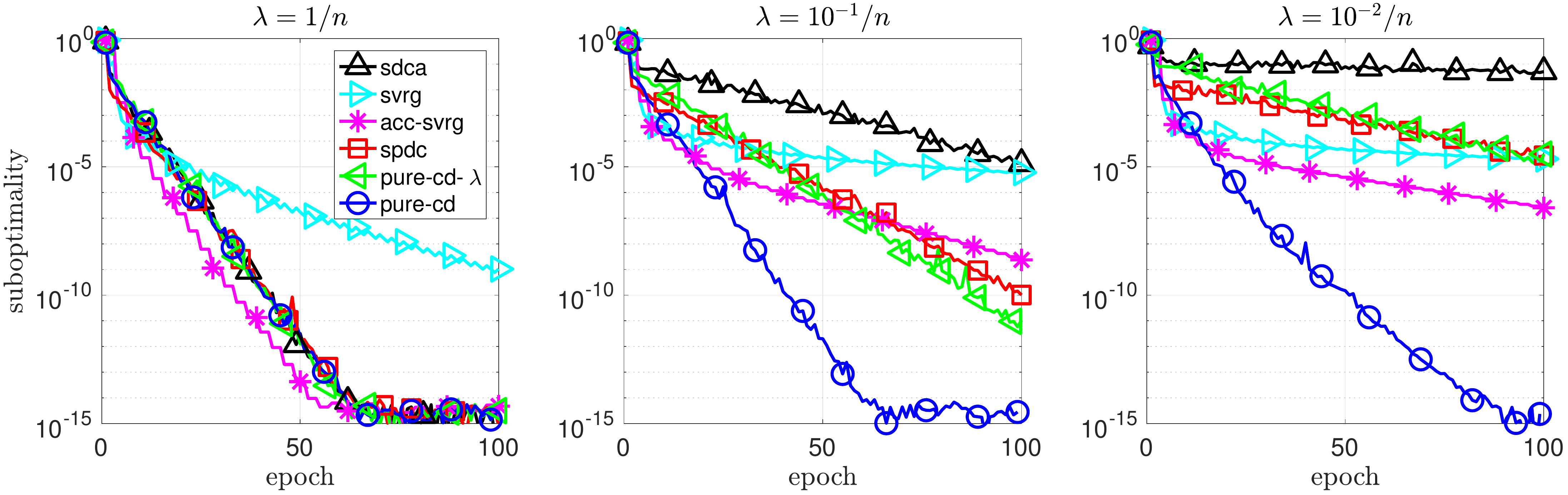}
\caption{w8a, $n=49,749$, $m=300$.}
\includegraphics[width=0.9\columnwidth]{a9a_ridge_ml.pdf}
\caption{a9a, $n=32,561$, $m=123$.}
\includegraphics[width=0.9\columnwidth]{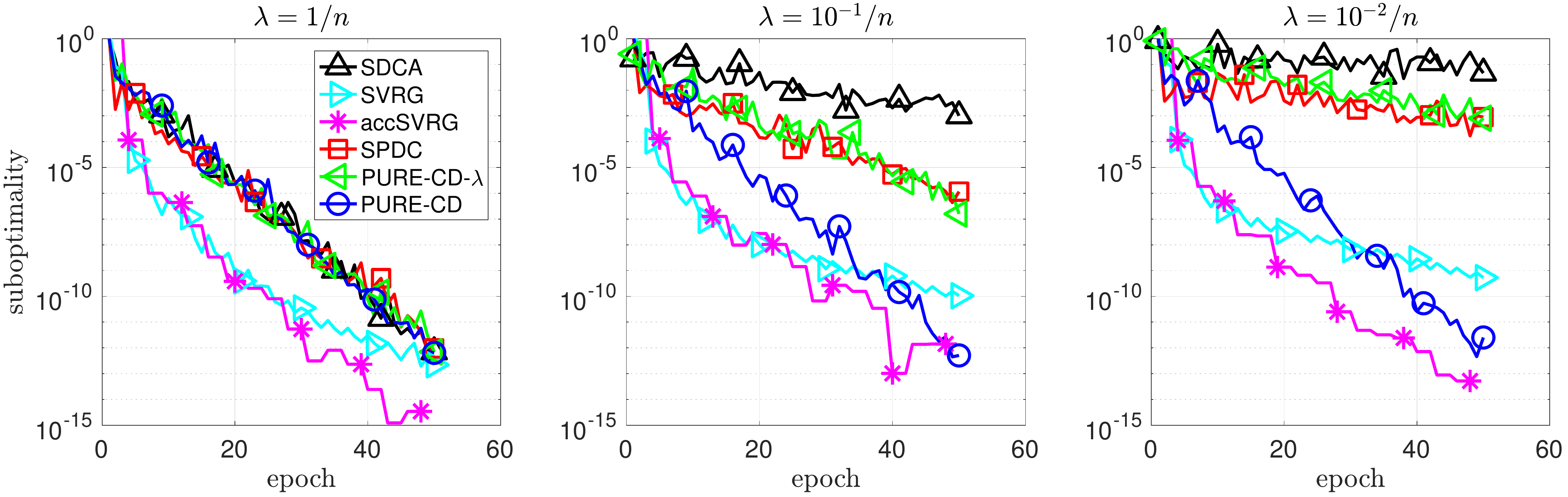}
\caption{covtype, $n=581,012$, $m=54$.}
\includegraphics[width=0.9\columnwidth]{sec_ridge_ml.pdf}
\caption{sector, $n=6,412$, $m=55,197$.}
\end{center}
\end{figure}
\begin{figure}[H]
\begin{center}
\includegraphics[width=0.9\columnwidth]{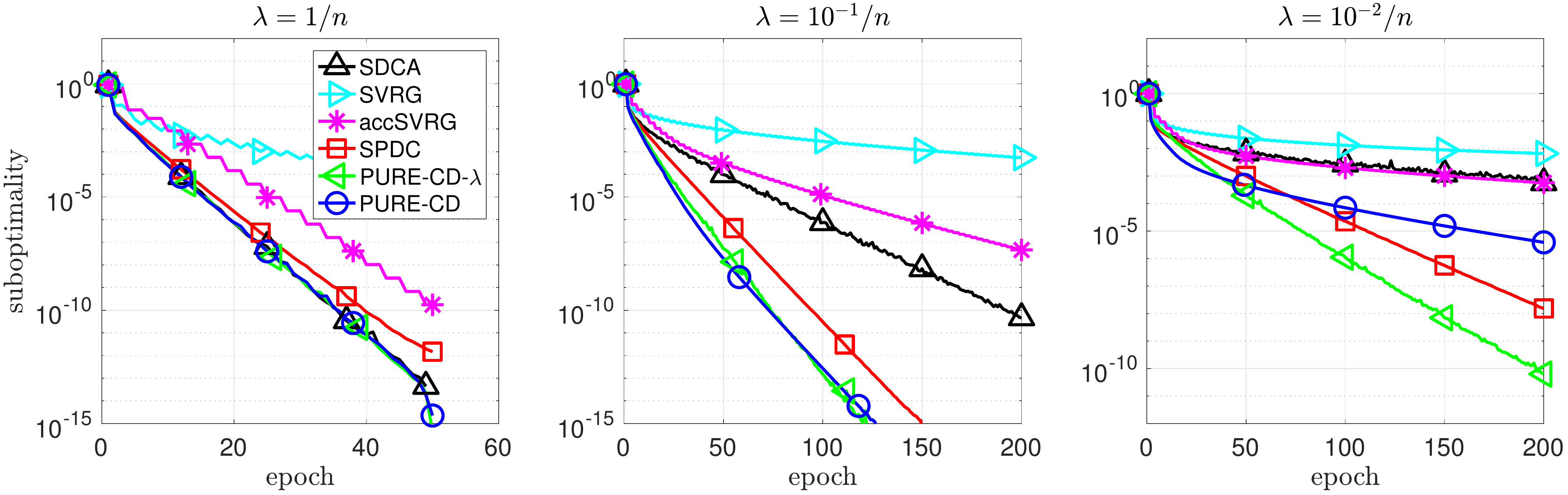}
\caption{rcv1.binary, $n=20,242$, $m=47,236$.}
\includegraphics[width=0.9\columnwidth]{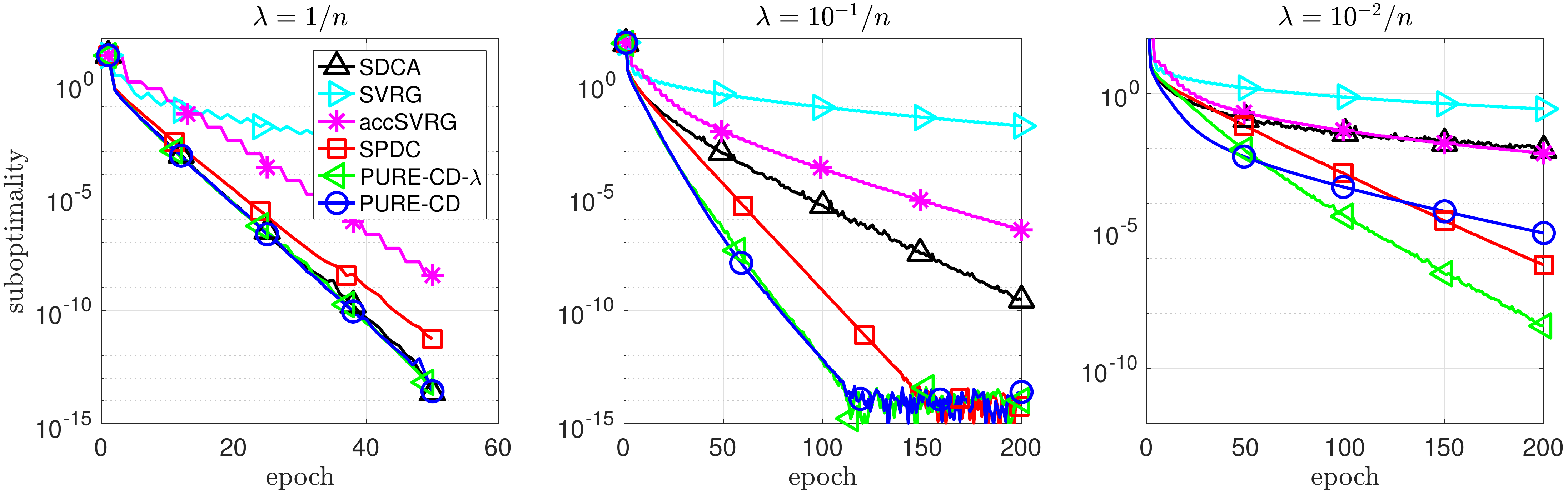}
\caption{news20, $n=15,935$, $m=62,061$.}
\includegraphics[width=0.9\columnwidth]{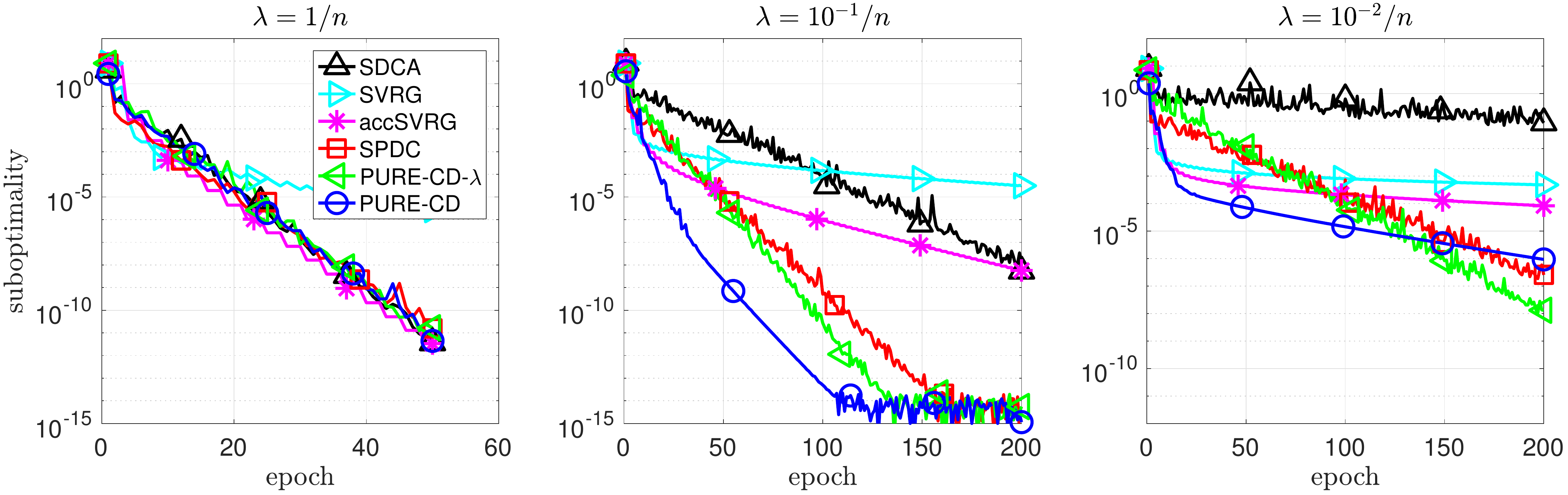}
\caption{mnist, $n=60,000$, $m=780$.}
\includegraphics[width=0.9\columnwidth]{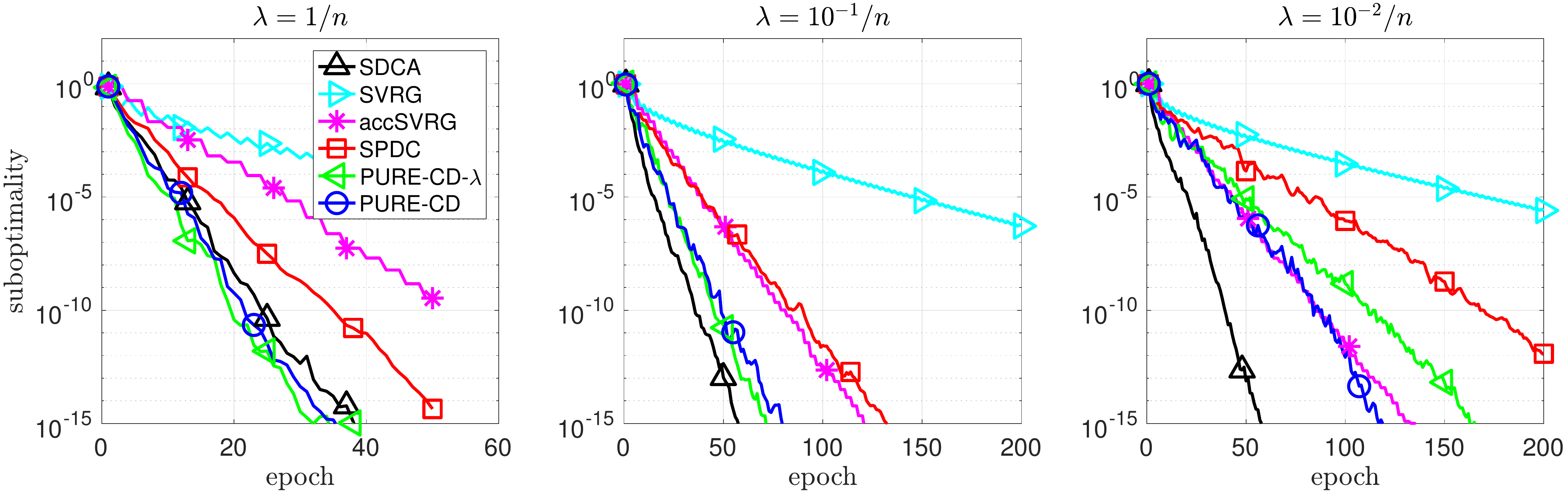}
\caption{leukemia, $n=38$, $m=7,129$.}
\end{center}
\end{figure}
\begin{figure}[H]
\begin{center}
\includegraphics[width=0.9\columnwidth]{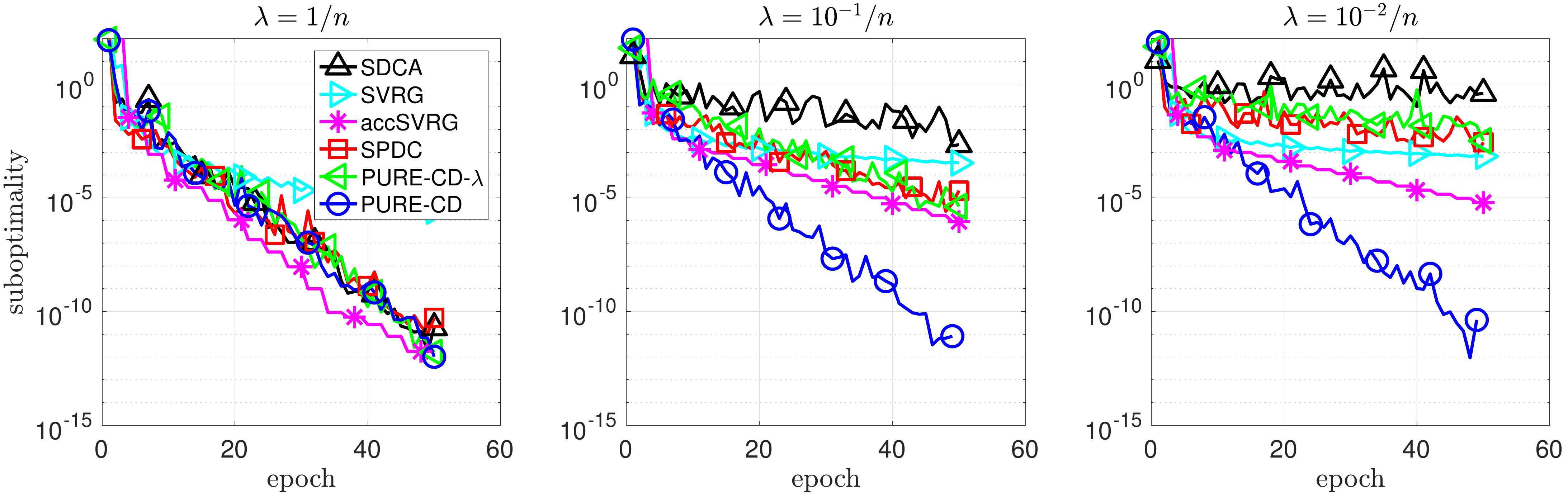}
\caption{YearPredictionMSD, $n=463,715$, $m=90$.}
\end{center}
\end{figure}

\subsection{Further details about experiments of~\Cref{sec: numexp}}
For preprocessing, we removed all-zero rows and all-zero columns of $A$ and we performed row normalization.
The experiments are done on a computer with Intel Core i7 CPUs at 2.9 GHz.

\newpage
\section{Proofs}\label{sec: proofs}
\subsection{Proofs for one iteration result}
We start with some technical lemmas.
Our first result characterizes the conditional expectation of $y_{k+1}$.

\begin{lemma}\label{lem: y_lem}
Let $y_{k+1}$ as defined in Algorithm~\ref{alg:stripd}, and recall that $\pi = \diag(\pi_1, \dots, \pi_m)$, such that $\pi_j = \sum_{i\in I(j)} p_i, \forall j \in \{ 1, \dots, m \}$.
Then it holds that for any $\mathcal{F}_k$-measurable $Y$ and $\forall \gamma = \{ \gamma_1, \dots, \gamma_m \}$, with $\gamma_i > 0$,
\begin{multline}
\mathbb{E}_k \left[ \| y_{k+1} - Y \|^2_{\gamma} \right] = \| \bar{y}_{k+1} - Y \|^2_{\gamma \pi} - \| y_k - Y \|^2_{\gamma\pi} +\| y_k - Y \|^2_{\gamma} + 2\langle \bar{y}_{k+1} - Y, \gamma \sigma \theta A P(\bar{x}_{k+1} - x_k) \rangle \\
+ \sum_{i=1}^n \sum_{j=1}^m p_i \gamma_j \sigma_j^2 \theta_j^2 A_{j, i}^2  (\bar{x}_{k+1}^i - x_k^i)^2.\notag
\end{multline}
\end{lemma}
\begin{proof}
First, we note that for $\mathcal{F}_k$-measurable $Y$, it follows that
\begin{align}
\mathbb{E}_k\left[ \| y_{k+1} - Y \|^2_{\gamma} \right] &= \mathbb{E}_k \left[ \sum_{j=1}^m \gamma_j \left( y_{k+1}^j - Y^j \right)^2 \right] \notag\\
&= \mathbb{E}_k \left[ \sum_{j\in J(i_{k+1})} \gamma_j (\bar{y}_{k+1}^j + \sigma_j\theta_j (A(x_{k+1} - x_k))_j - Y^j)^2 + \sum_{j \not\in J(i_{k+1})} \gamma_j (y_k^j - Y^j)^2 \right] \notag\\
&= \sum_{i=1}^n p_i \left[ \sum_{j\in J(i)}\gamma_j \left( \bar{y}_{k+1}^j + \sigma_j \theta_jA_{j, i} (\bar{x}_{k+1}^i - x_k^i) - Y^j \right)^2 + \sum_{j\notin J(i)} \gamma_j\left( y_k^j - Y^j \right)^2 \right] \notag\\
&= \sum_{i=1}^n \sum_{j\in J(i)} p_i \gamma_j (\bar{y}_{k+1}^j - Y^j)^2 + 2\sum_{i=1}^n \sum_{j\in J(i)} p_i \gamma_j \sigma_j \theta_j A_{j, i} (\bar{x}_{k+1}^i - x_k^i) (\bar{y}_{k+1}^j - Y^j) \notag\\
&+ \sum_{i=1}^n \sum_{j\in J(i)} p_i \gamma_j \left( \sigma_j \theta_jA_{j, i}  (\bar{x}_{k+1}^i - x_k^i) \right)^2 + \sum_{i=1}^n \sum_{j\notin J(i)} p_i \gamma_j (y_k^j - Y^j)^2\label{eq: yk_exp_first}
\end{align}
where for the third equality, we used the fact that $x_{k+1}$ is different from $x_k$ only on the coordinate $i_{k+1}$, which gives
\begin{equation}
(A(x_{k+1}-x_k))_j = (A(x_{k+1}^{i_{k+1}}-x_k^{i_{k+1}})e_{i_{k+1}})_j = A_{j, i_{k+1}}(\bar{x}_{k+1}^{i_{k+1}}-x_k^{i_{k+1}}).\notag
\end{equation}
We focus on the first term on the right hand side of~\eqref{eq: yk_exp_first}
\begin{align}
\sum_{i=1}^n\sum_{j\notin J(i)} p_i \gamma_j (y_k^j - Y^j)^2 &= \sum_{i=1}^n \sum_{j=1}^n p_i \gamma_j (y_k^j - Y^j)^2 - \sum_{i=1}^n \sum_{j\in J(i)} p_i \gamma_j (y_k^j - Y^j)^2\notag\\
&= \| y_k - Y \|^2_{\gamma} - \sum_{j=1}^m \sum_{i\in I(j)} p_i \gamma_j (y_k^j - Y^j)^2 = \| y_k - Y \|^2_{\gamma} - \sum_{j=1}^m \pi_j \gamma_j (y_k^j - Y^j)^2 \notag\\
&=\| y_k - Y \|^2_{\gamma} - \|y_k - Y\|^2_{\gamma \pi}\label{eq: yk_exp_rhs1}
\end{align}
where we use the fact that $\sum_{i=1}^n \sum_{j\in J(i)} \gamma'_{j, i} = \sum_{j=1}^m \sum_{i\in I(j)} \gamma'_{j, i}$, for any $\gamma'$, due to the definition of $J(i)$, $I(j)$ and $\pi_j = \sum_{i\in I(j)} p_i$.

We estimate the last term of~\eqref{eq: yk_exp_first}, similar to~\eqref{eq: yk_exp_rhs1}
\begin{align}
\sum_{i=1}^n \sum_{j\in J(i)} p_i \gamma_j (\bar{y}_{k+1}^j - Y^j)^2 = \sum_{j=1}^m \sum_{i\in I(j)} p_i \gamma_j (\bar{y}_{k+1}^j - Y^j)^2 = \| \bar{y}_{k+1} - Y \|^2_{\gamma \pi}.\label{eq: yk_exp_rhs4}
\end{align}
We lastly estimate the second and third terms of~\eqref{eq: yk_exp_first}. We use the fact that $A_{j, i} = 0$, if $j\notin J(i)$ to obtain
\begin{multline}
2\sum_{i=1}^n\sum_{j\in J(i)} p_i \gamma_j \sigma_j\theta_j A_{j, i}  (\bar{x}_{k+1}^i - x_k^i) (\bar{y}_{k+1}^j - Y^j) + \sum_{i=1}^n \sum_{j\in J(i)} p_i \gamma_j \left( \sigma_j \theta_j A_{j, i} (\bar{x}_{k+1}^i - x_k^i) \right)^2 \\
=2\sum_{i=1}^n \sum_{j=1}^m p_i \gamma_j \sigma_j \theta_j A_{j, i}  (\bar{x}_{k+1}^i - x_k^i) (\bar{y}_{k+1}^j - Y^j) + \sum_{i=1}^n \sum_{j=1}^m p_i \gamma_j \left( \sigma_j \theta_j A_{j, i} (\bar{x}_{k+1}^i - x_k^i) \right)^2 \\
=2 \langle \bar{y}_{k+1} - Y, \gamma \sigma\theta A P(\bar{x}_{k+1} - x_k) \rangle + \sum_{i=1}^n\sum_{j=1}^m p_i\gamma_j \left( \sigma_j \theta_j A_{j, i}  (\bar{x}_{k+1}^i - x_k^i) \right)^2.\label{eq: yk_exp_rhs23}
\end{multline}
We use~\eqref{eq: yk_exp_rhs1},~\eqref{eq: yk_exp_rhs4}, and~\eqref{eq: yk_exp_rhs23} in~\eqref{eq: yk_exp_first} to obtain the final result.
\end{proof}

We continue with the following lemma which handles necessary manipulations for the terms involving the primal variable, to handle arbitrary probabilities.

\begin{lemma}\label{lem: x_lem}
We recall that $P=\diag(p_1, \ldots, p_n)$, $\underline p = \min_{i} p_i$, and define
\begin{equation}
x' = x_k + P^{-1}\underline p (x - x_k) = P^{-1}\underline p x + (1-P^{-1}\underline p) x_k,\notag
\end{equation}
and
\begin{equation}
g_P(x) = \sum_{i=1}^n p_i g_i(x).\notag
\end{equation}
Then, for a function $g(x)=\sum_{i=1}^n g_i(x_i)$, the following conclusions hold:
\begin{equation}
g_p(x') \leq \underline p g(x) - \underline p g(x_k) + g_p (x_k),\notag
\end{equation}
\begin{align}
\| x' - {x}_{k+1} \|^2_{\tau^{-1}} &= \underline p\| x - {x}_{k+1}\|^2_{\tau^{-1}P^{-1}} +  \| {x}_{k+1} - x_k \|^2_{\tau^{-1}} - \underline p \| {x}_{k+1} - x_k\|^2_{\tau^{-1}P^{-1}}  \notag\\
&-\underline p \|x-x_k\|^2_{\tau^{-1}P^{-1}} + \underline p^2 \| x- x_k \|^2_{\tau^{-1}P^{-2}},\notag
\end{align}
\begin{equation}
-\| x' - x_k \|^2_{\tau^{-1}} = -\underline p^2 \| x - x_k \|^2_{\tau^{-1}P^{-2}}.\notag
\end{equation}
\end{lemma}
\begin{proof}
We have that $x'^i = p_i^{-1}\underline p x^i + (1-p_i^{-1}\underline p)x_k^i$.
It follows by the convexity of $g_i$ that
\begin{equation}
g_p(x') = \sum_{i=1}^n p_i g_i(x'^i) \leq \sum_{i=1}^n \underline p g_i(x) + (p_i - \underline p) g_i(x_k^i) = \underline p g(x) - \underline p g(x_k) + g_p (x_k).\notag
\end{equation}
Moreover, since for any $0\leq c \leq 1$ and any $u, v \in \mathbb{R}$, it is true that $\|cu + (1-c)v\|^2 = c\|u\|^2 + (1-c)\|v \|^2 - c(1-c)\| u- v\|^2$, we obtain
\begin{align}
\| x' - {x}_{k+1} \|^2_{\tau^{-1}} &= \underline p \| x - {x}_{k+1}\|^2_{\tau^{-1}P^{-1}} + \| {x}_{k+1} - x_k \|^2_{\tau^{-1}} - \underline p \| {x}_{k+1} - x_k\|^2_{\tau^{-1}P^{-1}} \notag\\
&-\underline p \|x-x_k\|^2_{\tau^{-1}P^{-1}} + \underline p^2 \| x- x_k \|^2_{\tau^{-1}P^{-2}}.\notag
\end{align}
Lastly, plugging in $x' = x_k + P^{-1}\underline p (x-x_k)$ to $\| x' - x_k\|^2$ gives
\begin{equation*}
-\| x' - x_k \|^2_{\tau^{-1}} =- \underline p^2 \| x - x_k \|^2_{\tau^{-1}P^{-2}}.\qedhere
\end{equation*}
\end{proof}
We are now ready to prove Lemma~\ref{lem: lem1} which describes the one iteration behavior of the algorithm.
We restate the lemma and provide its proof.

\begin{replemma}{lem: lem1}
Under Assumption~\ref{asmp: asmp1}, let $\theta = \diag(\theta_1, \ldots, \theta_m)$ and $\pi = \diag(\pi_1, \ldots, \pi_m)$ be chosen as
\begin{equation}
\theta_j = \frac{\pi_j}{\underline p}, \text{ where } \pi_j = \sum_{i \in I(j)} p_i, \text{ and } \underline p = \min_i p_i.\notag
\end{equation}
We define the functions, given $z=(x, y)$,
\begin{align}
&V(z) = \frac{\underline p}{2} \| x \|^2_{\tau^{-1}P^{-1}} + \frac{\underline p}{2} \| y \|^2_{\sigma^{-1}\pi^{-1}},\notag\\
&\tilde{V}(z) = \frac{\underline p}{2} \| x \|^2_{C(\tau)} + \frac{\underline p}{2} \| y \|^2_{\sigma^{-1}},\notag
\end{align}
where 
$
C(\tau)_i = \frac{2p_i}{\underline p \tau_i} - \frac{1}{\tau_i} - p_i \sum_{j=1}^m \pi_j^{-1}\sigma_j \theta_j^2 A_{j, i}^2 - \frac{\beta_i p_i}{\underline p}$.

Then, for the iterates of Algorithm~\ref{alg:stripd}, $\forall z=(x, y) \in \mathcal{Z}$, it holds that:
\begin{equation}
\mathbb{E}_k \left[ D_p(x_{k+1}; z) \right] +\underline p D_d(\bar{y}_{k+1}; z) + \mathbb{E}_k \left[ V(z_{k+1} - z) \right] \\ \leq (1-\underline p)D_p (x_k; z) + V(z_k - z) - \tilde{V}(\bar{z}_{k+1} - z_k).\notag
\end{equation}
\end{replemma}

\begin{proof}
By the definition of proximal operator and convexity, $\forall x', y$
\begin{align}
p_i g_i(x'^i) &\geq p_ig_i(\bar{x}_{k+1}^i) - p_i \langle \nabla_i f(x_k)+ (A^\top \bar{y}_{k+1})_i, x'^i -\bar{x}_{k+1}^i \rangle \notag \\
&+ \frac{1}{2} \left( \|x_k^i - \bar{x}_{k+1}^i \|^2_{\tau_i^{-1}p_i} + \| x'^i - \bar{x}_{k+1}^i\|^2_{\tau_i^{-1}p_i} - \| x'^i - x_k^i \|^2_{\tau_i^{-1}p_i} \right) \notag\\
\underline{p} h^\ast(y) &\geq \underline{p} h^\ast(\bar{y}_{k+1}) + \underline{p}\langle A x_k, y - \bar{y}_{k+1} \rangle + \frac{\underline{p}}{2} \left( \| y_k - \bar{y}_{k+1} \|^2_{\sigma^{-1}} + \| y - \bar{y}_{k+1}\|^2_{\sigma^{-1}} - \| y - y_k \|^2_{\sigma^{-1}} \right).\notag
\end{align}
We sum the first inequality for $i=1$ to $n$, then add it to the second inequality and use the definition $g_P(x) = \sum_{i=1}^n p_i g_i(x_i)$ to derive
\begin{align}
g_P(x') + \underline{p} h^\ast(y) &\geq g_P(\bar{x}_{k+1}) + \underline{p}h^\ast(\bar{y}_{k+1}) - \langle \nabla f (x_k), P(x'-\bar{x}_{k+1}) \rangle - \langle A^\top\bar{y}_{k+1}, P(x'-\bar{x}_{k+1}) \rangle \notag \\
&+ \underline p \langle Ax_k, y-\bar{y}_{k+1}\rangle+\frac{1}{2} \Big( \| x_k - \bar{x}_{k+1}\|^2_{\tau^{-1}P} + \| x'- \bar{x}_{k+1}\|^2_{\tau^{-1}P} - \| x'- x_k \|^2_{\tau^{-1}P} \Big) \notag\\
&+\frac{\underline p}{2} \left( \| y_k - \bar{y}_{k+1}\|^2_{\sigma^{-1}} + \| y- \bar{y}_{k+1}\|^2_{\sigma^{-1}} - \| y- y_k \|^2_{\sigma^{-1}} \right).\label{eq: eq1_lastline}
\end{align}
First we note that, for $\mathcal{F}_k$-measurable $X$ and any $\gamma = \diag(\gamma_1, \ldots, \gamma_n)$, such that $\gamma_i > 0$, the following hold
\begin{align}
&\mathbb{E}_k [g(x_{k+1})] = g_P(\bar{x}_{k+1}) - g_P(x_k) + g(x_k),\label{eq: gp_exp} \\
&\mathbb{E}_k [x_{k+1}] = P\bar{x}_{k+1} - Px_k + x_k,\notag\\
&\mathbb{E}_k \left[\|x_{k+1} - X \|^2_\gamma\right] = \| \bar{x}_{k+1} - X \|^2_{\gamma P} - \| x_k - X \|^2_{\gamma P} + \| x_k - X \|^2_{\gamma}.\label{eq: x_cond}
\end{align}
We use~\eqref{eq: x_cond} to obtain
\begin{multline}
\frac{1}{2} \Big( \|x_k - \bar{x}_{k+1} \|^2_{\tau^{-1}P} + \| x' - \bar{x}_{k+1}\|^2_{\tau^{-1}P} - \| x' - x_k \|^2_{\tau^{-1}P} \Big) \\
= \frac{1}{2}\Big( \|x_k - \bar{x}_{k+1} \|^2_{\tau^{-1}P} + \mathbb{E}_k \left[ \| x' - x_{k+1}\|^2_{\tau^{-1}}\right] - \| x' - x_k \|^2_{\tau^{-1}} \Big).\label{eq: eq_x_three}
\end{multline}
We use $\gamma = \pi^{-1}\sigma^{-1}$ and $Y=y$ in Lemma~\ref{lem: y_lem}, then
\begin{multline}
\| \bar{y}_{k+1} - y \|^2_{\sigma^{-1}} = \mathbb{E}_k\left[ \| y_{k+1} - y \|^2_{\sigma^{-1}\pi^{-1}} \right] + \| y_k - y \|^2_{\sigma^{-1}} - \| y_k - y \|^2_{\sigma^{-1}\pi^{-1}} - 2 \langle \bar{y}_{k+1} - y, \pi^{-1}\theta A P(\bar{x}_{k+1}-x_k) \rangle\\
- \sum_{i=1}^n \sum_{j=1}^m p_i \pi^{-1}_j\sigma_j\theta_j^2 A_{j, i}^2  (\bar{x}_{k+1}^i - x_k^i)^2. \label{eq: y_res}
\end{multline}
We let $x'^i = p_i^{-1}\underline p x^i + (1-p_i^{-1}\underline p)x_k^i$, and use Lemma~\ref{lem: x_lem}
\begin{equation}
g_p(x') \leq \underline p g(x) - \underline p g(x_k) + g_p (x_k).\label{eq: gp}
\end{equation}
We further use $x' = P^{-1}\underline p x + (1-P^{-1}\underline p)x_k = x_k + P^{-1}\underline p(x-x_k)$ in~\eqref{eq: eq1_lastline} to obtain
\begin{align}
&- \langle \nabla f (x_k), P(x'-\bar{x}_{k+1})  = - \underline p\langle \nabla f (x_k), x-  x_k \rangle - \langle \nabla f (x_k), P( x_k- \bar{x}_{k+1}) \rangle, \label{eq: eq1_lin1}
\end{align}
and
\begin{align}
- \langle A^\top\bar{y}_{k+1}, P(x'-\bar{x}_{k+1}) \rangle = -\underline p \langle A^\top \bar{y}_{k+1}, x-x_k\rangle - \langle A^\top \bar{y}_{k+1}, P(x_k - \bar{x}_{k+1}) \rangle. \label{eq: eq1_lin2}
\end{align}

Moreover, by using Lemma~\ref{lem: x_lem} in~\eqref{eq: eq_x_three}, we get
\begin{align}
\frac{1}{2} \mathbb{E}_k \left[ \| x' - {x}_{k+1} \|^2_{\tau^{-1}}\right] - \frac{1}{2}\| x'-x_k \|^2_{\tau^{-1}} &= \frac{\underline p}{2} \mathbb{E}_k\left[ \| x - {x}_{k+1}\|^2_{\tau^{-1}P^{-1}}\right] + \frac{1}{2} \mathbb{E}_k\left[ \| {x}_{k+1} - x_k \|^2_{\tau^{-1}}\right] \notag \\
&- \frac{\underline p}{2} \mathbb{E}_k\left[ \| {x}_{k+1} - x_k\|^2_{\tau^{-1}P^{-1}}\right]  - \frac{\underline p}{2} \| x-x_k \|^2_{\tau^{-1}P^{-1}}. \label{eq: eq1_xprime} 
\end{align}

We also note that $\mathbb{E}_k \left[ \| x_{k+1} - x_k \|^2_{\tau^{-1}P^{-1}}\right] = \| \bar{x}_{k+1} -x_k\|^2_{\tau^{-1}}$.

In~\eqref{eq: eq1_lastline}, we collect~\cref{eq: eq_x_three,eq: y_res,eq: gp,eq: eq1_lin1,eq: eq1_lin2,eq: eq1_xprime} and add and subtract $\langle A^\top y, x_{k+1} - x\rangle - \underline p \langle Ax, \bar{y}_{k+1} - y \rangle + \langle A^\top y, x_k - x \rangle - \underline p \langle A^\top y, x_k - x\rangle$ to obtain 
\begin{align}
0 &\geq \mathbb{E}_k [g(x_{k+1})] - g(x_k) + \underline p g(x_k) - \underline p g(x) + \underline{p}h^\ast(\bar{y}_{k+1}) - \underline p h^\ast(y) + \langle A^\top y, x_{k+1} - x\rangle -\underline p \langle Ax, \bar{y}_{k+1} - y\rangle \notag\\
&+\langle A^\top y, x_k - x \rangle -\underline p \langle A^\top y, x_k - x \rangle - \langle A^\top y, x_{k+1} - x\rangle +\underline p \langle Ax, \bar{y}_{k+1} - y\rangle-\langle A^\top y, x_k - x \rangle +\underline p \langle A^\top y, x_k - x \rangle \notag \\
&+ \underline p \langle Ax_k, y-\bar{y}_{k+1}\rangle- \underline p\langle \nabla f (x_k),  x-  x_k \rangle - \langle \nabla f (x_k), P( x_k- \bar{x}_{k+1}) \rangle -\underline p \langle A^\top \bar{y}_{k+1}, x-x_k\rangle \notag\\
&- \langle A^\top \bar{y}_{k+1}, P(x_k - \bar{x}_{k+1}) \rangle+\frac{\underline p}{2} \mathbb{E}_k \left[\| {x}_{k+1} - x \|^2_{\tau^{-1}P^{-1}}\right] - \frac{\underline p}{2} \| x_k - x \|^2_{\tau^{-1}P^{-1}} + \| \bar{x}_{k+1} - x_k \|^2_{\tau^{-1}P-\tau^{-1}\frac{\underline p}{2}} \notag\\
&+\frac{\underline p}{2} \|\bar{y}_{k+1} - y_k \|^2_{\sigma^{-1}} +\frac{\underline p}{2}\mathbb{E}_k \left[\| y_{k+1} - y \|^2_{\sigma^{-1}\pi^{-1}}\right] - \frac{\underline p}{2} \| y_k - y \|^2_{\sigma^{-1}\pi^{-1}} - \langle \bar{y}_{k+1} -y, \underline p \pi^{-1}\theta A P(\bar{x}_{k+1} - x_k) \rangle \notag\\
&-\frac{1}{2} \sum_{i=1}^n \sum_{j=1}^m \underline p p_i \pi_j^{-1} \sigma_j \theta_j^2A_{j, i}^2  (\bar{x}_{k+1}^i - x_k^i)^2.\label{eq: one_lem_main2}
\end{align}

We work on the bilinear terms to get
\begin{multline}
\langle A^\top y, x_k - x \rangle - \langle A^\top y, x_{k+1} - x \rangle - \langle A^\top \bar{y}_{k+1}, P(x_k - \bar{x}_{k+1}) \rangle - \langle y-\bar{y}_{k+1}, \underline p\pi^{-1}\theta A P(x_k - \bar{x}_{k+1}) \rangle \\
=\langle A^\top y, x_k - x_{k+1} \rangle - \langle A^\top \bar{y}_{k+1}, \mathbb{E}_k [x_k - x_{k+1}] \rangle - \langle y - \bar{y}_{k+1}, \underline p \pi^{-1} \theta A \mathbb{E}_k[x_k - x_{k+1}] \rangle \\
=\mathbb{E}_k \left[ \langle A^\top(y-\bar{y}_{k+1}), x_k - x_{k+1} \rangle - \langle y-\bar{y}_{k+1}, \underline p \pi^{-1}\theta A(x_k -x_{k+1}) \right].\label{eq: one_lem_bi1}
\end{multline}
We have that this quantity will be $0$ if
\begin{equation}
\underline p \pi ^{-1} \theta= I \iff \theta_j = \frac{\pi_j}{\underline p},~~~\forall j \in \{1, \ldots, m\},\notag
\end{equation}
which is the requirement in~\eqref{eq: theta_choice}.

Moreover, it holds that
\begin{multline}
\underline p\left[ -\langle A^\top y, x_k - x \rangle + \langle Ax, \bar{y}_{k+1} - y \rangle + \langle Ax_k , y-\bar{y}_{k+1} \rangle - \langle A^\top\bar{y}_{k+1}, x-x_k \rangle  \right] = \\
\underline p\left[ \langle A^\top (y-\bar{y}_{k+1}), x-x_k \rangle + \langle y-\bar{y}_{k+1}, A(x_k - x) \rangle \right] = 0.\label{eq: one_lem_bi2}
\end{multline}

We now use coordinatewise smoothness of $f$
\begin{align}
-\langle \nabla f(x_k), P(x_k - \bar{x}_{k+1}) \rangle &= -\langle \nabla f(x_k), \mathbb{E}_k[x_k - {x}_{k+1}] \rangle \geq \mathbb{E}_k \left[f(x_{k+1}) - f(x_k) - \frac{1}{2} \| x_{k+1} - x_k \|^2_{\beta}\right] \notag\\
&\geq \mathbb{E}_k [f(x_{k+1})]- f(x_k) - \frac{1}{2} \| \bar{x}_{k+1} - x_k \|^2_{\beta P}.\label{eq: one_lem_bi3}
\end{align}

We now define
\begin{align}
\tilde{V}(\bar{z}_{k+1}-z_k) &= \frac{\underline p}{2} \| \bar{y}_{k+1} - y_k \|^2_{\sigma^{-1}} +  \| \bar{x}_{k+1} - x_k\|^2_{\tau^{-1}P} - \frac{\underline p}{2} \| \bar{x}_{k+1}-x_k\|^2_{\tau^{-1}} \notag\\
&-\frac{1}{2} \sum_{i=1}^n \sum_{j=1}^m \underline p p_i \pi_j^{-1} \sigma_j \theta_j^2A_{j, i}^2  (\bar{x}_{k+1}^i - x_k^i)^2 -\frac{1}{2} \| \bar{x}_{k+1} - x_k \|^2_{\beta P} \notag\\
&= \frac{\underline p}{2} \| \bar{y}_{k+1} - y_k \|^2_{\sigma^{-1}} + \frac{\underline p}{2} \| \bar{x}_{k+1} - x_k\|^2_{C(\tau)},\label{eq: one_lem_bi4}
\end{align}
where
\begin{align}
C(\tau)_i &= \frac{2p_i}{\underline p \tau_i} - \frac{1}{\tau_i} - p_i\sum_{j=1}^m \pi_j^{-1}\sigma_j \theta_j^2A_{j, i}^2 - \frac{\beta_ip_i}{ \underline p}.\notag
\end{align}
We lastly note that $-\underline p \langle \nabla f(x_k), x-x_k \rangle \geq -\underline p f(x) + \underline p f(x_k)$.

We use the last estimation,~\eqref{eq: one_lem_bi1},~\eqref{eq: one_lem_bi2},~\eqref{eq: one_lem_bi3}, and~\eqref{eq: one_lem_bi4} in~\eqref{eq: one_lem_main2} to get
\begin{align}
0 &\geq \mathbb{E}_k [f(x_{k+1}) + g(x_{k+1}) + \langle x_{k+1} - x,A^\top y\rangle] - f(x) - g(x) + \underline p \bigg( h^\ast(\bar{y}_{k+1}) - \langle Ax, \bar{y}_{k+1} - y \rangle - h^\ast(y) \bigg) \notag\\
&-(1-\underline p)\Big( f(x_k) + g(x_k) - f(x) - g(x) + \langle x_k - x, A^\top y \rangle  \Big) +\tilde{V}(\bar{z}_{k+1} - z_k)\notag\\
&+\frac{\underline{p}}{2} \mathbb{E}_k \left[\| x_{k+1} - x \|^2_{\tau^{-1}P^{-1}}\right] - \frac{\underline{p}}{2} \| x_{k} - x \|^2_{\tau^{-1}P^{-1}} + \frac{\underline p}{2}\mathbb{E}_k \left[\| y_{k+1} - y \|^2_{\sigma^{-1}\pi^{-1}}\right] - \frac{\underline p}{2} \| y_k - y \|^2_{\sigma^{-1}\pi^{-1}}.\notag
\end{align}

We can use the definitions $V(z) = \frac{\underline p}{2} \| x \|^2_{\tau^{-1}P^{-1}} + \frac{\underline p}{2} \| y \|^2_{\sigma^{-1}\pi^{-1}}$ and
\begin{align}
&D_p(x_{k+1}; z) = f(x_{k+1}) + g({x}_{k+1}) - f(x) - g(x) + \langle A^\top y, {x}_{k+1} -x \rangle, \notag\\
&D_d(\bar{y}_{k+1}; z) = h^\ast(\bar{y}_{k+1}) - h^\ast(y) - \langle Ax, \bar{y}_{k+1} - y \rangle.\notag
\end{align}
to conclude.
\end{proof}

\subsection{Proof for almost sure convergence}
We include the statement and proof of~\Cref{thm: thm1}.

\begin{reptheorem}{thm: thm1}
Let $\theta=\diag(\theta_1, \dots, \theta_m), \pi$ be as in~\Cref{lem: lem1}.
The step sizes satisfy
\begin{equation}
\tau_i < \frac{2p_i - \underline{p}}{\beta_i p_i + \underline p^{-1}p_i\sum_{j=1}^m \pi_j \sigma_j A_{j, i}^2}.\notag
\end{equation}
Let Assumption~\ref{asmp: asmp1} hold.
Then, almost surely, there exist $z_\star = (x_\star, y_\star) \in \mathcal{Z}_\star$ such that $z_k \to z_\star$.\end{reptheorem}
\begin{proof}
Equipped with Lemma~\ref{lem: lem1}, we will follow the standard arguments as in~\citet{fercoq2019coordinate}. We refer to~\citep[Theorem 1]{fercoq2019coordinate} for the finer details of the arguments.

We first invoke the main result of Lemma~\ref{lem: lem1} with $z=z_\star=(x_\star, y_\star)$ where $z_\star \in \mathcal{Z}_\star$:
\begin{align}
\mathbb{E}_k \left[ S_p(x_{k+1}) \right] +\underline p S_d(\bar{y}_{k+1}) + \mathbb{E}_k \left[ V(z_{k+1} - z_\star) \right] \leq (1-\underline p)S_p (x_k) + V(z_k - z_\star) - \tilde{V}(\bar{z}_{k+1} - z_k),\label{eq: as_eq1}
\end{align}
where we have used the definitions
\begin{align}
S_p(x_{k+1}) = D_p(x_{k+1}; z_\star) \geq 0,~~~S_d(\bar{y}_{k+1}) = D_d(\bar{y}_{k+1}; z_\star)\geq 0.\notag
\end{align}
Nonnegativity of these quantities follow from the definition of $z_\star$ as in~\eqref{eq: kkt}.

Moreover, we see that $\tilde{V}(z)$ is nonnegative if it holds that $C(\tau)_i > 0$, and equivalently,
\begin{align}
\tau_i < \frac{2p_i \underline p^{-1} - 1}{\beta_i p_i \underline p^{-1} + p_i  \sum_{j=1}^m \pi_j^{-1}\sigma_j\theta_j^2 A_{j, i}^2} \iff \tau_i < \frac{2p_i - \underline{p}}{\beta_i p_i + \underline p^{-1}p_i\sum_{j=1}^m \pi_j \sigma_j A_{j, i}^2},\label{eq: ss_cond_pf}
\end{align}

where we used $\theta_j = \frac{\pi_j}{\underline p}$ and~\eqref{eq: ss_cond_pf} is exactly our step size requirement.

It is also immedate that $V(z)$ is nonnegative.

Then, we can write~\eqref{eq: as_eq1} as
\begin{equation}
\mathbb{E}_k\left[S_p(x_{k+1}) + V(z_{k+1} - z_\star) \right] \leq S_p(x_k)+V(z_k - z_\star) - \underline p \left( S_p(x_k) + S_d(\bar{y}_{k+1}) \right).\label{eq: main_rec_ineq}
\end{equation}
We use Robbins-Siegmund lemma on this inequality and nonnegativity of $S_p(x_k), S_d(\bar{y}_k), V(z), \tilde{V}(z)$ to conclude that $V(z_k - z_\star)$ converges almost surely, $\sum_{k} S_p(x_k) + S_d(\bar{y}_{k+1}) < \infty$, therefore $S_p(x_k)$ and $S_d(\bar{y}_k)$ converges to $0$ almost surely.
Then, we argue as in~\cite{fercoq2019coordinate,iutzeler2013asynchronous},~\citep[Proposition 2.3]{combettes2015stochastic}, to get $\forall w\in \Omega$ with $\mathbb{P}(\Omega)=1$, and $\forall z_\star \in \mathcal{Z}_\star$, $V(z_k(\omega)-z_\star)$ converges.

We now take full expectation of~\eqref{eq: as_eq1}, use the nonnegativity of $S_p(x_k)$, $S_d(\bar y_k)$, and sum the inequality to obtain for any $K$,
\begin{align}
\sum_{k=0}^K \mathbb{E} \left[ \tilde{V}(\bar{z}_{k+1} - z_k) \right] \leq S_p(x_0) + V(z_0 - z_\star) := \Delta_0 < \infty.\label{eq: sum_of_subseq}
\end{align}
It then follows by Fubini-Tonelli theorem that $\mathbb{E} \left[ \sum_{k=0}^{\infty} \tilde{V}(\bar{z}_{k+1} - z_k) \right] < \infty$. Since $\tilde{V}(\bar{z}_{k+1} - z_k)$ is nonnegative, by the step size rules, we have that $\tilde{V}(\bar{z}_{k+1} - z_k)$ converges almost surely to $0$. Then, since $\tilde{V}(z)$ is squared norm with the step size rules, we have that $\bar{z}_{k+1} - z_k$ converges to $0$ almost surely. 

We define $T\colon\mathcal{X}\times\mathcal{Y} \to \mathcal{X}\times\mathcal{Y}$ such that $\bar{z} = T(z)$ and
\begin{align}
&\bar{y} = \prox_{\sigma, h^\ast}\left( y + \sigma Ax \right) \notag\\
&\bar{x} = \prox_{\tau, g}\left( x - \tau (\nabla f(x) + A^\top \bar{y}) \right).\notag
\end{align}
Thus, we see that $\bar{z}_{k+1} = T(z_k)$.

We use the definition of proximal operator in the definition of $T$ and compare with the definition of a saddle point in~\eqref{eq: kkt} to conclude that the fixed points of $T$ correspond to the set of saddle points $\mathcal{Z}_\star$.

We recall that for any trajectory $\omega$ selected from a set of probability $1$, $z_k(\omega)$ is a bounded sequence, as we already proved that $V(z_k-z_\star)$ converges almost surely, and $V(z)$ is squared norm.
As $z_k(\omega)$ is bounded, it converges on at least one subsequence. 
We denote by $\check{z}$ the cluster point of this subsequence.
By the fact that $\bar{z}_{k+1}(\omega) = T(z_{k}(\omega))$ and $\bar{z}_{k+1}(\omega) - z_k(\omega) \to 0$, we conclude that $T(z_k(\omega)) - z_k(\omega) \to 0$.
As $T$ is continuous, by the nonexpansiveness of proximal operator, we get $T(\check{z})-\check{z} = 0$.
Since $\check{z}$ is a fixed point of $T$ we conclude that $\check{z}\in\mathcal{Z}_\star$.

Since we know that for any $z_\star\in\mathcal{Z}$, $V(z_k - z_\star)$ converges almost surely, and we have proved V($z_k(\omega)-\check{z})$ converges to $0$ at least on one subsequence, and as $V$ is squared norm, we conclude that the sequence $z_k$ converges to a point in $\mathcal{Z}_\star$, almost surely.
\end{proof}
\subsection{Proof for linear convergence}
In this section, we include the statement and proof of~\Cref{thm: thm2}.
First, we need a lemma to characterize the specific choice of projection onto the solution set, given in~\eqref{eq: def_zkstar}.
\begin{lemma}\label{eq: lemma_linconv}
Let us denote
\begin{align*}
z^{\star, b} &= \arg\min_{u \in \mathcal{Z}_\star} D_p(x; u) + V(z - u).\\
z^{\star, e} &= \arg\min_{u \in \mathcal{Z}_\star}V(z - u).
\end{align*}
We have that $V(z^{\star, b} - z^{\star, e}) \leq c V(z - z^{\star, e})$, where $c=C_{2,V}\sqrt{\frac{\|A\|}{2}}$ and $C_{2,V}$ is such that for any $z$, $\| z\| \leq C_{2,V} V(z)^{1/2}$.
\end{lemma}
\begin{proof}
Let us first remark that since $u =(u_x, u_y)\in \mathcal Z_*$ is a saddle point of the Lagrangian $L(x,y) =f(x) + g(x) + \langle Ax, y\rangle -h^*(y)$, we have that $D_p(x, u) = L(x, u_y) - L(u_x, u_y)$ is independent of $u_x$ and affine in $u_y$. In particular, there exists a constant, for any primal solution $x^\star$, $C(x, x^\star) = f(x) + g(x) - f(x^\star) - g(x^\star)$ such that
$D_p(x, u) = C(x, x^\star) + \langle A (x-x^\star), u_y\rangle$, for all $u \in \mathcal Z_*$. We have also used here the fact that for two different primal solutions $x_1^\star, x_2^\star$, it follows that $L(x_1^\star, u_y) = L(x_2^\star, u_y)$, where $u_y$ is a dual solution.

Let us now consider primal-dual point $z$.
By prox inequality,
\begin{align*}
&C(x, x^\star) + \langle A (x-x^\star), y^{*, b}\rangle + V(z - z^{*, b}) \leq C(x, x^\star) + \langle A (x-x^\star), y^{*, e}\rangle + V(z - z^{*, e}) - V(z^{*, b} - z^{*, e}) \\
&V(z^{*, e} - z) \leq V(z-z^{*, b}) - V(z^{*, e}-z^{*, b}),
\end{align*}
where $x^\star$ is any primal solution.

Summing both equalities and rearranging yields
\begin{align*}
2 V(z^{*, e} - z^{*, b}) &\leq \langle A(x-x^\star), y^{*, e} - y^{*, b} \rangle
\end{align*}
Since the inequality holds for any $x^\star$, we can plug in $x^{*, e}$ and use Cauchy-Schwarz inequality to get
\begin{equation*}
2 V(z^{*, e} - z^{*, b}) \leq C_{2,V}^2 \| A\| V(z-z^{*, e})^{1/2}V(z^{*, b}-z^{*,e})^{1/2}.
\end{equation*}
Lastly $C_{2,V} = \sqrt{\frac{2}{\underline p\min_i\{\tau_i^{-1}p_i^{-1}\}}}$.
\end{proof}

\begin{reptheorem}{thm: thm2}
Let~\Cref{asmp: asmp1} and~\ref{asmp: asmp2} hold.
Let $\theta$ and the step sizes $\tau, \sigma$ be chosen according to~\eqref{eq: theta_choice} and~\eqref{eq: ss_choice}, respectively.
Moreover, $z_k^\star = (x_k^\star, y_k^\star)$ is as defined in~\eqref{eq: def_zkstar}.
Then, for $z_k=(x_k, y_k)$ generated by Algorithm~\ref{alg:stripd}, it follows that
\begin{equation}
\mathbb{E} \left[\frac{\underline p}{2} \| x_k - x_k^\star \|^2_{\tau^{-1} P^{-1}} + \frac{\underline p}{2} \| y_k - y_k^\star \|^2_{\sigma^{-1}\pi^{-1}} \right] \\ 
\leq (1-\rho)^k \Delta_0,\notag
\end{equation}
where $\Delta_0 = D_p(x_0; z_0^\star) + V(z_0 - z_0^\star)$,
$\rho = \min\left( \underline p, \frac{C_{2, \tilde V}}{C_{V, 2} ((2+2c)+(1+c)(\eta \| H-M\|+\beta))^2} \right)$, $\bar \beta$ is the Lipschitz constant of $f$, $C_{2, \tilde V} = \frac{\underline{p}}{2} \min \left\{ \min_i C(\tau)_i, \min_j \sigma_j^{-1} \right\}$, 
$C_{V, 2}=\frac{1}{2} \max\left\{ \max_i \frac{1}{\tau_i}, \max_j \frac{1}{\sigma_j} \right\}$, $C_{2,V} = \sqrt{\frac{2}{\underline p\min_i\{\tau_i^{-1}p_i^{-1}\}}}$, $c=C_{2,V}\sqrt{\| A\|/2}$, and
\begin{equation}
H = \begin{bmatrix} \tau^{-1} & A^\top \\ 0 & \sigma^{-1} \end{bmatrix}, ~~~~ M = \begin{bmatrix} 0 & A^\top \\ -A & 0 \end{bmatrix}.\notag
\end{equation}
\end{reptheorem}
\begin{proof}
We note the definitions, as in~\cite{latafat2019new},
\begin{align}
&A\colon (x,y) \mapsto (\partial g(x) + \partial h^\ast(y)) \notag\\
&M\colon (x, y) \mapsto (A^\top y, -Ax) \notag\\
&C\colon (x, y) \mapsto (\nabla f(x), 0) \notag\\
&H\colon (x, y) \mapsto (\tau^{-1} x + A^\top y, \sigma^{-1}y).\notag
\end{align}
Under this notations, KKT operator defined in~\eqref{eq: kkt} can be written as
\begin{equation}\label{eq: f_op_def}
F = A + M + C.
\end{equation}
Moreover, $\bar{z}_{k+1} = (H+A)^{-1}(H-M-C)z_k$, which in fact (without the term $\nabla f(x_k)$) is the well-known Arrow-Hurwicz operator~\cite{arrow1958studies} which is one of the first primal-dual algorithms.

Moreover, we will use the following inequalities regarding squared norms $V$ and $\tilde{V}$
\begin{align}
&\tilde{V}(z) \geq C_{2, \tilde V} \left( \|x\|^2 + \|y\|^2 \right) :=  \frac{\underline{p}}{2} \min \left\{ \min_i C(\tau)_i, \min_j \sigma_j^{-1} \right\}  \left( \|x\|^2 + \|y\|^2 \right), \label{eq: defc2}\\
&V(z) \leq C_{V, 2} \left( \|x\|^2 + \|y\|^2 \right) := \frac{1}{2} \max\left\{ \max_i \frac{1}{\tau_i}, \max_j \frac{1}{\sigma_j} \right\} \left( \|x\|^2 + \|y\|^2 \right).\label{eq: defc3}
\end{align} 
We recall the definition of $z_k^\star$
\begin{equation}
z_k^\star = \arg\min_{u \in \mathcal{Z}_\star} D_p(x_{k}; u) + V(z_k - u).\notag
\end{equation}
We now argue that $z_k^\star$ is well-defined under our assumptions.
First, we know that the solution set is convex and closed. 
Second, $D_p(x_{k+1}; u) \geq 0$ for all $u\in\mathcal{Z}_\star$ and it is also lower semicontinuous under~\Cref{asmp: asmp1}. 
Third, we remark that $V(z_k-u)$ is squared norm, thus coercive, therefore the sum is coercive and lower semicontinuous over $\mathcal{Z}_\star$.
Hence, $z_k^\star$ exists.

We use the result of Lemma~\ref{lem: lem1} with $z=z_k^\star$ and use the fact that $D_d(\bar{y}_{k+1}; z_k^\star) \geq 0$
\begin{align*}
\mathbb E_k\Big[D_p(x_{k+1}; z_k^\star)  + V(z_{k+1} - z_k^\star)] \leq (1 - \underline p)D_p(x_k; z_k^\star) + V(z_k - z_k^\star) - \tilde{V}(\bar{z}_{k+1} - z_k).
\end{align*}
We use the definition of $z_{k+1}^\star$ to deduce
\begin{align}
\mathbb E_k\Big[D_p(x_{k+1}; z_{k+1}^\star)  + V(z_{k+1} - z_{k+1}^\star)] \leq (1 - \underline p)D_p(x_k; z_k^\star) + V(z_k - z_k^\star) - \tilde{V}(\bar{z}_{k+1} - z_k).\label{eq: lin0}
\end{align}
In addition to Bregman projections $z_{k}^\star$ and $\bar z_{k+1}^{\star}$, we introduce the definitions for Euclidean projections
\begin{align}
&z_k^{\star, e} = \arg\min_{u\in\mathcal{Z}_\star} V(z_k - u), \notag\\
&\bar z_{k+1}^{\star, e} = \arg\min_{u\in\mathcal{Z}_\star} V(\bar z_{k+1} - u).\notag
\end{align}
Now, we use triangle inequalities and~\Cref{eq: lemma_linconv} to get
\begin{align}
V(z_k - z_k^{\star})^{1/2} &\leq
V(z_k - z_k^{\star, e})^{1/2} + V(z_k^{\star, e} - z_k^{\star})^{1/2} \leq (1+c) V(z_k - z_k^{\star, e})^{1/2} \notag\\
&\leq (1+ c)(V(z_k - \bar z_{k+1})^{1/2} + V(\bar z_{k+1} - \bar z_{k+1}^{\star, e})^{1/2}+ V(z_k^{\star, e} - \bar z_{k+1}^{\star, e})^{1/2})\notag
\end{align}
We use nonexpansiveness with metric $V$ to obtain
\begin{align}
V(z_k - z_k^{\star})^{1/2} \leq (2+2c)V(z_k - \bar z_{k+1})^{1/2} + (1+c) V(\bar z_{k+1}^{\star, e} - \bar z_{k+1})^{1/2}.\notag
\end{align}
We use the definition of $\bar z_{k+1}^{\star, e}$ to get $V(\bar z_{k+1} - \bar z_{k+1}^{*, e}) \leq V(\bar z_{k+1} - P_{\mathcal{Z}_\star}(\bar z_{k+1}))$, and then we use the relation between $V$ and Euclidean norm.
\begin{align}
V(z_k - z_k^{\star}) &\leq (2+2c)^2C_{V,2} \|z_k - \bar z_{k+1} \|^2 + (1+c)^2 C_{V,2} \| P_{\mathcal{Z}_\star}(\bar z_{k+1}) - \bar z_{k+1}\|^2 \notag \\
&= (2+2c)^2C_{V,2} \|z_k - \bar z_{k+1} \|^2 + (1+c)^2 C_{V,2} \dist(\bar z_{k+1}, \mathcal{Z}_\star)^2.\label{eq: lin1}
\end{align}

We now use metric subregularity of $F$ for $0$, and the assumption that $\bar{z}_{k+1} \in \mathcal{N}(z_\star), \forall z_\star$,
\begin{align*}
\dist(\bar z_{k+1}, \mathcal{Z}_\star) \leq \eta \dist(0, F(\bar{z}_{k+1})).
\end{align*}
We now use that $(H-M-C)(z_k - \bar{z}_{k+1}) \in F(\bar{z}_{k+1})$, which can be obtained by using~\eqref{eq: f_op_def} and $\bar{z}_{k+1} = (H+A)^{-1}(H-M-C)z_k$. 
Therefore
\begin{align*}
\dist(\bar z_{k+1}, \mathcal{Z}_\star) \leq \eta \| (H-M-C)(z_k-\bar{z}_{k+1}) \| \leq \eta (\|H-M\|+\beta) \| z_k-\bar{z}_{k+1} \|,
\end{align*}
where $\beta$ is the global Lipschitz constant of $f$.

We plug this inequality into~\eqref{eq: lin1} to obtain
\begin{align*}
V(z_k - z_k^\star) \leq C_{V, 2} ((2+2c)+(1+c)(\eta \| H-M\|+\beta))^2 \| \bar{z}_{k+1} - z_k \|^2.
\end{align*}
Moreover, since $\tilde{V}(\bar{z}_{k+1} - z_k)$ is a squared norm, under the step size condition, it follows that $\tilde{V}(\bar{z}_{k+1} - z_k) \geq C_{2,\tilde V} \| \bar{z}_{k+1} - z_k \|^2$, as in~\eqref{eq: defc2}, therefore,
\begin{equation}
V(z_k - z_k^\star) \leq \frac{C_{V, 2} ((2+2c)+(1+c)(\eta \| H-M\|+\beta))^2}{C_{2, \tilde V}} \tilde{V}(\bar{z}_{k+1}-z_k).\notag
\end{equation}
We use this inequality in~\eqref{eq: lin0} to obtain
\begin{multline*}
\mathbb E_k\Big[D_p(x_{k+1}; z_{k+1}^\star)  + V(z_{k+1} - z_{k+1}^\star)] \leq (1- \underline p)D_p(x_k; z_{k}^\star) \\
+ \left( 1- \frac{C_{2, \tilde V}}{C_{V, 2} ((2+2c)+(1+c)(\eta \| H-M\|+\beta))^2} \right) V(z_k - z_k^\star),
\end{multline*}
where the constants $C_{2, \tilde V}$, $C_{V, 2}$ are as defined in~\eqref{eq: defc2},~\eqref{eq: defc3}.

We take full expectation and define 
$\rho = \min\left( \underline p, \frac{C_{2, \tilde V}}{C_{V, 2} ((2+2c)+(1+c)(\eta \| H-M\|+\beta))^2} \right)$.
Then, we have that
\begin{equation*}
\mathbb{E} \left[ D_p(x_{k+1}; z_{k+1}^\star)  + V(z_{k+1} - z_{k+1}^\star) \right] \leq (1-\rho) \mathbb{E}\left[ D_p(x_{k}; z_{k}^\star)  + V(z_{k} - z_{k}^\star) \right].
\end{equation*}
We have that $0 < \rho < 1$, as metric subregularity constant $\eta > 0$. Hence, linear convergence of $D_p(x_k, z_k^\star)$ and $V(z_k -z_k^\star)$ follows.
We obtain the final result after using the definition of $V$, and the fact that $D_p(x_{k+1}; z_{k+1}^\star) \geq 0$.
\end{proof}

\subsection{Ergodic convergence rates}
We introduce the following lemma, which establishes the properties of the sequence $\breve y_k$.

\begin{lemma}\label{lem: breve}
We define the iterate, $\breve{y}_1 = y_1 = \bar{y}_1$, and
\begin{align}
&\breve{y}_{k+1}^j = \bar{y}_{k+1}^j, \forall j \in J(i_{k+1}) \notag\\
&\breve{y}_{k+1}^j = \breve{y}_k^j, \forall j\not\in J(i_{k+1}),\notag
\end{align}
where computing $\breve{y}_{k+1}$ requires the same number of operations as computing $y_{k+1}$ every iteration, and $\breve{y}_k$ is $\mathcal{F}_k$-measurable.

Moreover, if it holds for a function $l$ that $l(y) = \sum_{j=1}^m l_j(y_j)$ and
\begin{align}
l_\gamma(y) = \sum_{j=1}^m \gamma_j l_j (y_j),\notag
\end{align}
we have the following for the sequence $\breve{y}_k$, and $\mathcal{F}_k$-measurable $Y$:
\begin{align}
&\mathbb{E}_k\left[ \breve{y}_{k+1}^j - \breve{y}_k^j \right] = \pi_j \left( \bar{y}_{k+1}^j - \breve{y}_k^j \right), \forall j \notag\\
&\mathbb{E}_k\left[ \|\breve{y}_{k+1} - Y \|^2_{\gamma} \right] = \| \bar{y}_{k+1} - Y \|^2_{\gamma \pi} - \|\breve{y}_k - Y \|^2_{\gamma\pi} + \| \breve{y}_k - Y \|^2_{\gamma} \notag\\
&\mathbb{E}_k \left[ l(\breve{y}_{k+1})\right] = l_\pi(\bar{y}_{k+1}) - l_{\pi}(\breve{y}_k) + l(\breve{y}_k) \notag\\
&\mathbb{E}_{k} \left[\| \breve{y}_{k+1} - y_{k+1} \|^2_\gamma\right] =  \| \bar{x}_{k+1} - x_k \|^2_{B(\gamma)} + \| \breve{y}_k - y_k \|^2_{\gamma} - \| \breve{y}_k - y_k \|^2_{\gamma\pi} \notag\\
&\sum_{k=1}^K \mathbb{E} \left[ \| \breve{y}_{k+1} - \breve{y}_k \|^2_\gamma \right] \leq 2\sum_{k=1}^K \mathbb{E}\left[ \| \bar{y}_{k+1} - y_k \|^2_{\gamma\pi}\right] +2 \sum_{k=1}^K \mathbb{E} \left[\| \bar{x}_{k+1} - x_k \|^2_{B(\gamma)}\right] \notag\\
&\mathbb{E}\left[\| \breve{y}_k\|^2_{\gamma\pi}\right]\leq 2\mathbb{E} \left[\| y_k \|^2_{\gamma \pi}\right] + 2\sum_{k=1}^K \mathbb{E} \left[\| \bar{x}_{k+1} - x_k\|^2_{B(\gamma)}\right],\notag
\end{align}
where 
\begin{align}
B(\gamma)_i = p_i \sum_{j=1}^m  \theta_j^2 \gamma_j \sigma_j^2 A_{j, i}^2,  \text{ and } \pi_j = \sum_{i\in I(j)} p_i.\notag
\end{align}
\end{lemma}
\begin{proof}
We first use the definition of $\breve{y}_k$ to get the first equality.
\begin{align*}
\mathbb E_k[\breve{y}_{k+1}^j] &= \mathbb{E}_k\left[ \mathds{1}_{j \in J(i_{k+1})} \bar{y}_{k+1}^j + \mathds{1}_{j\not\in J(i_{k+1})} \breve{y}_k^j \right] = \sum_{i=1}^n p_i \Big [\mathds{1}_{j \in J(i)} \big( \bar {y}_{k+1}^j \big) + \mathds{1}_{j \not \in J(i)} \breve{y}_k^j\Big]\\
&=\sum_{i\in I(j)} p_i \bar{y}_{k+1}^j + \sum_{i\not\in I(j)} p_i \breve{y}_k^j = \sum_{i\in I(j)} p_i \bar{y}_{k+1}^j + \sum_{i=1}^n p_i \breve{y}_k^j - \sum_{i\in I(j)} p_i \breve{y}_k^j \\
&= \breve{y}_k^j + \pi_j \left( \bar{y}_{k+1}^j - \breve{y}_k^j \right).
\end{align*}
For the second inequality, we estimate as
\begin{align}
\mathbb{E}_k \left[ \| \breve{y}_{k+1} - Y \|^2_{\gamma}\right] &= \mathbb{E}_k \left[ \sum_{j\in J(i_{k+1})} \gamma_j(\bar{y}_{k+1}^j - Y^j)^2 + \sum_{j\not\in J(i_{k+1})} \gamma_j(\breve{y}_k^j - Y^j)^2 \right]\notag \\
&=\sum_{i=1}^n p_i \left[\sum_{j\in J(i)} \gamma_j(\bar{y}_{k+1}^j - Y^j)^2 + \sum_{j\not\in J(i)} \gamma_j(\breve{y}_k^j - Y^j)^2 \right] \notag\\
&=\sum_{j=1}^m \sum_{i \in I(j)} p_i \gamma_j(\bar{y}_{k+1}^j - Y^j)^2 + \sum_{j=1}^m\sum_{i=1}^n p_i \gamma_j(\breve{y}_k^j - Y^j)^2 - \sum_{j=1}^m \sum_{i\in I(j)} p_i \gamma_j(\breve{y}_k^j - Y^j)^2 \notag \\
&= \| \bar{y}_{k+1} -Y \|^2_{\gamma \pi} + \| \breve{y}_k - Y \|^2_{\gamma} - \| \breve{y}_k - Y \|^2_{\gamma \pi}.\notag
\end{align}
We derive the third equality using similar estimations
\begin{align}
\mathbb{E}_k \left[ l(\breve{y}_{k+1}) \right] &= \mathbb{E}_k \left[ \sum_{j\in J(i_{k+1})} l_j(\bar{y}_{k+1}^j) + \sum_{j\not\in J(i_{k+1})} l_j(\breve{y}_k^j) \right]\notag \\
&=\sum_{i=1}^n p_i \left[ \sum_{j\in J(i)} l_j(\bar{y}_{k+1}^j) + \sum_{j\not\in J(i)} l_j(\breve{y}_k^j) \right] \notag\\
&=\sum_{j=1}^m \pi_j l_j(\bar{y}_{k+1}^j) +\sum_{j=1}^m \sum_{i=1}^n p_i l_j(\breve{y}_k^j) - \sum_{j=1}^m \pi_j l_j(\breve{y}_k^j) \notag\\
&= l_{\pi}(\bar{y}_{k+1}) - l_\pi(\breve{y}_k) + l(\breve{y}_k).\notag
\end{align}

For the fourth inequality, we use the definitions of both $\breve{y}_{k+1}$ and $y_{k+1}$ (see Algorithm~\ref{alg:stripd}),
\begin{align}
\mathbb{E}_k \left[\| \breve{y}_{k+1} - y_{k+1} \|^2_{\gamma}\right] &= \mathbb{E}_k \left[ \sum_{j\in J(i_{k+1})} \gamma_j \left( \bar{y}_{k+1}^j - (\bar{y}_{k+1}^j + \sigma_j \theta_j A_{j, i_{k+1}} ({x}_{k+1}^{i_{k+1}} - x_k^{i_{k+1}}))  \right)^2 +\sum_{j\not\in J(i_{k+1})}\gamma_j \left( \breve{y}_k^j - y_k^j \right)^2 \right] \notag\\
&=\sum_{i=1}^n p_i \left[ \sum_{j\in J(i)}\gamma_j \left( \sigma_j \theta_j A_{j, i}  (\bar{x}_{k+1}^i - x_k^i) \right)^2 + \sum_{j\not\in J(i)} \gamma_j \left( \breve{y}_k^j-y_k^j \right)^2 \right] \notag\\
&= \sum_{i=1}^n p_i \sum_{j\in J(i)} \gamma_j \sigma_j^2 A_{j, i}^2\theta_i^2 (\bar{x}_{k+1}^i - x_k^i )^2 + \sum_{i=1}^n p_i \sum_{j=1}^m \gamma_j (\breve{y}_k^j - y_k^j)^2 - \sum_{i=1}^n \sum_{j\in J(i)} p_i \gamma_j (\breve{y}_k^j - y_k^j)^2 \notag\\
&= \| \bar{x}_{k+1} - x_k \|^2_{B(\gamma)} + \| \breve{y}_k - y_k \|^2_{\gamma} - \| \breve{y}_k - y_k \|^2_{\gamma\pi},\label{eq: lem4_pf}
\end{align}
where for the second equality, we used the fact that $x_{k+1}$ is different from $x_k$ only on the coordinate $i_{k+1}$, which gives
\begin{equation}
(A(x_{k+1}-x_k))_j = (A(x_{k+1}^{i_{k+1}}-x_k^{i_{k+1}})e_{i_{k+1}})_j = A_{j, i_{k+1}}(\bar{x}_{k+1}^{i_{k+1}}-x_k^{i_{k+1}}),\notag
\end{equation} 
and for the last equality, as $A_{j, i} = 0, \forall j \not \in J(i)$,
\begin{align}
B(\gamma)_i = p_i \sum_{j=1}^m \gamma_j \sigma_j^2 \theta_j^2 A_{j, i}^2.\notag
\end{align}
For the fifth inequality, we first take full expectation and then sum the inequality~\eqref{eq: lem4_pf}
\begin{align}
\sum_{k=1}^K \mathbb{E} \left[\| \breve{y}_k - y_k \|^2_{\gamma \pi} \right] \leq \sum_{k=1}^K \mathbb{E} \left[ \| \bar{x}_{k+1} - x_k \|^2_{B(\gamma)}\right] + \| y_1 - \breve{y}_1 \|^2_{\gamma}.\label{eq: sum_brev_norm}
\end{align}
Then, we write from the second inequality that
\begin{align}
\mathbb{E}_k\left[ \| \breve{y}_{k+1} - \breve{y}_k \|^2_{\gamma}\right] = \| \bar{y}_{k+1} - \breve{y}_k \|^2_{\gamma\pi} \leq 2 \| \bar{y}_{k+1} - y_k \|^2_{\gamma\pi} + 2 \| \breve{y}_k - y_k\|^2_{\gamma\pi}.\notag
\end{align}
We take full expectation and sum to get
\begin{align}
\sum_{k=1}^K \mathbb{E} \left[\| \breve{y}_{k+1} - \breve{y}_k\|^2_{\gamma}\right] &\leq 2\sum_{k=1}^K \mathbb{E}\left[\| \bar{y}_{k+1} - y_k \|^2_{\gamma\pi}\right] + 2\sum_{k=1}^K \mathbb{E} \left[\|\breve{y}_k - y_k \|^2_{\gamma\pi}\right] \notag\\
&\leq 2\sum_{k=1}^K \mathbb{E} \left[\| \bar{y}_{k+1} - y_k \|^2_{\gamma\pi} \right]+ 2 \sum_{k=1}^K \mathbb{E} \left[\| \bar{x}_{k+1} - x_k \|^2_{B(\gamma)}\right],\notag
\end{align}
where we have used~\eqref{eq: sum_brev_norm} and the fact that $\breve{y}_1 = y_1$.

For the last inequality, we note that
\begin{align}
\E{\| \breve{y}_k \|^2_{\gamma \pi}} \leq 2\E{\| y_k \|^2_{\gamma\pi}} + 2\E{\| \breve{y}_k - y_k \|^2_{\gamma\pi}},\notag
\end{align}
and we use~\eqref{eq: sum_brev_norm} with $\breve{y}_1=y_1$, for the second term.\qedhere
\end{proof}
We continue with the restatement and the proof of~\Cref{th:erg}.
This length of the proof is due to the complications discussed earlier.
First, as discussed in~\cite{alacaoglu2019convergence}, the order of expectation and supremum requires a special proof which delays taking expectations of the estimates (which prohibits simplifications and results in long expressions).~\Cref{lem: lem_erg2} thus can be seen as a version of~\Cref{lem: lem1} with expectation not taken.

However, this is not enough due to the special structure of our new method suited for sparse settings.
In particular, we have to manipulate the terms with dual variable carefully, as we cannot average $\bar y_k$ (see~\Cref{lem: lem1}).
Therefore, the treatment with $\breve y_k$, which is characterized in~\Cref{lem: breve} and~\Cref{lem: lem_erg3}, is the intricate part of our proof.
\begin{lemma}\label{lem: lem_erg2}
Let~\Cref{asmp: asmp1} hold.
Given the definitions of $D_p$ and $D_d$ given from~\Cref{lem: lem1},
it follows that
\begin{align}
0 &\geq D_p(x_{k+1}; z) + \underline p D_d(\bar{y}_{k+1}; z) - (1-\underline p) D_p(x_k; z)+\tilde{V}(\bar{z}_{k+1} - z_k) + \frac{1}{2} \| \bar{x}_{k+1} - x_k \|^2_{\beta P} + S_1 + S_2 \notag \\
&+\frac{\underline p}{2} \| x - x_{k+1} \|^2_{\tau^{-1}P^{-1}} - \frac{\underline p}{2} \| x-x_k \|^2_{\tau^{-1}P^{-1}}+ \frac{\underline p}{2} \| y-y_{k+1} \|^2_{\sigma^{-1}\pi^{-1}} - \frac{\underline p}{2}\| y - y_k \|^2_{\sigma^{-1}\pi^{-1}},\notag
\end{align}
where
\begin{align}
S_1 &= g_P(\bar{x}_{k+1}) - g_P(x_k) - \left(g(x_{k+1}) - g(x_k) \right) \notag \\
&-f(x_{k+1}) + f(x_k) -\underline p f(x_k)+\underline p f(x) - \underline p \langle \nabla f(x_k), x-x_k \rangle - \langle \nabla f(x_k), P(x_k -\bar{x}_{k+1}) \rangle \notag \\
&+ \frac{\underline p}{2} \|\bar{x}_{k+1}\|^2_{\tau^{-1}} - \frac{\underline p}{2} \| x_k \|^2_{\tau^{-1}} -\left(\frac{\underline p}{2} \|x_{k+1}\|^2_{\tau^{-1}P^{-1}} -\frac{\underline p}{2} \| x_k\|^2_{\tau^{-1}P^{-1}} \right) \notag \\
&+\frac{\underline p}{2} \| \bar{y}_{k+1} \|^2_{\sigma^{-1}} - \frac{\underline p}{2} \| y_k \|^2_{\sigma^{-1}} + \underline p\langle \bar{y}_{k+1},  \pi^{-1}\theta AP(\bar{x}_{k+1} - x_k) \rangle + \frac{\underline p}{2} \sum_{i=1}^n p_i \sum_{j=1}^m  \pi_j^{-1} \sigma_j \theta_j^2 A_{j, i}^2 (\bar{x}_{k+1}^i - x_k^i)^2 \notag\\
&-\frac{\underline p}{2} \bigg( \| y_{k+1}\|^2_{\sigma^{-1}\pi^{-1}} - \| y_k \|^2_{\sigma^{-1}\pi^{-1}} \bigg),\notag \\
S_2 &= \langle y, A(x_k - x_{k+1}) - AP(x_k - \bar{x}_{k+1}) \rangle +\underline p \langle x, x_k - \bar{x}_{k+1} - P^{-1}(x_k-x_{k+1}) \rangle_{\tau^{-1}}\notag\\
&-\underline p \langle y, \pi^{-1}\sigma^{-1}(y_k-y_{k+1}) - \sigma^{-1}(y_k -\bar{y}_{k+1}) + \pi^{-1}\theta AP(\bar{x}_{k+1}-x_k) \rangle. \notag
\end{align}
\end{lemma}
\begin{proof}
We now follow the proof of Lemma 1 without taking conditional expectations, similar to ergodic convergence rate proof of~\citet{alacaoglu2019convergence}.

First, we have, from~\eqref{eq: eq1_lastline}
\begin{align}
g_P(x') + \underline{p} h^\ast(y) &\geq \underbrace{g_P(\bar{x}_{k+1}) + \underline{p}h^\ast(\bar{y}_{k+1}) - \langle A^\top\bar{y}_{k+1}, P(x'-\bar{x}_{k+1}) \rangle + \underline p \langle Ax_k, y-\bar{y}_{k+1}\rangle}_{T_1} \notag \\
&\underbrace{- \langle \nabla f (x_k), P(x'-\bar{x}_{k+1}) \rangle}_{T_2} + \underbrace{\frac{1}{2} \Big( \| x_k - \bar{x}_{k+1}\|^2_{\tau^{-1}P} + \| x'- \bar{x}_{k+1}\|^2_{\tau^{-1}P} - \| x'- x_k \|^2_{\tau^{-1}P} \Big)}_{T_3} \notag\\
&+\underbrace{\frac{\underline p}{2} \left( \| y_k - \bar{y}_{k+1}\|^2_{\sigma^{-1}} + \| y- \bar{y}_{k+1}\|^2_{\sigma^{-1}} - \| y- y_k \|^2_{\sigma^{-1}} \right)}_{T_4}.\label{eq: erg_lem1_eq1}
\end{align}
We start with $T_1$ and add and subtract $\langle A^\top y, x_{k+1} - x\rangle - \underline p \langle Ax, \bar{y}_{k+1} - y \rangle + \langle A^\top y, x_k - x \rangle - \underline p \langle A^\top y, x_k - x\rangle + g(x_{k+1}) + g_P(x_k) - g(x_k) - \underline p \langle y- \bar{y}_{k+1} , \pi^{-1}\theta  AP( x_k - \bar{x}_{k+1})\rangle$
\begin{align}
T_1 &= g(x_{k+1}) - g(x_k) + g_P(x_k) + g_P(\bar{x}_{k+1}) - g(x_{k+1}) + g(x_k) - g_P(x_k) + \underline p h^\ast(\bar{y}_{k+1})  \notag \\
&- \langle A^\top\bar{y}_{k+1}, P(x'-\bar{x}_{k+1}) \rangle + \underline p \langle Ax_k, y-\bar{y}_{k+1}\rangle +\underline p \langle y- \bar{y}_{k+1}, \pi^{-1}\theta AP(x_k - \bar{x}_{k+1}) \rangle \notag \\
&+\langle A^\top y, x_{k+1} - x\rangle - \underline p \langle Ax, \bar{y}_{k+1} - y \rangle + \langle A^\top y, x_k - x \rangle - \underline p \langle A^\top y, x_k - x\rangle\notag \\
&-\langle A^\top y, x_{k+1} - x\rangle + \underline p \langle Ax, \bar{y}_{k+1} + y \rangle - \langle A^\top y, x_k - x \rangle + \underline p \langle A^\top y, x_k - x\rangle \notag \\
&-\underline p \langle y- \bar{y}_{k+1}  , \pi^{-1}\theta AP(x_k - \bar{x}_{k+1})\rangle .\label{eq: erg_t11}
\end{align}
We  first use $x'^i = p_i^{-1}\underline p x^i + (1-p_i^{-1}\underline p)x_k^i$ as in Lemma~\ref{lem: x_lem} to get
\begin{align}
- \langle A^\top\bar{y}_{k+1}, P(x'-\bar{x}_{k+1}) \rangle = -\underline p \langle A^\top \bar{y}_{k+1}, x-x_k\rangle - \langle A^\top \bar{y}_{k+1}, P(x_k - \bar{x}_{k+1}) \rangle.\notag
\end{align}
Next, we use that
\begin{multline}
\underline p\left[ -\langle A^\top y, x_k - x \rangle + \langle Ax, \bar{y}_{k+1} - y \rangle + \langle Ax_k , y-\bar{y}_{k+1} \rangle - \langle A^\top\bar{y}_{k+1}, x-x_k \rangle  \right] = \\
\underline p\left[ \langle A^\top (y-\bar{y}_{k+1}), x-x_k \rangle + \langle y-\bar{y}_{k+1}, A(x_k - x) \rangle \right] = 0,\notag
\end{multline}
to obtain
\begin{align}
T_1 &= g(x_{k+1}) - g(x_k) + g_P(x_k) + g_P(\bar{x}_{k+1}) - g(x_{k+1}) + g(x_k) - g_P(x_k) + \underline p h^\ast(\bar{y}_{k+1})  \notag \\
&- \langle A^\top \bar{y}_{k+1}, P(x_k - \bar{x}_{k+1}) \rangle  +\underline p \langle y- \bar{y}_{k+1}, \pi^{-1}\theta AP(x_k - \bar{x}_{k+1}) \rangle \notag \\
&+\langle A^\top y, x_{k+1} - x\rangle - \underline p \langle Ax, \bar{y}_{k+1} - y \rangle + \langle A^\top y, x_k - x \rangle \notag \\
&-\langle A^\top y, x_{k+1} - x\rangle  - \langle A^\top y, x_k - x \rangle + \underline p \langle A^\top y, x_k - x\rangle \notag \\
&-\underline p \langle y- \bar{y}_{k+1}  , \pi^{-1}\theta AP(x_k - \bar{x}_{k+1})\rangle .\label{eq: erg_t112}
\end{align}

Then, we rearrange~\eqref{eq: erg_t112} in a similar way to~\Cref{lem: lem1} to get
\begin{align}
T_1 &= g(x_{k+1}) - g(x_k) + g_P(x_k) +\underline p h^\ast(\bar{y}_{k+1}) +\langle x_{k+1}-x, A^\top y\rangle - \underline p \langle Ax, \bar{y}_{k+1} - y \rangle + (\underline p -1)\langle x_k - x, A^\top y \rangle \notag \\
&+ g_P(\bar{x}_{k+1}) - g(x_{k+1}) + g(x_k) - g_P(x_k) + \underline p \langle y- \bar{y}_{k+1} , \pi^{-1}\theta AP(x_k - \bar{x}_{k+1} )\rangle \notag \\
&+ \langle A^\top y, x_k - x_{k+1} \rangle - \langle A^\top\bar{y}_{k+1}, P(x_k - \bar{x}_{k+1}) \rangle - \langle y-\bar{y}_{k+1}, \underline p\pi^{-1}\theta A P(x_k - \bar{x}_{k+1}) \rangle.\notag
\end{align}
We now use the rule $\theta_j = \frac{\pi_j}{\underline p}$ to have $\underline p \pi^{-1}\theta=I$, consequently $T_1$ simplifies to
\begin{align}
T_1 &= g(x_{k+1}) - g(x_k) + g_P(x_k) + \underline p h^\ast(\bar{y}_{k+1}) +\langle x_{k+1}-x, A^\top y\rangle - \underline p \langle Ax, \bar{y}_{k+1} - y \rangle + (\underline p -1)\langle x_k - x, A^\top y \rangle \notag \\
&+ g_P(\bar{x}_{k+1}) - g(x_{k+1}) + g(x_k) - g_P(x_k) + \langle y, A(x_k - x_{k+1}) - AP(x_k - \bar{x}_{k+1})\rangle \notag \\
&+\underline p\langle y- \bar{y}_{k+1}, \pi^{-1}\theta AP(x_k - \bar{x}_{k+1}) \rangle.\notag
\end{align}
We recall the definitions of $D_p$ and $D_d$ to write $T_1$ as
\begin{align}
T_1&= D_p(x_{k+1}; z) - f(x_{k+1})+f(x) + \underline p D_d(\bar{y}_{k+1}; z) + \underline p h^\ast(y) - (1-\underline p) D_p(x_k; z) -\underline p \left( g(x_k) - g(x) \right) + g_P(x_k)\notag \\
&+ g_P(\bar{x}_{k+1}) - g(x_{k+1}) + g(x_k) - g_P(x_k) + \langle y, A(x_k - x_{k+1}) - AP(x_k - \bar{x}_{k+1})\rangle \notag \\
&+\underline p\langle y- \bar{y}_{k+1}, \pi^{-1}\theta AP(x_k - \bar{x}_{k+1}) \rangle + (1-\underline p)(f(x_k) - f(x)).\notag
\end{align}
Second, for $T_2$, we use $x' = P^{-1}\underline p x + (1-P^{-1}\underline p)x_k = x_k + P^{-1}\underline p(x-x_k)$ to obtain
\begin{align}
&T_2 = - \langle \nabla f (x_k), P(x'-\bar{x}_{k+1})  = - \underline p\langle \nabla f (x_k), x-  x_k \rangle - \langle \nabla f (x_k), P( x_k- \bar{x}_{k+1}) \rangle.\notag
\end{align}
We now combine these two estimates
\begin{align}
T_1 + T_2 &= D_p(x_{k+1}; z) + \underline p D_d(\bar{y}_{k+1}; z) - (1-\underline p) D_p(x_k; z) + g_P(x_k) + \underline p g(x) -\underline pg(x_k) + \underline p h^\ast(y) \notag \\
&-f(x_{k+1}) + f(x_k) -\underline p f(x_k)+\underline p f(x) - \underline p \langle \nabla f(x_k), x-x_k \rangle - \langle \nabla f(x_k), P(x_k -\bar{x}_{k+1}) \rangle \notag \\
&+ g_P(\bar{x}_{k+1}) - g(x_{k+1}) + g(x_k) - g_P(x_k) + \langle y, A(x_k - x_{k+1}) - AP(x_k - \bar{x}_{k+1})\rangle \notag \\
&+\underline p\langle y- \bar{y}_{k+1}, \pi^{-1}\theta AP(x_k - \bar{x}_{k+1}) \rangle.\label{eq: t1t2}
\end{align}
We now work on $T_3$ in~\eqref{eq: erg_lem1_eq1}, in order to make terms depending on $x$ telescope.
First, we note that by Lemma~\ref{lem: x_lem}, with the slight change of using $\bar{x}_{k+1}$ instead of $x_{k+1}$ and $\tau^{-1}P$ instead of $\tau^{-1}$ in the metric, we get
\begin{align}
\frac{1}{2} \|x' - \bar{x}_{k+1} \|^2_{\tau^{-1}P} - \frac{1}{2} \| x' - x_k \|^2_{\tau^{-1}P} &= \frac{\underline p}{2} \| x-\bar{x}_{k+1}\|^2_{\tau^{-1}} - \frac{\underline p}{2} \| x- x_k \|^2_{\tau^{-1}} \notag \\
& +\frac{1}{2} \| \bar{x}_{k+1} -x_k\|^2_{\tau^{-1}P} - \frac{\underline p}{2} \| \bar{x}_{k+1} - x_k \|^2_{\tau^{-1}}.\notag
\end{align}
Thus, on $T_3$, we add and subtract $\frac{\underline p}{2} \| x - x_{k+1} \|^2_{\tau^{-1}P^{-1}} - \frac{\underline p}{2} \| x-x_k \|^2_{\tau^{-1}P^{-1}}$
\begin{align}
T_3 &= \|\bar{x}_{k+1} - x_k\|^2_{\tau^{-1}P} - \frac{\underline p}{2} \| \bar{x}_{k+1} - x_k \|^2_{\tau^{-1}} +\frac{\underline p}{2} \| x - x_{k+1} \|^2_{\tau^{-1}P^{-1}} - \frac{\underline p}{2} \| x-x_k \|^2_{\tau^{-1}P^{-1}} \notag \\
&+ \frac{\underline p}{2} \| x - \bar{x}_{k+1} \|^2_{\tau^{-1}} - \frac{\underline p}{2} \| x-x_k \|^2_{\tau^{-1}} - \left(\frac{\underline p}{2} \| x - x_{k+1} \|^2_{\tau^{-1}P^{-1}} - \frac{\underline p}{2} \| x-x_k \|^2_{\tau^{-1}P^{-1}} \right) \notag \\
&=\|\bar{x}_{k+1} - x_k\|^2_{\tau^{-1}P} - \frac{\underline p}{2} \| \bar{x}_{k+1} - x_k \|^2_{\tau^{-1}} +\frac{\underline p}{2} \| x - x_{k+1} \|^2_{\tau^{-1}P^{-1}} - \frac{\underline p}{2} \| x-x_k \|^2_{\tau^{-1}P^{-1}} + \frac{\underline p}{2} \|\bar{x}_{k+1}\|^2_{\tau^{-1}} \notag \\
&- \frac{\underline p}{2} \| x_k \|^2_{\tau^{-1}} + \underline p \langle x, x_k - \bar{x}_{k+1} \rangle_{\tau^{-1}} - \bigg( \frac{\underline p}{2} \|x_{k+1}\|^2_{\tau^{-1}P^{-1}} -\frac{\underline p}{2} \| x_k\|^2_{\tau^{-1}P^{-1}} + \underline p \langle x, (x_k - x_{k+1}) \rangle_{\tau^{-1}P^{-1}} \bigg) \notag \\
&=\|\bar{x}_{k+1} - x_k\|^2_{\tau^{-1}P} - \frac{\underline p}{2} \| \bar{x}_{k+1} - x_k \|^2_{\tau^{-1}} +\frac{\underline p}{2} \| x - x_{k+1} \|^2_{\tau^{-1}P^{-1}} - \frac{\underline p}{2} \| x-x_k \|^2_{\tau^{-1}P^{-1}} \notag \\
&+ \frac{\underline p}{2} \|\bar{x}_{k+1}\|^2_{\tau^{-1}} - \frac{\underline p}{2} \| x_k \|^2_{\tau^{-1}} -\left(\frac{\underline p}{2} \|x_{k+1}\|^2_{\tau^{-1}P^{-1}} -\frac{\underline p}{2} \| x_k\|^2_{\tau^{-1}P^{-1}} \right) + \underline p \langle x, x_k - \bar{x}_{k+1} - P^{-1}\left( x_k -x_{k+1} \right) \rangle_{\tau^{-1}}. \label{eq: t3}
\end{align}
We estimate $T_4$ in~\eqref{eq: erg_lem1_eq1} similarly.
First note that $\theta_j = \frac{\pi_j}{\underline p}$ and on $T_4$, we add and subtract $\frac{\underline p}{2} \bigg( \| y-y_{k+1} \|^2_{\sigma^{-1}\pi^{-1}} - \| y-y_k\|^2_{\sigma^{-1}\pi^{-1}} + 2\langle y-\bar{y}_{k+1}, \pi^{-1}\theta AP(x_k - \bar{x}_{k+1}) \rangle + \sum_{i=1}^n p_i \sum_{j=1}^m \pi_j^{-1}\sigma_j \theta_j^2 A_{j,i} \left( \bar{x}_{k+1}^i -x_k^i \right)^2 \bigg)$
\begin{align}
&T_4 = \frac{\underline p}{2} \| y_k - \bar{y}_{k+1} \|^2_{\sigma^{-1}} + \frac{\underline p}{2} \| y-y_{k+1} \|^2_{\sigma^{-1}\pi^{-1}} - \frac{\underline p}{2} \| y-y_k \|^2_{\sigma^{-1}\pi^{-1}} \notag \\
&+\frac{\underline p}{2} \| y-\bar{y}_{k+1} \|^2_{\sigma^{-1}} - \frac{\underline p}{2} \| y - y_k \|^2_{\sigma^{-1}} + \underline p \langle \bar{y}_{k+1} - y, \pi^{-1} \theta A P(\bar{x}_{k+1} - x_k) \rangle+ \frac{\underline p}{2} \sum_{i=1}^n p_i \sum_{j=1}^m \pi_j^{-1} \theta_j^2 \sigma_j A_{j,i}^2 (\bar{x}_{k+1}^i - x_k^i)^2\notag\\
&-\bigg( \frac{\underline p}{2} \| y-y_{k+1}\|^2_{\sigma^{-1}\pi^{-1}} - \frac{\underline p}{2}\| y - y_k \|^2_{\sigma^{-1}\pi^{-1}} \bigg) \notag\\
&- \underline p \langle \bar{y}_{k+1} - y, \pi^{-1} \theta A P(\bar{x}_{k+1} - x_k) \rangle - \frac{\underline p}{2} \sum_{i=1}^n p_i \sum_{j=1}^m \pi_j^{-1} \sigma_j\theta_j^2 A_{j,i}^2 (\bar{x}_{k+1}^i - x_k^i)^2 \notag\\
&= \frac{\underline p}{2} \| y_k - \bar{y}_{k+1}\|^2_{\sigma^{-1}} + \frac{\underline p}{2} \| y-y_{k+1} \|^2_{\sigma^{-1}\pi^{-1}} - \frac{\underline p}{2}\| y - y_k \|^2_{\sigma^{-1}\pi^{-1}} \notag\\
&+\frac{\underline p}{2} \| \bar{y}_{k+1} \|^2_{\sigma^{-1}} - \frac{\underline p}{2} \| y_k \|^2_{\sigma^{-1}} + \underline p\langle \bar{y}_{k+1},  \pi^{-1}\theta AP(\bar{x}_{k+1} - x_k) \rangle + \frac{\underline p}{2} \sum_{i=1}^n p_i \sum_{j=1}^m  \pi_j^{-1} \sigma_j \theta_j^2 A_{j, i}^2 (\bar{x}_{k+1}^i - x_k^i)^2 \notag\\
&+\underline p \langle y, y_k - \bar{y}_{k+1} \rangle_{\sigma^{-1}} -\underline p \langle y, \pi^{-1}\theta AP(\bar{x}_{k+1} - x_k) \rangle \notag\\
&-\frac{\underline p}{2} \bigg( \| y_{k+1}\|^2_{\sigma^{-1}\pi^{-1}} - \| y_k \|^2_{\sigma^{-1}\pi^{-1}} \bigg) - \underline p \langle y, y_k - y_{k+1} \rangle_{\sigma^{-1}\pi^{-1}} \notag\\
&-\underline p \langle \bar{y}_{k+1} - y, \pi^{-1} \theta AP (\bar{x}_{k+1} - x_k) \rangle - \frac{\underline p}{2} \sum_{i=1}^n p_i \sum_{j=1}^m  \pi_j^{-1} \sigma_j \theta_j^2 A_{j,i}^2  (\bar{x}_{k+1}^i - x_k^i)^2.\label{eq: t4}
\end{align}
To simplify, let us introduce some more definitions to have simpler expression when we combine $T_1, T_2, T_3, T_4$ from~\cref{eq: t1t2,eq: t3,eq: t4},
\begin{align}
S_1 &= g_P(\bar{x}_{k+1}) - g_P(x_k) - \left(g(x_{k+1}) - g(x_k) \right) \notag \\
&-f(x_{k+1}) + f(x_k) -\underline p f(x_k)+\underline p f(x) - \underline p \langle \nabla f(x_k), x-x_k \rangle - \langle \nabla f(x_k), P(x_k -\bar{x}_{k+1}) \rangle \notag \\
&+ \frac{\underline p}{2} \|\bar{x}_{k+1}\|^2_{\tau^{-1}} - \frac{\underline p}{2} \| x_k \|^2_{\tau^{-1}} -\left(\frac{\underline p}{2} \|x_{k+1}\|^2_{\tau^{-1}P^{-1}} -\frac{\underline p}{2} \| x_k\|^2_{\tau^{-1}P^{-1}} \right) \notag \\
&+\frac{\underline p}{2} \| \bar{y}_{k+1} \|^2_{\sigma^{-1}} - \frac{\underline p}{2} \| y_k \|^2_{\sigma^{-1}} + \underline p\langle \bar{y}_{k+1},  \pi^{-1}\theta AP(\bar{x}_{k+1} - x_k) \rangle + \frac{\underline p}{2} \sum_{i=1}^n p_i \sum_{j=1}^m  \pi_j^{-1} \sigma_j \theta_j^2 A_{j, i}^2 (\bar{x}_{k+1}^i - x_k^i)^2 \notag\\
&-\frac{\underline p}{2} \bigg( \| y_{k+1}\|^2_{\sigma^{-1}\pi^{-1}} - \| y_k \|^2_{\sigma^{-1}\pi^{-1}} \bigg)\label{eq: def_s1} \\
S_2 &= \langle y, A(x_k - x_{k+1}) - AP(x_k - \bar{x}_{k+1}) \rangle +\underline p \langle x, x_k - \bar{x}_{k+1} - P^{-1}(x_k-x_{k+1}) \rangle_{\tau^{-1}}\notag\\
&-\underline p \langle y, \pi^{-1}\sigma^{-1}(y_k-y_{k+1}) - \sigma^{-1}(y_k -\bar{y}_{k+1}) + \pi^{-1}\theta AP(\bar{x}_{k+1}-x_k) \rangle\notag
\end{align}
We can now collect $T_1, T_2, T_3, T_4$, and use the definitions of $S_1, S_2$, in~\eqref{eq: erg_lem1_eq1}
\begin{align}
g_P(x') + \underline p h^\ast(y) &\geq D_p(x_{k+1}; z) + \underline p D_d(\bar{y}_{k+1}; z) - (1-\underline p) D_p(x_k; z) + g_P(x_k) + \underline p g(x) -\underline pg(x_k) +\underline ph^\ast(y) \notag \\
&+\frac{\underline p}{2} \| x - x_{k+1} \|^2_{\tau^{-1}P^{-1}} - \frac{\underline p}{2} \| x-x_k \|^2_{\tau^{-1}P^{-1}}+ \frac{\underline p}{2} \| y-y_{k+1} \|^2_{\sigma^{-1}\pi^{-1}} - \frac{\underline p}{2}\| y - y_k \|^2_{\sigma^{-1}\pi^{-1}} \notag \\
&+\|\bar{x}_{k+1} - x_k\|^2_{\tau^{-1}P} - \frac{\underline p}{2} \| \bar{x}_{k+1} - x_k \|^2_{\tau^{-1}} - \frac{\underline p}{2} \sum_{i=1}^n p_i \sum_{j=1}^m  \pi_j^{-1} \sigma_j \theta_j^2 A_{j,i}^2  (\bar{x}_{k+1}^i - x_k^i)^2 \notag \\
&-\frac{1}{2} \| \bar{x}_{k+1} - x_k \|^2_{\beta P} +\frac{1}{2} \| \bar{x}_{k+1} - x_k \|^2_{\beta P}  +\frac{\underline p}{2} \| y_k - \bar{y}_{k+1}\|^2_{\sigma^{-1}}+ S_1 + S_2.\label{eq: main5}
\end{align}
We make few observations on this inequality.
First, by Lemma~\ref{lem: x_lem}, as in~\eqref{eq: gp}
\begin{equation}
g_P(x_k) + \underline p g(x) - \underline p g(x_k) - g_P(x') \geq 0.\notag
\end{equation}
Second, we have, as in~\eqref{eq: one_lem_bi4}
\begin{align}
\tilde{V}(\bar{z}_{k+1}-z_k) &= \frac{\underline p}{2} \| \bar{y}_{k+1} - y_k \|^2_{\sigma^{-1}} +  \| \bar{x}_{k+1} - x_k\|^2_{\tau^{-1}P} - \frac{\underline p}{2} \| \bar{x}_{k+1}-x_k\|^2_{\tau^{-1}} \notag\\
&-\frac{1}{2} \sum_{i=1}^n \sum_{j=1}^m \underline p p_i \pi_j^{-1} \sigma_j \theta_j^2A_{j, i}^2  (\bar{x}_{k+1}^i - x_k^i)^2 -\frac{1}{2} \| \bar{x}_{k+1} - x_k \|^2_{\beta P} \notag\\
&= \frac{\underline p}{2} \| \bar{y}_{k+1} - y_k \|^2_{\sigma^{-1}} + \frac{\underline p}{2} \| \bar{x}_{k+1} - x_k\|^2_{C(\tau)},\notag
\end{align}
where
\begin{align}
C(\tau)_i &= \frac{2p_i}{\underline p \tau_i} - \frac{1}{\tau_i} - p_i\sum_{j=1}^m \pi_j^{-1}\sigma_j \theta_j^2A_{j, i}^2 - \frac{\beta_ip_i}{ \underline p}.\notag
\end{align}
We use these estimates in~\eqref{eq: main5}
\begin{align*}
0 &\geq D_p(x_{k+1}; z) + \underline p D_d(\bar{y}_{k+1}; z) - (1-\underline p) D_p(x_k; z)+\tilde{V}(\bar{z}_{k+1} - z_k) + \frac{1}{2} \| \bar{x}_{k+1} - x_k \|^2_{\beta P} + S_1 + S_2  \\
&+\frac{\underline p}{2} \| x - x_{k+1} \|^2_{\tau^{-1}P^{-1}} - \frac{\underline p}{2} \| x-x_k \|^2_{\tau^{-1}P^{-1}}+ \frac{\underline p}{2} \| y-y_{k+1} \|^2_{\sigma^{-1}\pi^{-1}} - \frac{\underline p}{2}\| y - y_k \|^2_{\sigma^{-1}\pi^{-1}}.\qedhere
\end{align*}
\end{proof}
\begin{lemma}\label{lem: lem_erg3}
Let~\Cref{asmp: asmp1} hold and let $h$ be separable. Given the definitions of $D_p$ and $D_d$ from~\Cref{lem: lem1}, we define
\begin{align}
D_d^\gamma(\bar{y}_{k+1}; z) = \sum_{j=1}^m \gamma_j \left(h^\ast_j(\bar{y}_{k+1}^j) - h^\ast_j(y^j) - \langle (Ax)^j, \bar{y}_{k+1}^j - y^j \rangle\right).\notag
\end{align}
Moreover let $S_1, S_2$ be as~\Cref{lem: lem_erg2}, and $\breve y_k$ as~\Cref{lem: breve}, it follows that
\begin{align}
0 &\geq \underline p D_p(x_{k+1}; z) + \underline p D_d(\breve{y}_{k+1}; z) + \tilde{V}(\bar{z}_{k+1} - z_k) + \frac{1}{2} \| \bar{x}_{k+1} - x_k \|^2_{\beta P} + S_1 + S_2 \notag \\
&+\frac{\underline p}{2} \| x - x_{k+1} \|^2_{\tau^{-1}P^{-1}} - \frac{\underline p}{2} \| x-x_k \|^2_{\tau^{-1}P^{-1}}+ \frac{\underline p}{2} \| y-y_{k+1} \|^2_{\sigma^{-1}\pi^{-1}} - \frac{\underline p}{2}\| y - y_k \|^2_{\sigma^{-1}\pi^{-1}} \notag \\
&+ (1-\underline p) D_p(x_{k+1}; z) - (1-\underline p) D_p(x_k; z)+ \underline p D_d^{\pi^{-1}-I}(\breve y_{k+1}; z) - \underline p D_d^{\pi^{-1}-I}(\breve y_{k}; z)\notag \\
&+ \underline p h^\ast(\bar{y}_{k+1}) - \underline p h^\ast(\breve{y}_k) - \underline p\left( h^\ast_{\pi^{-1}}(\breve{y}_{k+1}) - h^\ast_{\pi^{-1}}(\breve{y}_k)\right)+\underline p \langle Ax, \breve{y}_k - \bar{y}_{k+1} - \pi^{-1}\left( \breve{y}_k - \breve{y}_{k+1} \right) \rangle.\notag
\end{align}
\end{lemma}
\begin{proof}
This lemma is the intricate part of the proof of~\Cref{th:erg}. 
If we finish the estimations as in~\citet{alacaoglu2019convergence}, then we will end up needing to average $\bar{y}_k$.
However, this is not feasible in our algorithm, since we do not update full dual vector, thus we do not compute $\bar{y}_k$ unless the data is fully dense.
We will use Lemma~\ref{lem: breve} to go from $\bar{y}_k$ to $\breve{y}_k$.
Let us repeat the definition of $\breve{y}_k$ from Lemma~\ref{lem: breve}:
Let $\breve{y}_1 = y_1 = \bar{y}_1$, and
\begin{equation}
\begin{aligned}
&\breve{y}_{k+1}^j = \bar{y}_{k+1}^j, &&\forall j \in J(i_{k+1}) \notag\\
&\breve{y}_{k+1}^j = \breve{y}_k^j, &&\forall j\not\in J(i_{k+1}).\notag
\end{aligned}
\end{equation}
We now work on $D_d(\bar{y}_{k+1}; z)$ and note that $h^\ast_{\gamma}$ is defined as in~\Cref{lem: breve}.
\begin{align}
D_d(\bar{y}_{k+1} ;z) &= D_d(\breve{y}_{k+1}; z) - \langle Ax, \bar{y}_{k+1} \rangle + h^\ast(\bar{y}_{k+1}) + \langle Ax, \breve{y}_{k+1} \rangle - h^\ast(\breve{y}_{k+1}) \notag\\
&= D_d(\breve{y}_{k+1}; z) - \langle Ax, \bar{y}_{k+1} \rangle + h^\ast(\bar{y}_{k+1}) + \langle Ax, \breve{y}_{k+1} \rangle - h^\ast(\breve{y}_{k+1}) \notag\\
&+ h^\ast_{I - \pi^{-1}}(\breve{y}_{k+1}) - h^\ast_{I - \pi^{-1}}(\breve{y}_{k+1}) + h^\ast_{I - \pi^{-1}}(\breve{y}_{k}) - h^\ast_{I - \pi^{-1}}(\breve{y}_{k}) \notag\\
&+ \langle Ax, (I - \pi^{-1})(\breve{y}_{k+1}-\breve{y}_k) \rangle -  \langle Ax, (I - \pi^{-1})(\breve{y}_{k+1}-\breve{y}_k) \notag \\
&=D_d( \breve{y}_{k+1}; z) + h^\ast(\bar{y}_{k+1}) - h^\ast(\breve{y}_k) - \left( h^\ast_{\pi^{-1}}(\breve{y}_{k+1}) - h^\ast_{\pi^{-1}}(\breve{y}_k)\right) \notag\\
&+ h^\ast_{\pi^{-1}-1}(\breve{y}_{k+1}) - h^\ast_{\pi^{-1}-1}(\breve{y}_k) +\langle Ax, \breve{y}_k - \bar{y}_{k+1} - \pi^{-1}\left( \breve{y}_k - \breve{y}_{k+1} \right) \rangle \notag \\
&+ \langle Ax, (I - \pi^{-1})(\breve{y}_{k+1} - \breve{y}_k) \rangle \notag \\
&=D_d( \breve{y}_{k+1}; z) + h^\ast(\bar{y}_{k+1}) - h^\ast(\breve{y}_k) - \left( h^\ast_{\pi^{-1}}(\breve{y}_{k+1}) - h^\ast_{\pi^{-1}}(\breve{y}_k)\right) \notag\\
&+ \langle Ax, \breve{y}_k - \bar{y}_{k+1} - \pi^{-1}\left( \breve{y}_k - \breve{y}_{k+1} \right) \rangle + D_d^{\pi^{-1}-I}(\breve y_{k+1}; z) - D_d^{\pi^{-1} - I}(\breve y_k; z)\notag
\end{align}
We can insert this estimate into the result of~\Cref{lem: lem_erg2}
\begin{align*}
0 &\geq \underline p D_p(x_{k+1}; z) + \underline p D_d(\breve{y}_{k+1}; z) + \tilde{V}(\bar{z}_{k+1} - z_k) + \frac{1}{2} \| \bar{x}_{k+1} - x_k \|^2_{\beta P} + S_1 + S_2 \notag \\
&+\frac{\underline p}{2} \| x - x_{k+1} \|^2_{\tau^{-1}P^{-1}} - \frac{\underline p}{2} \| x-x_k \|^2_{\tau^{-1}P^{-1}}+ \frac{\underline p}{2} \| y-y_{k+1} \|^2_{\sigma^{-1}\pi^{-1}} - \frac{\underline p}{2}\| y - y_k \|^2_{\sigma^{-1}\pi^{-1}} \notag \\
&+ (1-\underline p) D_p(x_{k+1}; z) - (1-\underline p) D_p(x_k; z)+ \underline p D_d^{\pi^{-1}-I}(\breve y_{k+1}; z) - \underline p D_d^{\pi^{-1}-I}(\breve y_{k}; z)\notag \\
&+ \underline p h^\ast(\bar{y}_{k+1}) - \underline p h^\ast(\breve{y}_k) - \underline p\left( h^\ast_{\pi^{-1}}(\breve{y}_{k+1}) - h^\ast_{\pi^{-1}}(\breve{y}_k)\right)+\underline p \langle Ax, \breve{y}_k - \bar{y}_{k+1} - \pi^{-1}\left( \breve{y}_k - \breve{y}_{k+1} \right) \rangle.\qedhere
\end{align*}
\end{proof}

The following lemma is similar to~\citep[Lemma 4.8]{alacaoglu2019convergence}.
\begin{lemma}\label{lem: random_proc_lemma}
Given a Euclidean space $\mathcal{S}$, a fixed diagonal matrix $\gamma \succeq 0$, let the random sequences $u_{k}, v_{k} \in \mathcal{S}$ be $\mathcal{F}_k$-measurable with 
\begin{equation}
u_{k+1} = v_{k+1} - \Ec{k}{v_{k+1}}.\notag
\end{equation} 
Let $\tilde x_1$ be arbitrary and set for $k \geq 1$,
\begin{equation}
\tilde x_{k+1} = \tilde x_{k} - u_{k+1}.\notag
\end{equation}
Then, $\tilde x_k$ is $\mathcal{F}_k$-measurable and we have for any $S \subset \mathcal{S}$
\begin{equation}
\E{\sup_{x\in{S}} \left\{\sum_{k=1}^K \langle x, u_{k+1} \rangle_{\gamma} -\frac{1}{2} \| \tilde x_1 - x \|^2_{\gamma}\right\}} \leq \frac{1}{2} \sum_{k=1}^K \E{\| v_{k+1} \|_{\gamma}^2}.\notag
\end{equation}
\begin{proof}
First, by the definition of $\tilde x_{k+1}$, it follows that for all $x\in\mathcal{S}$
\begin{equation}
\| \tilde x_{k+1} - x \|^2_{\gamma} = \| \tilde x_k - x \|^2_{\gamma} - 2\langle \tilde x_k - x, u_{k+1} \rangle_{\gamma} + \| u_{k+1} \|^2_{\gamma}.\notag
\end{equation}
Summing this inequality gives
\begin{equation}
\sum_{k=1}^K \langle x, u_{k+1} \rangle_\gamma -\frac{1}{2}\| \tilde x_1 - x \|^2_\gamma \leq - \sum_{k=1}^K \langle \tilde x_k, u_{k+1} \rangle_\gamma + \sum_{k=1}^K \frac{1}{2}\| u_{k+1}\|^2_\gamma.\notag
\end{equation}
We take first supremum and then expectation of both sides to get
\begin{equation}
\mathbb{E} \left[ \sup_{x\in S}\left\{- \frac{1}{2}\| \tilde x_1 - x \|^2_\gamma+\sum_{k=1}^K \langle x, u_{k+1} \rangle_\gamma\right\}\right] \leq  - \sum_{k=1}^K \mathbb{E}\left[ \langle \tilde x_k, u_{k+1} \rangle_\gamma \right] + \sum_{k=1}^K \frac{1}{2}\mathbb{E}\left[ \| u_{k+1}\|^2_\gamma \right].\notag
\end{equation}
By the law of total expectation, $\mathcal{F}_k$-measurability of $\tilde x_k$ and $\Ec{k}{u_{k+1}} = 0$, we have 
\begin{equation}
\sum_{k=1}^K \mathbb{E} \left[ \langle \tilde x_k, u_{k+1}\rangle_\gamma \right] = \sum_{k=1}^K \mathbb{E}\left[\mathbb{E}_k \left[ \langle \tilde x_k, u_{k+1}\rangle_\gamma \right]\right] = \sum_{k=1}^K \mathbb{E}\left[ \langle \tilde x_k, \Ec{k}{u_{k+1}} \rangle_\gamma \right] = 0.\notag
\end{equation}
Finally, we use the definition of $u_k$ and the inequality $\E{\| X - \E{X} \|^2} \leq \E{\| X \|^2}$ which holds for any random variable $X$.
\end{proof}
\end{lemma}

As mentioned in the main text, we will give the theorems in the appendix with tighter, but more complicated constants. After~\Cref{cor: erg}, we show how we obtained the simplified bounds in our main text.
\begin{reptheorem}{th:erg}
Let Assumption~\ref{asmp: asmp1} hold and $\theta, \tau, \sigma$ are chosen as in~\eqref{eq: theta_choice},~\eqref{eq: ss_choice}. Moreover, let $h$ be separable.

We define $x^{av}_K = \frac{1}{K} \sum_{k=1}^K x_{k}$ and $y^{av}_K = \frac{1}{K} \sum_{k=1}^K \breve{y}_{k}$, where $\breve{y}_k$ is defined in~\eqref{eq: breve_def}, then it holds that for any bounded set $\mathcal{C}=\mathcal{C}_x\times\mathcal{C}_y\subset\mathcal{Z}$
\begin{align}
\mathbb{E} \left[G_\mathcal{C}(x^{av}_K, y^{av}_K) \right] \leq \frac{C_{g}}{\underline p K},\notag
\end{align}
where $C_g = C_{g,1} + C_{g, 2} + C_{g, 3} + C_{g, 4}$, $C_{\tau, \tilde V} =  \min_i C(\tau)_i\tau_i$,\\
$C_{g, 1} = \sup_{z\in \mathcal{C}} \big\{2\underline p \| x_0 - x\|^2_{\tau^{-1}P^{-1}} + 2\underline p\| y_0 - y \|^2_{\sigma^{-1}\pi^{-1}} \big\} + (1-\underline p)4\sqrt{\Delta\underline p^{-1}}\|A\| \sup_{y\in \mathcal{C}_y}\|y\|_{\tau P} \\
+ \underline p \sqrt{ \Delta_0\underline p^{-1} - \| 2P-\underline p\| \Delta_0\underline p^{-2}C_{\tau, \tilde V}^{-1} }\|A\|\|\pi^{-1} - I \| \sup_{x\in\mathcal{C}_x}\| x\|_{\sigma \pi}$, \\
$C_{g, 2} = \|2P - \underline p \| \Big(1+\| P/\underline p \| + \| \tau^{1/2}P^{1/2}A^\top \pi^{-1/2}\sigma^{1/2}\|^2 \Big) \frac{2\Delta_0}{\underline p C_{\tau, \tilde V}} + \frac{\Delta_0}{C_{\tau, \tilde V}} + (4+4\|\tau^{1/2}P^{1/2}A^\top \pi^{-1/2}\sigma^{1/2}\|^2)\Delta_0$,
$C_{g, 3} = (1-\underline p)\left(f(x_0) + g(x_0) - f(x_\star)  - g(x_\star) + \| A^\top y_\star \|_{\tau P} \sqrt{\Delta_0 \underline p^{-1}}\right)$,\\
$C_{g, 4} = \underline p h^\ast_{\pi^{-1}-I}(\breve y_0) + \underline p\sum_{j=1}^n (\pi_j^{-1}-1)h^\ast_j(\breve y_\star^j) + \frac{\underline p}{2} \| Ax_\star \|^2_{\sigma\pi^{-1}}
+\Delta_0 + \frac{\| 2P - \underline p \|\Delta_0}{\underline p C_{\tau, \tilde V}}$.
\end{reptheorem}
\begin{proof}
We start with the result of~\Cref{lem: lem_erg3}.
First, we will manipulate the terms arising in $S_2 + \langle Ax, \breve{y}_k -\bar{y}_{k+1} - \pi^{-1}(\breve{y}_k - \breve{y}_{k+1})$ (see definition of $S_2$ in~\Cref{lem: lem_erg2}).
\begin{multline}
-\left(S_2 + \underline p \langle Ax, \breve{y}_k -\bar{y}_{k+1} - \pi^{-1}(\breve{y}_k - \breve{y}_{k+1})\right)\rangle = -\langle y, A(x_k - x_{k+1}) - AP(x_k - \bar{x}_{k+1}) \rangle\\
-\underline p \langle x, x_k - \bar{x}_{k+1} - P^{-1}(x_k-x_{k+1}) \rangle_{\tau^{-1}} +\underline p \langle y, \pi^{-1}\sigma^{-1}(y_k-y_{k+1})- \sigma^{-1}(y_k -\bar{y}_{k+1}) + \pi^{-1}\theta AP(\bar{x}_{k+1}-x_k) \rangle\\ 
- \underline p \langle Ax, \breve{y}_k -\bar{y}_{k+1} - \pi^{-1}(\breve{y}_k - \breve{y}_{k+1})\rangle \label{eq: erg_eq1_4terms}
\end{multline}
For the four terms on the right hand side, we will apply~\Cref{lem: random_proc_lemma}.
We note first, $\mathbb{E}_k \left[ \pi^{-1}\left( \breve{y}_k - \breve{y}_{k+1} \right) \right] = \breve{y}_k - \bar{y}_{k+1}$ from Lemma~\ref{lem: breve}, $\mathbb{E}_k \left[ P^{-1}(x_k - x_{k+1}) = x_k-\bar{x}_{k+1} \right]$, by coordinate wise updates.
Finally, as in the proof of Lemma~\ref{lem: breve}, we can derive, as $A_{j, i} = 0, \forall i \not \in I(j)$,
\begin{align*}
\mathbb E_k[y_{k+1}^j] &= \sum_{i=1}^n p_i \Big [\mathds{1}_{j \in J(i)} \big( \bar y_{k+1}^j + \sigma_j \theta_j A_{j,i} (\bar x_{k+1}^i - x_k^i) \big) + \mathds{1}_{j \not \in J(i)} y_k^j\Big] \\
& = y_k^j + \sum_{i \in I(j)} p_i \big( \bar y_{k+1}^j - y_k^j) + \sum_{i=1}^n p_i \sigma_j \theta_j A_{j,i}  (\bar x_{k+1}^i - x_k^i)  = y_k^j + \pi_j (\bar y_{k+1}^j - y_k^j) + \sigma_j\theta_j (A P (\bar x_{k+1} - x_k))_j \\
\mathbb E_k[y_{k+1}] &=y_k + \pi (\bar y_{k+1} - y_k) + \sigma\theta A P (\bar x_{k+1} - x_k) \iff \mathbb{E}_k\left[ y_k - y_{k+1} \right] = \pi (y_k - \bar{y}_{k+1}) -  \sigma \theta AP (\bar{x}_{k+1} - x_k).
\end{align*}
In particular, for~\eqref{eq: erg_eq1_4terms}, we set in~\Cref{lem: random_proc_lemma}
\begin{alignat}{4}
&u_{k+1} = -\underline p^{-1}\sigma\pi A(x_k - x_{k+1}) + \underline p^{-1}\sigma\pi AP(x_k - \bar{x}_{k+1}),&&\gamma = \sigma^{-1}\pi^{-1}\underline p, && \mathcal{S} = \mathcal{Y},~~~~&&\tilde x_1 = y_1. \notag \\
&u_{k+1} = (x_k - x_{k+1})-P(x_k - \bar{x}_{k+1}),&&\gamma = \tau^{-1}P^{-1},~~~~&& \mathcal{S} = \mathcal{X},~~~~&&\tilde x_1 = x_1, \notag \\
&u_{k+1} = \left( y_k - y_{k+1} \right) - \pi \left( y_k - \bar{y}_{k+1} \right) +  \sigma \theta AP(\bar{x}_{k+1} - x_k),~~~~&&\gamma = \sigma^{-1}\pi^{-1}, &&\mathcal{S} = \mathcal{Y},~~~~&&\tilde x_1 = y_1, \notag\\
&u_{k+1} = \tau P A^\top \left( \pi^{-1}\left( \breve{y}_k - \breve{y}_{k+1} \right) - \left(\breve{y}_k - \bar{y}_{k+1}\right) \right),&&\gamma = \tau^{-1}P^{-1}, && \mathcal{S} = \mathcal{X}, ~~~~&&\tilde x_1 = x_1,\notag
\end{alignat}

Then, we can apply~\Cref{lem: random_proc_lemma} for these cases to bound~\eqref{eq: erg_eq1_4terms} as
\begin{align}
\E{\sup_{z\in \mathcal{C}} \eqref{eq: erg_eq1_4terms}} &\leq \sup_{z\in \mathcal{C}} \left\{ \underline p \| x-x_1\|^2_{\tau^{-1}P^{-1}} + \underline p \| y-y_1\|^2_{\sigma^{-1}\pi^{-1}}\right\} + \sum_{k=1}^K \frac{\underline p}{2} \E{\| x_k - x_{k+1} \|^2_{\tau^{-1}P^{-1}} + \| y_k -y_{k+1} \|^2_{\sigma^{-1}\pi^{-1}}}\notag \\
&+\sum_{k=1}^K \frac{1}{2\underline p} \E{\|\sigma \pi A(x_k - x_{k+1}) \|^2_{\sigma^{-1}\pi^{-1}}} + \sum_{k=1}^K \frac{\underline p}{2}\E{\|\tau P A^\top\pi^{-1}(\breve y_k - \breve y_{k+1})\|^2_{\tau^{-1}P^{-1}}}.\label{eq: four_terms_bd}
\end{align}
We now recall the definition of $S_1$ from~\eqref{eq: def_s1}, and use the identities~\eqref{eq: gp_exp},~\eqref{eq: x_cond}, Lemma~\ref{lem: y_lem}, along with the law of total expectation to estimate
\begin{align}
\mathbb{E} \left[ S_1 + \frac{1}{2} \| \bar{x}_{k+1} - x_k \|^2_{\beta P} \right] &= \mathbb{E} \left[ -f(x_{k+1}) + f(x_k) - \langle \nabla f(x_k), P(x_k - \bar{x}_{k+1})\rangle + \frac{1}{2} \| \bar{x}_{k+1} - x_k \|^2_{\beta P} \right] \notag \\
&+\underline p \mathbb{E}\left[ f(x) - f(x_k)-\langle \nabla f(x_k), x-x_k\rangle \right] \notag \\
&\geq \mathbb{E} \left[ -f(x_{k+1}) + f(x_k) - \mathbb{E}_k \left[\langle\nabla  f(x_k), x_k - {x}_{k+1}\rangle\right] + \frac{1}{2} \| \bar{x}_{k+1} - x_k \|^2_{\beta P} \right]\notag \\
&\geq \mathbb{E}\left[ -\frac{1}{2} \mathbb{E}_k \left[\| x_k - x_{k+1} \|^2_{\beta}\right] + \frac{1}{2} \| \bar{x}_{k+1} -x_k \|^2_{\beta P} \right] \notag \\
&= \mathbb{E}\left[ -\frac{1}{2} \| x_k - \bar{x}_{k+1} \|^2_{\beta P} + \frac{1}{2} \| \bar{x}_{k+1} -x_k \|^2_{\beta P} \right] \notag \\
&=0,\label{eq: exp1}
\end{align}
where the first inequality is by convexity, second inequality is by coordinatewise smoothness of $f$.

Furthermore, for the result of~\Cref{lem: lem_erg3}, by Lemma~\ref{lem: breve} and the law of total expectation
\begin{equation}
\mathbb{E}\left[ h^\ast(\bar{y}_{k+1}) - h^\ast(\breve{y}_k) - \left( h^\ast_{\pi^{-1}}(\breve{y}_{k+1}) - h^\ast_{\pi^{-1}}(\breve{y}_k)\right) \right] = 0.\label{eq: exp2}
\end{equation}

We rearrange and sum the result of~\Cref{lem: lem_erg3}, take supremum and expectation, plug in~\cref{eq: four_terms_bd,eq: exp1,eq: exp2}, and use $\tilde V$ is a squared norm 
\begin{align}
\mathbb{E}\bigg[\sup_{z\in \mathcal{C}}\sum_{k=1}^K \underline p&\left(D_p(x_{k}) +  D_d(\breve y_{k})\right)\bigg] \leq \sup_{z\in \mathcal{C}} \frac{3\underline p}{2} \left( \| x-x_0 \|^2_{\tau^{-1}P^{-1}} + \| y-y_0 \|^2_{\sigma^{-1}\pi^{-1}} \right) \notag \\
&+\E{\sup_{z\in \mathcal{C}}(1-\underline p) \left( D_p(x_0; z) - D_p(x_{K}; z) \right)} + \E{\sup_{z\in \mathcal{C}} \underline pD_d^{\pi^{-1}-I}(\breve y_0; z) - \underline pD_d^{\pi^{-1}-I}(\breve y_{K}; z)}\notag \\
&+ \sum_{k=1}^K \frac{\underline p}{2} \E{\| x_k - x_{k+1} \|^2_{\tau^{-1}P^{-1}} + \| y_k -y_{k+1} \|^2_{\sigma^{-1}\pi^{-1}}} +\sum_{k=1}^K \frac{1}{2\underline p} \E{\|\sigma \pi A(x_k - x_{k+1}) \|^2_{\sigma^{-1}\pi^{-1}}} \notag \\
&+ \sum_{k=1}^K \frac{\underline p}{2}\E{\|\tau P A^\top\pi^{-1}(\breve y_k - \breve y_{k+1})\|^2_{\tau^{-1}P^{-1}}}.\label{eq: breg_tow_last}
\end{align}
We have
\begin{align}
&\E{\sup_{z\in\mathcal{C}}D_p(x_0; z) - D_p(x_{K}; z)} = \E{\sup_{z\in\mathcal{C}}D_p (x_0; (x_{K}, y))} \notag \\
&=\E{\sup_{z\in \mathcal{C}_y} f(x_0)+g(x_0) - f(x_K) - g(x_K) + \langle A^\top y, x_0 - x_K \rangle} \notag \\
&\leq \E{\sup_{z\in \mathcal{C}_y} f(x_0)+g(x_0) - f(x_K) - g(x_K) + \| A\| \| y\|_{\tau P} \| x_0 - x_K \|_{\tau^{-1}P^{-1}}}.
\end{align}
Then, we use the optimality conditions, convexity, and~\eqref{eq: as_eq1}
\begin{align}
\mathbb{E} \left[ f(x_{K}) + g(x_{K})\right] &\geq \mathbb{E}\left[ f(x_\star) + g(x_\star) - \langle A^\top y_\star, x_{K} - x_\star \rangle\right]\notag\\ &\geq \mathbb{E}\left[f(x_\star) + g(x_\star) - \| A^\top y_\star \|^2_{\tau P}\|x_{K} - x_\star \|_{\tau^{-1}P^{-1}}\right] \notag \\
&\geq  f(x_\star) + g(x_\star) - \| A^\top y_\star \|_{\tau P}\sqrt{\frac{\Delta_0}{\underline p}}.\notag
\end{align}
to obtain for this estimation
\begin{align}
\E{\sup_{z\in C} D_p(x_0; (x_K; y)) } \leq f(x_0) + g(x_0) - f(x_\star) - g(x_\star) &+ \| A^\top y_\star \|_{\tau P} \sqrt{\Delta_0 \underline p^{-1}} \notag \\
&+ 4\sqrt{\Delta_0\underline p^{-1}}\|A\| \sup_{y\in C_y} \| y \|_{\tau P}.\label{eq: dp_last}
\end{align}
We estimate similarly to obtain
\begin{multline}
\E{\sup_{z\in \mathcal{C}} \underline pD_d^{\pi^{-1}-I}(\breve y_0; z) - \underline pD_d^{\pi^{-1}-I}(\breve y_{K}; z)} = \underline p\E{ h^\ast_{\pi^{-1} - I}(\breve y_0) - h^\ast_{\pi^{-1}-I}(\breve y_{K}) - \langle Ax,(\pi^{-1}-I) (\breve y_0 - \breve y_{K}) \rangle }.\\
\leq \underline p\E{ h^\ast_{\pi^{-1} - I}(\breve y_0) - h^\ast_{\pi^{-1}-I}(\breve y_{K}) +\|A\|\|\pi^{-1} - I \| \| x\|_{\sigma \pi}\|\breve y_0 - \breve y_K \|_{\sigma^{-1}\pi^{-1}}  }.\label{eq: dd_new}
\end{multline}
Similarly,
\begin{align}
\mathbb{E} \left[ h^\ast_{\pi^{-1}-I}(\breve{y}_{K})\right] &=\mathbb{E} \left[ \sum_{j=1}^m (\pi_j^{-1} - 1) h^\ast_j (\breve{y}_{K}^j) \geq \sum_{j=1}^m (\pi_j^{-1} - 1) \left(h^\ast_j(y_\star^j) + \langle (Ax_\star)_j, \breve{y}_{K}^j - y_\star^j\rangle \right)\right] \notag \\
&\geq\mathbb{E} \left[\sum_{j=1}^m (\pi_j^{-1} - 1) h^\ast_j(y_\star^j) - \frac12\|Ax_\star \|^2_{\sigma \pi^{-1}} - \frac12 \| \breve y_{K} - y_\star \|^2_{\sigma^{-1}\pi^{-1}} \right] \notag \\
&\geq\mathbb{E} \left[\sum_{j=1}^m (\pi_j^{-1} - 1) h^\ast_j(y_\star^j) - \frac12\|Ax_\star \|^2_{\sigma \pi^{-1}} - \|y_{K} - y_\star \|^2_{\sigma^{-1}\pi^{-1}} - \sum_{k=1}^K \| \bar x_{k+1} - x_k \|^2_{B(\pi^{-1}\sigma^{-1})} \right] \notag \\
&\geq \sum_{j=1}^m (\pi_j^{-1} - 1) h^\ast_j(y_\star^j) - \frac12\|Ax_\star \|^2_{\sigma \pi^{-1}} - \frac{\Delta_0}{\underline p} - \frac{\| 2P - \underline p \|\Delta_0}{\underline p^2 C_{\tau, \tilde V}}, \notag
\end{align}
where we used
\begin{equation}
\| x\|^2_{B(\pi^{-1}\sigma^{-1})} \leq \frac{1}{\underline p} \| 2P - \underline p \| \| x \|^2_{\tau^{-1}},\label{eq: B_conv}
\end{equation}
which follows by using the step size rule from~\eqref{eq: ss_choice} and definition of $B(\gamma)$ from~\Cref{lem: breve}.

Thus, the final bound for~\eqref{eq: dd_new}
\begin{multline}
\E{\sup_{z\in \mathcal{C}} \underline pD_d^{\pi^{-1}-I}(\breve y_0; z) - \underline pD_d^{\pi^{-1}-I}(\breve y_{K}; z)} \leq \underline p h^\ast_{\pi^{-1}-I}(\breve y_0) + \underline p\sum_{j=1}^n (\pi_j^{-1}-1)h^\ast_j(\breve y_\star^j) + \frac{\underline p}{2} \| Ax_\star \|^2_{\sigma\pi^{-1}} \\
+\Delta_0 + \frac{\| 2P - \underline p \|\Delta_0}{\underline p C_{\tau, \tilde V}} + \underline p \sqrt{ \Delta_0\underline p^{-1} - \| 2P-\underline p\| \Delta_0\underline p^{-2}C_{\tau, \tilde V}^{-1} }\|A\|\|\pi^{-1} - I \| \sup_{x\in\mathcal{C}_x}\| x\|_{\sigma \pi}.\label{eq: dd_last}
\end{multline}

By~\eqref{eq: sum_of_subseq}
\begin{align}
&\sum_{k=1}^\infty \mathbb{E} \left[ \tilde{V}(\bar{z}_{k+1} - z_k) \right] \leq \Delta_0. \label{eq: vt_final_bd}\\
&\tilde V(\bar z_{k+1} - z_k )\geq \frac{\underline p C_{\tau, \tilde V}}{2} \| \bar x_{k+1} - x_k \|^2_{\tau^{-1}} + \frac{\underline p}{2} \| \bar y_{k+1} - y_k \|^2_{\sigma^{-1}} ,\label{eq: erg_conv}
\end{align}
with $C_{\tau, \tilde V} = \min_i C(\tau)_i\tau_i$, where we used the definition of $\tilde V$ from~\Cref{lem: lem1}.

We continue to estimate, by~\Cref{lem: y_lem} and the definition of $B(\gamma)$ from~\Cref{lem: breve}
\begin{align}
\Ec{k}{\| y_{k+1} - y_k \|^2_{\sigma^{-1}\pi^{-1}}} &= \| \bar y_{k+1} - y_k \|^2_{\sigma^{-1}} + 2\langle \bar y_{k+1} - y_k, \pi^{-1} \theta AP(\bar x_{k+1} - x_k) \rangle \notag\\
&\qquad\qquad\qquad\qquad\qquad\qquad\qquad\qquad\qquad
+ \sum_{i=1}^n \sum_{j=1}^m p_i \pi^{-1}_j \sigma_j \theta_j^2 A_{j, i}^2 (\bar x_{k+1}^i - x_k^i)^2 \notag \\
&\leq \| \bar y_{k+1} - y_k \|^2_{\sigma^{-1}} + \| \bar y_{k+1} - y_k \|^2_{\sigma^{-1}} + \sum_{i=1}^n \sum_{j=1}^m p_i^2 \pi_j^{-2}\sigma_j \theta_j^2 A_{j, i}^2 (\bar x_{k+1}^i - x_k^i)^2 \notag \\
&+ \sum_{i=1}^n \sum_{j=1}^m p_i \pi^{-1}_j \sigma_j \theta_j^2 A_{j, i}^2 (\bar x_{k+1}^i - x_k^i)^2 \notag \\
&\leq 2 \| \bar y_{k+1} - y_k \|^2_{\sigma^{-1}} + \left( 1+ \| P/\underline p\| \right)\| \bar x_{k+1} - x_k \|^2_{B(\pi^{-1}\sigma^{-1})}, \label{eq: erg_y_bd} \\
\frac{\underline p}{2} \Ec{k}{\| x_{k+1} - x_k \|^2_{\tau^{-1}P^{-1}}} &= \frac{\underline p}{2} \| \bar x_{k+1} - x_k\|^2_{\tau^{-1}} \leq \frac{\Delta_0}{C_{\tau, \tilde V}}.\label{eq: erg_x_bd}
\end{align}
Moreover, it holds that
\begin{align}
\frac{1}{2\underline p} \Ec{k}{\|\sigma \pi A(x_k - x_{k+1}) \|^2_{\sigma^{-1}\pi^{-1}}} &=\Ec{k}{\frac{1}{2\underline p}  \sum_{j=1}^m \sigma_j \pi_j (A(x_k - x_{k+1}))_j^2} \notag \\
&= \Ec{k}{\frac{1}{2\underline p}  \sum_{j=1}^m \sigma_j \pi_j A_{j, i_{k+1}}^2(x_k^{i_{k+1}} - \bar x_{k+1}^{i_{k+1}})^2} \notag \\
&=\frac{1}{2\underline p}  \sum_{j=1}^m \sum_{i=1}^n p_i \sigma_j \pi_j A_{j, i}^2(x_k^{i} - \bar x_{k+1}^{i})^2 \notag\\
&= \frac{\underline p}{2} \| \bar x_{k+1} - x_k \|^2_{B(\pi^{-1}\sigma^{-1})}.\label{eq: erg_x_bd2}
\end{align}
Finally,
\begin{align}
\sum_{k=1}^K \mathbb{E}_k \Big[\| \tau P A^\top \pi^{-1}&(\breve y_{k+1} - \breve y_k) \|^2_{\tau^{-1}P^{-1}}\Big] \leq  \|\tau^{1/2}P^{1/2}A^\top \pi^{-1/2}\sigma^{1/2} \|^2 \sum_{k=1}^K \Ec{k}{\| \breve y_{k+1} - \breve y_k \|^2_{\sigma^{-1}\pi^{-1}}} \notag \\
&\leq \| \tau^{1/2}P^{1/2}A^\top \pi^{-1/2} \sigma^{1/2} \|^2 \sum_{k=1}^K \left( 2\| \bar y_{k+1} - y_k \|^2_{\sigma^{-1}} + 2\| \bar x_{k+1} - x_k \|^2_{B(\pi^{-1}\sigma^{-1})} \right), \label{eq: erg_y_bd2}
\end{align}
where we use~\Cref{lem: breve} for the last inequality.

By,~\cref{eq: erg_x_bd,eq: erg_x_bd2,eq: erg_y_bd,eq: erg_y_bd2}, we denote, in~\eqref{eq: breg_tow_last}, we denote
\begin{align}
&+ \sum_{k=1}^K \frac{\underline p}{2} \E{\| x_k - x_{k+1} \|^2_{\tau^{-1}P^{-1}} + \| y_k -y_{k+1} \|^2_{\sigma^{-1}\pi^{-1}}} +\sum_{k=1}^K \frac{1}{2\underline p} \E{\|\sigma \pi A(x_k - x_{k+1}) \|^2_{\sigma^{-1}\pi^{-1}}} \notag \\
&+ \sum_{k=1}^K \frac{\underline p}{2}\E{\|\tau P A^\top\pi^{-1}(\breve y_k - \breve y_{k+1})\|^2_{\tau^{-1}P^{-1}}} \leq C_{g, 2},\label{eq: cg2}
\end{align}
where the exact expression for $C_{g,2}$ is given in the statement of~\Cref{th:erg} in the appendix.

Then, on~\eqref{eq: breg_tow_last}, we use~\cref{eq: dp_last,eq: dd_last,eq: vt_final_bd,eq: erg_conv,eq: erg_x_bd,eq: erg_y_bd,eq: erg_x_bd2,eq: erg_y_bd2,eq: B_conv}, definition of primal-dual gap function in~\eqref{eq: def_pdgap}, and Jensen's inequality to conclude.
\end{proof}
\begin{reptheorem}{cor: erg} 
Let Assumption~\ref{asmp: asmp1} hold. We use the same parameters $\theta,\tau,\sigma$ and the definitions for $x_K^{av}$ and $y_K^{av}$ as Theorem~\ref{th:erg}. 
We consider two cases separately: \\
$\triangleright$ If $h(\cdot) = \delta_{\{ b \}}(\cdot)$, we obtain
\begin{align}
&\mathbb{E} \left[ f(x_K^{av}) + g(x_K^{av}) - f(x_\star) - g(x_\star)\right] \leq \frac{C_o}{\underline pK}. \notag\\
&\mathbb{E}\left[ \|Ax_K^{av}-b \|\right] \leq \frac{C_f}{\underline pK}.\notag
\end{align}
$\triangleright$ If $h$ is $L_h$-Lipschitz continuous, we obtain
\begin{equation}
\mathbb{E}\Big[ f(x_K^{av}) + g(x_K^{av}) + h(Ax_K^{av})
- f(x_\star) - g(x_\star) -h(Ax_\star)\Big] \leq \frac{C_l}{\underline pK},\notag
\end{equation}
where $C_f = 2c_2\sqrt{\|y_\star - y_0 \|_{\sigma^{-1}\pi^{-1}}^2 + \frac{C_s}{c_2} + \frac{2c_1}{c_2}}+2c_2\|y_\star - y_0 \|_{\sigma^{-1}\pi^{-1}}$, \\
$C_o = C_s + \| y_\star \|_{\sigma^{-1}\pi^{-1}} C_f + c_1 \| x_0 - x_\star \|^2_{\tau^{-1}P^{-1}} + c_2\|y_\star - y_0 \|^2_{\sigma^{-1}\pi^{-1}}$,\\
$C_l = C_s + c_1 \| x_\star-x_0\|^2_{\tau^{-1}P^{-1}} + 4 c_2 L_h^2$,\\
$c_1 = \frac{3\underline p}{2} + \underline p \| (2P -\underline p)^{1/2}\|\|\pi^{-1} - I \|$,\\
$c_2 = \frac{3\underline p}{2} + (1-\underline p)\|(2P-\underline p)^{1/2} \|$,
$C_s =  C_{g, 2}+ C_{g, 5}+C_{g,6}$, with $C_{g, 2}$ as defined in Theorem~\ref{th:erg} and $C_{g,5}, C_{g, 6}$ are defined in the proof in~\eqref{eq: cg3_def},~\eqref{eq: cg4_def}.
\end{reptheorem}
\begin{proof}
First, we will use~\Cref{lem: random_proc_lemma} on the result of~\Cref{lem: lem_erg3}, similar to~\eqref{eq: breg_tow_last}. The difference is that we process the terms $(1-\underline p)\left(D_p(x_k; z) - D_p(x_{k+1}; z)\right) + \underline p \left( D_d^{\pi^{-1}-I}(\breve y_k; z) - D_d^{\pi^{-1}-I}(\breve y_{k+1}; z) \right)$ with small differences. In particular,
\begin{align}
\sum_{k=0}^{K-1} (1-\underline p)\left(D_p(x_k; z) - D_p(x_{k+1}; z)\right) &=\sum_{k=0}^{K-1} (1-\underline p)\left(f(x_k) + g(x_k) - f(x_{k+1}) - g(x_{k+1}) + \langle A^\top y, x_k - x_{k+1} \rangle \right)\notag \\
&= (1-\underline p)\left( f(x_0) + g(x_0) - f(x_{K}) - g(x_{K}) + \langle A^\top y, x_0 - x_{K}\rangle \right).\notag
\end{align}
For the final term, we estimate using the step size rule~\eqref{eq: ss_choice}
\begin{align}
\langle A^\top y, x_0 - x_{K} \rangle &= \sum_{i=1}^n \sum_{j=1}^m A_{j, i} y^j (x_0^i - x_{K}^i) \leq \sum_{i=1}^n \sqrt{\sum_{j=1}^m A_{j, i}^2 (x_0^i - x_{K}^i)^2 \sigma_j \pi_j}\sqrt{\sum_{j=1}^m (y^j)^2 \sigma_j^{-1}\pi^{-1}_j} \notag \\
&\leq \sum_{i=1}^n \sqrt{\underline p(2p_i - \underline p) (x_0^i - x_{K}^i)^2 \tau^{-1}_i p_i^{-1}}\sqrt{\sum_{j=1}^m (y^j)^2\sigma_j^{-1}\pi_j^{-1}}\notag \\
&\leq \sum_{i=1}^n \left( \frac{1}{2} \sqrt{2p_i - \underline p} (x_0^i - x_{K}^i)^2 \tau^{-1}_i p_i^{-1} + \frac{\underline p\sqrt{2p_i - \underline p}}{2} \| y \|^2_{\sigma^{-1}\pi^{-1}}\right) \notag \\
&\leq \frac{\| (2P - \underline p)^{1/2} \|}{2} \| x_0 - x_{K} \|^2_{\tau^{-1}P^{-1}} + \frac{\| (2P - \underline p)^{1/2} \|}{2} \| y \|^2_{\sigma^{-1} \pi^{-1}} \notag \\
&\leq \frac{\| (2P - \underline p)^{1/2} \|}{2} \| x_0 - x_{K} \|^2_{\tau^{-1}P^{-1}} + \| (2P - \underline p)^{1/2} \|\left( \| y - y_0 \|^2_{\sigma^{-1} \pi^{-1}} + \| y_0 \|_{\sigma^{-1}\pi^{-1}}^2\right).\notag
\end{align}
We estimate similarly to obtain
\begin{align}
\sum_{k=0}^{K-1} \underline p\left(D_d^{\pi^{-1}-I}(\breve y_k; z) - D_d^{\pi^{-1}-I}(\breve y_{k+1}; z)\right) = \underline p\left( h^\ast_{\pi^{-1} - I}(\breve y_0) - h^\ast_{\pi^{-1}-I}(\breve y_{K}) - \langle Ax,(\pi^{-1}-I) (\breve y_0 - \breve y_{K}) \rangle \right),\notag
\end{align}
and
\begin{align}
-\langle Ax, (\pi^{-1}-I)&(\breve y_0 - \breve y_{K}) \rangle \leq \frac{1}{2} \| (2P - \underline p)^{1/2} \| \| \pi^{-1} - I \| \left( \| x \|^2_{\tau^{-1}P^{-1}} + \| \breve y_0 - \breve y_{K} \|^2_{\sigma^{-1}\pi^{-1}}\right) \notag \\
&\leq \| (2P - \underline p)^{1/2} \| \| \pi^{-1} - I \| \left(  \| x-x_0 \|^2_{\tau^{-1}P^{-1}} + \| x_0 \|^2_{\tau^{-1}P^{-1}}+ \frac{1}{2}\| \breve y_0 - \breve y_{K} \|^2_{\sigma^{-1}\pi^{-1}}\right)\notag
\end{align}
With these differences, instead of~\eqref{eq: breg_tow_last}, we get
\begin{align}
\mathbb{E}\bigg[\sup_{z\in \mathcal{Z}}\sum_{k=1}^K &\underline p\left(D_p(x_{k}) +  D_d(\breve y_{k}) \right) - c_1 \| x_0 - x \|^2_{\tau^{-1}P^{-1}} - c_2 \| y_0 - y \|^2_{\sigma^{-1}\pi^{-1}}\bigg] \leq \underline p\left( h^\ast_{\pi^{-1}-I}(\breve y_0) - h^\ast_{\pi^{-1}-I}(\breve y_{K})\right) \notag \\
&+ (1-\underline p)\left( f(x_0) + g(x_0) - f(x_{K}) - g(x_{K}) \right) + \sum_{k=1}^K \frac{\underline p}{2} \E{\| x_k - x_{k+1} \|^2_{\tau^{-1}P^{-1}} + \| y_k -y_{k+1} \|^2_{\sigma^{-1}\pi^{-1}}}\notag \\
& +\sum_{k=1}^K \frac{1}{2\underline p} \E{\|\sigma \pi A(x_k - x_{k+1}) \|^2_{\sigma^{-1}\pi^{-1}}} + \sum_{k=1}^K \frac{\underline p}{2}\E{\|\tau P A^\top\pi^{-1}(\breve y_k - \breve y_{k+1})\|^2_{\tau^{-1}P^{-1}}},\notag\\
&+\| (2P - \underline p)^{1/2} \|\left( \frac{1}{2} \| x_0 - x_{K} \|^2_{\tau^{-1}P^{-1}} + \| y_0 \|^2_{\sigma^{-1}\pi^{-1}} \right)\notag \\
&+\| (2P - \underline p)^{1/2} \| \| \pi^{-1} - I \| \left(\| x_0 \|^2_{\tau^{-1}P^{-1}}+ \frac{1}{2}\| \breve y_0 - \breve y_{K} \|^2_{\sigma^{-1}\pi^{-1}}\right),
\end{align}
where $c_1 = \frac{3\underline p}{2} + \underline p \| (2P -\underline p)^{1/2}\|\| \pi^{-1} - I \|$, $c_2 = \frac{3\underline p}{2} + (1-\underline p)\|(2P-\underline p)^{1/2} \|$.

We divide both sides by $\underline p$ and use Jensen's inequality to obtain the smoothed gap function~\cite{tran2018smooth}
\begin{align}
\mathcal{G}_{\frac{2c_1}{\underline pK}, \frac{2c_2}{\underline pK}}(x_K^{av}, y_K^{av}; x_0, y_0) &= \sup_{z=(x, y)\in \mathcal{Z}} D_p(x_K^{av}; z) + D_d(y_K^{av}; z) - \frac{c_1}{\underline pK} \| x-x_0 \|^2 - \frac{c_2}{ \underline pK} \|y-y_0\|^2.\notag
\end{align}
Then, we have, as in the proof of~\Cref{th:erg} that (see~\cref{eq: cg2,eq: breg_tow_last})
\begin{align}
\underline p K\E{\mathcal{G}_{\frac{2c_1}{\underline pK}, \frac{2c_2}{\underline pK}}(x_K^{av}, y_K^{av}; x_0, y_0)} &\leq C_{g, 2} + \underline p\left( h^\ast_{\pi^{-1}-I}(\breve y_0) - h^\ast_{\pi^{-1}-I}(\breve y_{K})\right) \notag \\
&+ (1-\underline p)\left( f(x_0) + g(x_0) - f(x_{K}) - g(x_{K}) \right) \notag \\
&+(1-\underline p)\| (2P - \underline p)^{1/2}\| \left( \frac{1}{2} \| x_0 - x_{K} \|^2_{\tau^{-1}P^{-1}} + \| y_0 \|^2_{\sigma^{-1}\pi^{-1}} \right)\notag \\
&+\underline p \| (2P - \underline p)^{1/2} \| \| \pi^{-1} - I \| \left(\| x_0 \|^2_{\tau^{-1}P^{-1}}+ \frac{1}{2}\| \breve y_0 - \breve y_{K} \|^2_{\sigma^{-1}\pi^{-1}}\right)\label{eq: sm_gap_last_bd}
\end{align}
Then, we use the optimality conditions and convexity,
\begin{align}
\mathbb{E} \left[ f(x_{K}) + g(x_{K})\right] &\geq \mathbb{E}\left[ f(x_\star) + g(x_\star) - \langle A^\top y_\star, x_{K} - x_\star \rangle\right]\notag\\ &\geq \mathbb{E}\left[f(x_\star) + g(x_\star) - \| A^\top y_\star \|^2_{\tau P}\|x_{K} - x_\star \|_{\tau^{-1}P^{-1}}\right] \notag \\
&\geq  f(x_\star) + g(x_\star) - \| A^\top y_\star \|_{\tau P}\sqrt{\frac{\Delta_0}{\underline p}}.\notag
\end{align}
Similarly,
\begin{align}
\mathbb{E} \left[ h^\ast_{\pi^{-1}-I}(\breve{y}_{K})\right] &=\mathbb{E} \left[ \sum_{j=1}^m (\pi_j^{-1} - 1) h^\ast_j (\breve{y}_{K}^j) \geq \sum_{j=1}^m (\pi_j^{-1} - 1) \left(h^\ast_j(y_\star^j) + \langle (Ax_\star)_j, \breve{y}_{K}^j - y_\star^j\rangle \right)\right] \notag \\
&\geq\mathbb{E} \left[\sum_{j=1}^m (\pi_j^{-1} - 1) h^\ast_j(y_\star^j) - \frac12\|Ax_\star \|^2_{\sigma \pi^{-1}} - \frac12 \| \breve y_{K} - y_\star \|^2_{\sigma^{-1}\pi^{-1}} \right] \notag \\
&\geq\mathbb{E} \left[\sum_{j=1}^m (\pi_j^{-1} - 1) h^\ast_j(y_\star^j) - \frac12\|Ax_\star \|^2_{\sigma \pi^{-1}} - \|y_{K} - y_\star \|^2_{\sigma^{-1}\pi^{-1}} - \sum_{k=1}^K \| \bar x_{k+1} - x_k \|^2_{B(\pi^{-1}\sigma^{-1})} \right] \notag \\
&\geq \sum_{j=1}^m (\pi_j^{-1} - 1) h^\ast_j(y_\star^j) - \frac12\|Ax_\star \|^2_{\sigma \pi^{-1}} - \frac{\Delta_0}{\underline p} - \frac{\| 2P - \underline p \|\Delta_0}{\underline p^2 C_{\tau, \tilde V}} \notag
\end{align}
So we denote
\begin{align}
C_{g, 5} &= (1-\underline p)\left( f(x_0) + g(x_0) - f(x_\star) - g(x_\star) + \| A^\top y_\star \|_{\tau P}\sqrt{\frac{\Delta_0}{\underline p}}  \right)  \notag \\
&+\underline p \left(h^\ast_{\pi^{-1}-I}(y_0) - \sum_{j=1}^m (\pi_j^{-1} - 1) h^\ast_j(y_\star^j) + \frac12\|Ax_\star \|^2_{\sigma \pi^{-1}} + \frac{\Delta_0}{\underline p} + \frac{\Delta_0\| 2P-\underline p\|}{\underline p^2 C_{\tau, \tilde V}}\right).\label{eq: cg3_def}
\end{align}
Next, we bound the last two terms in~\eqref{eq: sm_gap_last_bd} using~\eqref{eq: main_rec_ineq} and we denote the bound as $C_{g, 6}$
\begin{multline}
C_{g, 6} = (1-\underline p)\| (2P - \underline p)^{1/2}\| \left( \frac{4\Delta_0}{\underline p} +2\| y_\star \|^2_{\sigma^{-1}\pi^{-1}}\right)\\
+ \underline p\| (2P - \underline p)^{1/2} \| \| \pi^{-1} - I \| \left( \frac{4\Delta_0}{\underline p} + 2\| x_\star \|^2_{\tau^{-1}P^{-1}} \right),\label{eq: cg4_def}
\end{multline}
so that RHS of~\eqref{eq: sm_gap_last_bd} is $C_{g, 2} + C_{g,5} + C_{g, 6}$.

We consider two cases:

$\bullet$ If $h$ is $L_h$ Lipschitz continuous in norm $\| \cdot \|_{\sigma\pi}$, then $\| y-y_0\|^2_{\sigma^{-1}\pi^{-1}} \leq 4L_h^2$.
Then, we argue as in~\citep[Theorem 11]{fercoq2019coordinate} to get
\begin{multline}
\mathbb{E} \left[ f(x_K^{av}) + g(x_K^{av}) + h(Ax_K^{av}) - f(x_\star) - g(x_\star) - h(Ax_\star) \right] \leq \mathbb{E}\left[\mathcal{G}_{\frac{2c_1}{\underline pK}, \frac{2c_2}{\underline pK}}(x_K^{av}, y_K^{av}; x_0, y_0)\right] \\
+ \frac{c_1}{\underline pK} \| x_\star-x_0\|^2_{\tau^{-1}P^{-1}} + \frac{4 c_2}{\underline pK}L_h^2. \notag
\end{multline}
$\bullet$ If $h(\cdot) = \delta_{b}(\cdot)$, we use~\citep[Lemma 1]{tran2018smooth} to obtain
\begin{align}
&\mathbb{E}\left[ f(x_K^{av}) + g(x_K^{av}) - f(x_\star) - g(x_\star) \right] \leq \mathbb{E}\left[\mathcal{G}_{\frac{2c_1}{\underline pK}, \frac{2c_2}{\underline pK}}(x_K^{av}, y_K^{av}; x_0, y_0)\right] + \frac{c_1}{\underline pK}\| x_0 - x_\star \|^2_{\tau^{-1}P^{-1}} \notag \\
&\qquad\qquad\qquad\qquad\qquad\qquad\qquad\qquad\qquad\qquad\qquad\qquad+ \frac{c_2}{\underline pK}\|y_\star - y_0 \|^2_{\sigma^{-1}\pi^{-1}}+\mathbb{E}\left[ \| y_\star \|_{\sigma^{-1}\pi^{-1}} \| Ax_K^{av}-b\|_{\sigma\pi} \right]\notag \\
&\mathbb{E}\left[\| Ax-b \|_{\sigma\pi}\right] \leq \frac{2c_2}{\underline pK} \| y_\star - y_0 \|_{\sigma^{-1}\pi^{-1}} \notag \\
&~~~~~~~~~~~~~~~~~~~~~~~~~~~~~+ \frac{2c_2}{\underline pK}\sqrt{ \| y_\star - y_0 \|^2_{\sigma^{-1}\pi^{-1}}+\frac{\underline pK}{c_2}\left(\mathbb{E}\left[\mathcal{G}_{\frac{2c_1}{\underline pK}, \frac{2c_2}{\underline pK}}(x_K^{av}, y_K^{av}; x_0, y_0) \right]+ \frac{2c_1}{\underline pK} \| x_0 - x_\star \|^2_{\tau^{-1}P^{-1}}\right)}.\notag
\end{align}
We plug in the bound of $\mathbb{E} \left[ \mathcal{G}_{\frac{2c_1}{\underline p K}, \frac{2c_2}{\underline pK}}(x_K^{av}, y_K^{av}; x_0, y_0)\right]$ to obtain the final results.
\end{proof}
\textbf{Simplification of the constants.} As mentioned before~\Cref{th:erg}, we now give the inequalities we use to obtain the bounds we have in the main text for~\Cref{th:erg} and~\Cref{cor: erg} compared to the ones we have in the appendix.
It is easy to see by using coarse inequalities, we first $p_i\underline p^{-1}\leq\underline p^{-1}$, second, $2p_i - \underline p \leq 2$, third, $\pi_j^{-1}-1\leq \underline p^{-1}$ as $\pi_j \geq \underline p$.
Finally, by the definition of $\tau_i$ in~\eqref{eq: ss_choice}, we can derive $\| \tau^{1/2}P^{1/2}A^\top\pi^{-1/2}\sigma^{1/2}\|^2 \leq 2\underline p^{-1}$.
By using these constants in the bounds of~\Cref{th:erg} and~\Cref{cor: erg} in the appendix, we arrive at the bounds given for these theorems in the main text.
\end{document}